\documentclass[1p]{elsarticle}
\usepackage{nicematrix}

\usepackage[T1]{fontenc}
\usepackage{amsmath,amssymb,latexsym,mathrsfs}
\usepackage{mathabx} 
\usepackage{xcolor}
\usepackage[all,color,cmtip]{xy}
\usepackage{rotating}
\usepackage{amsthm}
\usepackage{mathtools}
 \usepackage{mathdots}
\usepackage{enumitem}
\usepackage{txfonts}
\usepackage{titlesec}
\usepackage{tikz}
\biboptions{sort&compress}
\usepackage{blkarray}
\usepackage{hyperref} 

\usepackage{csquotes}

\usepackage{xcolor}

\usepackage{bbm}

\makeatletter
\def\ps@pprintTitle{%
  \let\@oddhead\@empty
  \let\@evenhead\@empty
  \def\@oddfoot{\reset@font\hfil\thepage\hfil}
  \let\@evenfoot\@oddfoot
}
\makeatother

\titleformat{\section}
  {\normalfont\fontsize{12}{15}\bfseries}{\thesection}{1em}{}

\titleformat{\subsection}
  {\normalfont\fontsize{11}{15}\bfseries}{\thesubsection.}{0.5em}{}

\def\am#1{{\color{red}#1}} 

\def\stringmod#1{M_{#1}}
\def\stringcpx#1{X^\bullet_{#1}}
\def\stringCpx#1{\widehat{X}^\bullet_{#1}}
\def\CO{\mathcal{O}}

\def\Db#1{{\mathcal D}^b(#1)} 

\def\La{A}

\def\Car{C}
\def\idem{\varepsilon}
\def\fmrgs#1{\tilde{#1}}  

\newcommand{\perm}{\mathcal{P}}
\newcommand{\forb}{\mathcal{F}}
\newcommand{\Sperm}{\mathcal{SP}}
\newcommand{\Sforb}{\mathcal{SF}}
\newcommand{\gldim}{{\rm gl.dim}}

\newcommand{\Hom}{\mathrm{Hom}}

\newcommand{\rad}{\mathrm{rad}}

\newcommand{\dimk}{\mathrm{dim}_k}
\def\vdim#1{\underline{\mathrm{dim}}(#1)}
\def\vdimp#1{\underline{\mathrm{dim}}^{\rm p}(#1)}

\newcommand{\inc}{\underline{{\mathrm{inc}}}}

\newcommand{\Img}{\mathbf{Im} }

\newcommand{\Rnk}{\mathrm{rk}}
\newcommand{\Ind}{{\rm ind}^0} 
\newcommand{\ind}{{\rm ind}}  
\newcommand{\IndSt}{\Ind_{\rm st}(\per\,A)}
\newcommand{\IndBd}{\Ind_{\rm bd}(\per\,A)}

\newcommand{\tr}{{\rm{tr}}}
\newcommand{\Dyn}{\mathrm{Dyn}}
\newcommand{\CDyn}{\mathrm{CarDyn}}
\newcommand{\R}{\mathbb{R}}
\newcommand{\Z}{\mathbb{Z}}

\newcommand{\M}{\mathbb{M}}
\newcommand{\MM}{\mathbb{M}}

\newcommand{\NN}{\mathbb{N}}
\newcommand{\Inc}{\mathrm{Inc}}
\def\Incsh{\mathcal{I}} 
\def\AA{\mathbb{A}}
\newcommand{\BB}{\mathbb{B}}
\newcommand{\CC}{\mathbb{C}}
\newcommand{\DD}{\mathbb{D}}
\newcommand{\EE}{\mathbb{E}}
\newcommand{\FF}{\mathbb{F}}
\newcommand{\GG}{\mathbb{G}}
\newcommand{\CRnk}{\mathrm{crk}}

\newcommand{\Id}{\mathbf{I}}

\newcommand{\roots}{\mathcal{R}}

\newcommand{\bas}{\mathbf{e}}
\newcommand{\sgn}{\mathrm{sgn}}
\newcommand{\sou}{\mathbf{s}}
\newcommand{\tar}{\mathbf{t}}
\newcommand{\souh}{\mathbf{s}^h}
\newcommand{\tarh}{\mathbf{t}^h}

\newcommand{\chr}{\mathrm{char}}

\newcommand{\bulito}{\mathmiddlescript{\bullet}}
\newcommand{\diag}{\mathrm{diag}}
\newcommand{\wt}[1]{\widetilde{#1}}

\newcommand{\Null}{\CRnk} 

\newcommand{\s}{\mathrm{z}}
\newcommand{\triv}{\mathbf{1}} 
\newcommand{\quiver}{Q} 

\newcommand{\GBQ}{\triangledown}
\newcommand{\GPC}{\triangle}
\newcommand{\GFC}{\blacktriangleup}

\newcommand{\mult}{\mathrm{M}} 

\def\verts#1{V(#1)} 
\def\edgs#1{E(#1)}  
\def\hedgs#1{H(#1)}  
\def\Oedgs#1{\vec{E}(#1)}  
\def\invw#1{{#1}^{-1}}  
\def\pmw#1{{#1}^{\pm}}
\def\unit{\mathbf{1}}

\def\walksred{\mathfrak{w}} 

\def\leftres{\mathfrak{lw}}
\def\rightres{\mathfrak{wr}}
\def\twores{\mathfrak{lwr}}
\def\wlk{\omega} 

\def\CRinc{\mathcal{R}^{\rm inc}}

\def\ZZ{\Z}
\def\cyc{\wlk}
\def\cycles{{\rm cyc}}
\def\aoi{\xi} 
\def\anp{\alpha} 
\def\permut{\mathrm{p}} 
\def\m{\mathrm{m}} 
\def\jeden{J}

\def\congder{\cong_{\rm der}}
\def\gcc{{\nabla}}
\def\cyccond{\nabla}
\def\balanced{c}

\def\weak{\sim}

\def\dim{{\rm dim}}

\def\Brgr{\Gamma}

\def\Bralg{B}

\def\Kb{\mathcal{K}^b}
\def\Kmb{\mathcal{K}^{-,b}}
\def\proj{{\rm proj}}
\def\per{{\rm per}}
\def\mod#1{\mathrm{mod}\,#1}
\def\perdims{\mathrm{P}} 

\def\projbas{P}
\def\simplebas{S}

\def\ppath#1#2{\langle #1,#2 \rangle}
\def\vtx#1{\overline{#1}}

\def\lhalf#1{\hat{#1}}
\def\rhalf#1{\check{#1}}
\def\TO#1{\overrightarrow{#1}}
\def\OT#1{\overleftarrow{#1}}
\def\Wlk{\wt{\wlk}}
\def\jj#1{{\color{red}#1}} 
\def\succ{\jj{>}}
\def\pred{\jj{<}}
\def\nrr#1{{\rm nr}(#1)}
\def\nrp#1{{\rm nr}_+(#1)}
\def\nrm#1{{\rm nr}_-(#1)}

\newcommand\mathmiddlescript[1]{\vcenter{\hbox{$\scriptstyle #1$}}}

\newcommand{\vBlock}[2]{\mathrel{\stackrel{\makebox[0pt]{\mbox{\normalfont\tiny \am{#2}}}}{#1}}}

\numberwithin{equation}{section}

\newtheorem{lemma}[equation]{Lemma}{\bf}{\it}
\newtheorem{proposition}[equation]{Proposition}{\bf}{\it}
\newtheorem{corollary}[equation]{Corollary}{\bf}{\it}
\newtheorem{theorem}[equation]{Theorem}{\bf}{\it}
{\bf}{\it}
\newtheorem{claim}[equation]{Claim}{\bf}{\it}
{\bf}{\it}
\newtheorem*{theoremA}{Theorem A}{\bf}{\it}
{\bf}{\it}
\newtheorem*{theoremB}{Theorem B}{\bf}{\it}
\newtheorem*{theoremC}{Theorem C}{\bf}{\it}
\theoremstyle{definition}
\newtheorem{remark}[equation]{Remark}{\bf}{\rm}
{\bf}{\rm}
\newtheorem{definition}[equation]{Definition}{\bf}{\rm}
\newtheorem{example}[equation]{Example}{\bf}{\rm}

\theoremstyle{remark}

\begin{document}
\title{A graph theoretic model for the derived categories of gentle algebras\\ and their homological bilinear forms}

\author{Jes\'us Arturo Jim\'enez Gonz\'alez\fnref{UG}}

\ead{jesus.jimenez@academicos.udg.mx}

\author{Andrzej Mr\'oz\fnref{UMK,cora}}
\ead{amroz@mat.umk.pl}

\fntext[UG]{Departamento de Ciencia de Datos, Centro Universitario de Chapala, Universidad de Guadalajara, Chapala Jalisco, Mexico}
\fntext[UMK]{Faculty of Mathematics and Computer Science, Nicolaus Copernicus University in Toru\'n, ul. Chopina 12/18, 87-100 Toru\'n, Poland}

\fntext[cora]{Corresponding author}
\journal{\ }

\begin{abstract}
We customize the existing models for the bounded derived category of gentle algebras to obtain simple graph theoretic tools to analyze indecomposable objects, Auslander-Reiten triangles, and their interaction with the associated homological bilinear form and the Coxeter transformation.
We apply these tools to explore related new and classical derived invariants. We exhibit the non-negativity and Dynkin type of the homological quadratic form of a gentle algebra, classify indecomposable perfect complexes by means of its roots, describe  the Coxeter polynomial and relate it with  the Avella-Alaminos Geiss invariant. We also derive some consequences for Brauer graph algebras.
\end{abstract}

\begin{keyword}
Gentle algebra \sep marked ribbon graph \sep bounded derived category \sep homological bilinear form \sep Dynkin type \sep root systems \sep Coxeter polynomial
\MSC[2020] primary:
16G20 
\sep
16E35 
\sep
05E10 
\sep
16G70 

secondary:
16E20 
\sep
17B22 
\sep
05C76 
\end{keyword}

\maketitle

\section{Introduction}

The computation and characterization of derived categories is an important task in many areas of mathematics, because of their heavy content of homological information. Some of the homological invariants that have attracted attention of representation theorists over the last decades involve Grothendieck groups, Hochschild (co-)homology and Auslander-Reiten quivers. Also a prominent example is the  homological bilinear form, associated quadratic (Euler) form and the Coxeter transformation, classically defined for the algebras of finite global dimension, whose study helped to develop the representation theory of quivers and many other classes of algebras, especially in connection with representation types, root systems and tilting theory in the works of Gabriel~\cite{GabrielI}, Bongartz~\cite{BongartzTits}, Dlab-Ringel~\cite{DlabRingel}, Kac \cite{Kac80}, Reiten-Skowro\'nski~\cite{ReitenSkowron}, Br\"ustle-de la Pe\~na-Skowro\'nski~\cite{BPS}, Happel~\cite{Happel.Trace}, Lenzing-de la Pe\~na~\cite{LPtrends}, Mr\'oz-de la Pe\~na~\cite{MP2}, Zhang~\cite{Zhang}, to mention a few examples. Descriptions of such homological invariants are difficult in full generality, thus efforts are usually directed to classes of algebras of particular interest, gentle algebras in our case.

Before their formal introduction, the underlying combinatorics of gentle algebras appeared in several instances within representation theory: in the study of representations of the Lorentz group~\cite{GePo}, self-reproducing systems~\cite{NazarovaRoiterTwo}, representations of groups~\cite{BondaDrozd,BondarenkoI, Ri1975} and skew-group algebras~\cite{NazarovaRoiterOne}. These ideas have been systematized and generalized in the framework of matrix problems as clans~\cite{CBclan}, bushes~\cite{Deng}, bundles of chains~\cite{BondarenkoII} or bunches of chains~\cite{BurbanDrozd}, structures that have in common an initial combinatorial setting: a \emph{disjoint union of (semi-) chains with an involution}. After the development of tilting theory in the 1980s, such structures were recognized independently as a particular class of well-behaved special biserial algebras, starting with the tilted classification of hereditary algebras of Dynkin type $\mathbb{A}$~\cite{AH} and the formal introduction of \emph{gentle algebras} by Assem-Skowro\'nski for  type $\widetilde{\mathbb{A}}$~\cite{AS}, and soon after in the derived discrete~\cite{BGSdisc, VossieckDerDisc} and piecewise hereditary~\cite{HZ08} classifications, see also~\cite{Hph, RvdB}. An early explicit connection between these two approaches can be found in Bekkert-Merklen's description of indecomposable complexes of the derived category of gentle algebras~\cite{BekMer} following techniques presented in~\cite{BondarenkoI}. Many homological aspects of gentle algebras came to light with the help of repetitive algebras and Happel's embedding functor, cf.~\cite{Ringel1997TheRA}, from their closedness up to derived equivalence~\cite{SchZ}, the description of their derived categories and almost split triangles in terms of \emph{string combinatorics}~\cite{Bob,ALP}, to the construction of purely combinatorial derived invariants~\cite{AAG} that describe part of the tubular structure of the derived category (the latter also with help of clans). Recently, graded versions of gentle algebras were related to Fukaya categories of Riemann surfaces for considerations in homological mirror symmetry~\cite{HKK,LPFukaya} and to cluster combinatorics as Jacobian algebras for certain quivers with potentials~\cite{Labardini}. These connections prompted extensive research in the area and, together with developments on symmetric special biserial algebras and Brauer graph algebras in terms of \emph{ribbon graphs}~\cite{MarshSchroll,Schroll.trivext}, motivated the construction of geometric models for the derived category of (graded) gentle algebras~\cite{HKK,OPS18,PPP19} successfully applied for the characterization of derived equivalence within these algebras~\cite{LPFukaya,APS23}. The presentation of the geometric model in~\cite{OPS18} is based on the notion of {\it marked ribbon graphs} and their embeddings into marked oriented surfaces, where the reformulation of a gentle algebra as disjoint union of chains with an involution (having no fixed point) was exploited again. In this geometric framework, the surface, its curves and their intersections serve as model for the derived category of a gentle algebra.

Here we propose a reformulation of the geometric model to study the derived category of a gentle algebra purely in terms of marked ribbon graphs, which may be considered as a discrete version of existing surface presentations.
Our exposition makes no direct use of newly developed geometric methods and considerably simplifies parts of the classical string combinatorics. Our \lq\lq graph model'' 
is based on three technical results dealing with  the presentation of string complexes as reduced walks of the marked ribbon graph and the corresponding presentation of bands (Propositions~\ref{pro:walkBelt} and \ref{prop:nonper}, compare to~\cite{BekMer}), and the construction of Auslander-Reiten triangles, which simplifies the well-known construction of Bobi\'nski \cite{Bob} (Proposition~\ref{thm:ARtriangle}); see also Remark \ref{rem:mor}. The presentation is intended to provide simple tools for the construction and analysis of Auslander-Reiten components  and the homological bilinear form of gentle algebras. Although the homological form is easy to compute for concrete examples (its Gram matrix in an appropriate basis is the Cartan matrix of the algebra), the description of its properties in a unified manner is challenging work. With the proposed graph model and the use of the \emph{incidence matrix} of a graph as connection to linear aspects of quadratic forms as in~\cite{JM}, we explore the definiteness of the homological bilinear form, related new and classical derived invariants, and its interrelation with the classification of indecomposable objects of the derived category. The results of \cite{JM} form a combinatorial framework for our present study, and they may be viewed as yet another path leading to combinatorics of gentle algebras (this time from  algebraic graph theory and the theory of integral quadratic forms), see also \cite{JJ.JA2018}.

\smallskip

In the paper we work with finite dimensional (bound quiver) algebras over a field $k$, see Section \ref{sec:combgentle} for the details. Direct inspiration for our main results below are the following well-known facts concerning a specific class of gentle algebras, proved when $k$ is algebraically closed. Namely, by Assem-Skowro\'nski \cite{AS},  derived hereditary algebras of type $\widetilde{\AA}$ (thus having non-negative Euler form of corank 1 and Dynkin type $\AA$) are precisely the 1-cycle gentle algebras $A=kQ/I$ such that the bound quiver $(Q,I)$ satisfies the {\em clock condition} (cf.~Remark \ref{rem:tcycle}). Whereas, if a 1-cycle gentle algebra $A$ does not satisfy the {clock condition}, then by the results of Vossieck \cite{VossieckDerDisc} $A$ is derived discrete and not derived hereditary (cf.~\cite{SchZ,AH}). However, in this case and if the global dimension of $A$ is finite, $A$ has non-negative Euler form of corank at most one and Dynkin type  $\AA$ or $\DD$ by Bobi\'nski-Geiss-Skowro\'nski~\cite{BGSdisc} (see Subsections \ref{subsec:qf} and \ref{subsec:eulf} for  detailed definitions).

The following first main result of the paper
establishes a relationship between the marked ribbon graph and the Euler form of a gentle algebra. It generalizes some aspects of the above results \cite{AS,BGSdisc} (cf.~also \cite{AH}) to all gentle algebras by providing the non-negativity, Dynkin type and corank of the Euler form and a related simple derived invariant.
From now on $k$ denotes an arbitrary fixed field (not necessarily algebraically closed). We deal with gentle bound quivers of arbitrary global dimension.

\begin{theoremA}[{Theorem~\ref{thm:mrgvsCar}}]\label{mainthm:A}
Let $(Q,I)$ be a gentle bound quiver,  and $(G,\leq)$  the associated marked ribbon graph. Then the Cartan matrix $\Car_\La$ of  the gentle algebra $\La:=kQ/I$ and the incidence matrix $\Inc(G)$ of $G$ are related by the identity
$$
\Car_\La+\Car_\La^\tr=\Inc(G)\Inc(G)^\tr.
$$
 In particular, the Euler form $q_\La:K_0({\rm per}\,\La)\to\ZZ$ of $\La$ is non-negative. Moreover:
\begin{enumerate}[label={\textnormal{(\alph*)}},topsep=3px,parsep=0px]
\item  $(Q,I)$ satisfies the \textbf{multi-clock condition} $($that is, every closed walk in $Q$ contains even number of relations of length two from $I$, see Definition~\ref{def:mcc}) if and only if the graph $G$ is bipartite.
\item The corank of the Euler form is given by the formula
$$\CRnk(q_\La)=|Q_1|-|Q_0|+\gcc_\La,$$ where $\gcc_\La\in\{0,1\}$ and $\gcc_\La=1$ if and only if $(Q,I)$ satisfies the multi-clock condition. In particular,  $\gcc_\La$  is a derived equivalence invariant.  
\item With the exception of certain degenerate cases, the Dynkin type $\Dyn(q_A)$ of $q_A$ is $\AA$, $\DD$ or $\CC$, and $\Dyn(q_A)=\AA$ if and only if $(Q,I)$ satisfies the multi-clock condition.
\end{enumerate}
\end{theoremA}
We refer to the full version (Theorem~\ref{thm:mrgvsCar}) of the theorem for more precise statements on the Dynkin type, as well as some additional claims, among others, for the distinguished case of finite global dimension and a certain graph \lq\lq Koszul dual''  to $(G,\leq)$. The proof is given in Subsection~\ref{subsec:proofA}.
We present in Subsection~\ref{subsec:applA} several applications of Theorem A. Among other things, we provide there a collection of graph theoretic invariants for gentle algebras in the spirit of the results of Antipov-Zvonareva~\cite{AZ22.Brauer} and Opper-Zvonareva~\cite{OppZvo22} for Brauer graph algebras (see Corollary~\ref{cor:OZspirit}); a characterization of the properties of the Euler form for derived discrete algebras (see Corollaries~\ref{cor:BGS} and~\ref{cor:derdiscsimple}) which recover and extend some of the results of Bobi\'nski-Geiss-Skowro\'nski~\cite{BGSdisc}; and we show that the Euler form of a Brauer graph algebra is non-negative, and we interrelate its definiteness with the representation type (see Theorem~\ref{thm:Brauer}, Corollary~\ref{cor:BrauerDomestic} and Remark \ref{rem:Antipov}).

The following second main result of the paper may be viewed as an analogue of classical results for algebras of finite global dimension providing a classification of indecomposable modules in terms of  positive roots of the Euler (or Tits) form, cf.~\cite{GabrielI, BongartzTits, DlabRingel, Kac80}.

\begin{theoremB}[{Theorem~\ref{thm:rootsind}}]\label{mainthm:B}
Let $(Q,I)$ be a gentle bound quiver and $A=kQ/I$.  The following inclusions hold:
$$
\roots_q(0) \cup \roots_q(1)\  \subseteq  \  \perdims(A) \  \subseteq \ \roots_q(0) \cup \roots_q(1) \cup \roots_q(2),
$$
where $\perdims=\perdims(A):=[\ind({\rm per}\,\La)]\subseteq K_0({\rm per}\,\La)$ denotes the subset of all the classes of indecomposable perfect complexes in the Grothendieck group $K_0({\rm per}\,\La)$ of the perfect category, $q=q_\La:K_0({\rm per}\,\La)\to\ZZ$ is the  Euler form of $A$  and $\roots_q(d)=q^{-1}(d)$ is the set of all {$d$-roots} of $q$ for $d\in\ZZ$. Moreover,
 $\roots_q(0) \cup \roots_q(1) = \perdims$  if and only if\  $(Q,I)$ satisfies the multi-clock condition.
\end{theoremB}
We refer to the full version (Theorem~\ref{thm:rootsind}) of the theorem for more precise statements, including a complete characterization of indecomposable perfect complexes whose classes in the Grothendieck group are 0-, 1-, or 2-roots, respectively, by means of walks in the associated marked ribbon graph, see also Corollary~\ref{cor:2roots}.  The proof of the theorem and some applications are  given in Subsections~\ref{subsec:proofB} and~\ref{subs:rootsys}, respectively. We show that, analogously as in  the classical situation of representations of quivers \cite{GabrielI, Kac80}, in case the Euler form $q_A$ is positive, the set $\perdims(A)$ is finite and it carries an additional structure of  a root system in the sense of Bourbaki~\cite{Bourbaki}, see Theorem~\ref{thm:rootsys}.

The following third main result provides a complete description of the Coxeter polynomial $\psi_A(z)\in \ZZ[z]$ of a gentle algebra $A$ in terms of Avella-Alaminos Geiss invariant of $A$ (shortly, AAG invariant; see \cite{AAG} and \ref{subsec:antiwalks}). Recall that $\psi_A(z)$ is the characteristic polynomial of the automorphism of the Grothendieck group $K_0(\per\,\La)$ induced by the Auslander-Reiten translation  on the perfect category (see \ref{subsec:Cox} for the details). The Coxeter polynomial is a widely studied (mainly for algebras of finite global dimension)  derived  invariant  which  encodes certain  interesting data on the module categories and derived categories,  see for instance~\cite{LPtrends, Happel.Trace,RCox,PMahler,Zhang, Sim.AM2022}.

\begin{theoremC}[{Theorem~\ref{cor:CoxPol}}]\label{mainthm:C}
Let $(Q, I)$ be a gentle bound quiver. The Coxeter polynomial $\psi_A(z)$ of the gentle algebra $A=kQ/I$ with AAG invariant $\varphi_A$ is given by
\begin{equation}\label{eq:coxpAAGC}
\psi_A(z)=(z-1)^{|Q_1|-|Q_0|}\prod_{\substack{n>0 \\ m \geq 0}}(z^n-(-1)^{n+m})^{\varphi_A(n,m)}.
\end{equation}
\end{theoremC} 

Observe that the theorem shows in particular that the Coxeter polynomial of a gentle algebra is cyclotomic (that is, it is a product of cyclotomic polynomials) and the formula~\eqref{eq:coxpAAGC} provides its cyclotomic decomposition as those studied in~\cite{Pcyclotalgebras, PMahler, MP2}. The proof of Theorem C is provided in Subsection \ref{subsec:Cox}. It exploits, among others, the interplay between the Cartan matrix and the marked ribbon graph from Theorem A and the compact formula for the Auslander-Reiten translation given in Proposition \ref{thm:ARtriangle}.

To sum up this introductory part, one of the basic ideas of the paper is to replace string combinatorics (or a geometry of curves in the geometric model) with the analysis of walks in marked ribbon graphs, equipped with methods related to integral quadratic forms. This approach provides simple tools to perform computations in the derived category of gentle algebras and it allows us to prove new general results for this class. It seems to us that these methods can be extended to related classes like skew-gentle algebras and (skew-)Brauer graph algebras. In particular, we believe that this may lead to a construction of further combinatorial derived invariants, cf.~\cite{AAG, APS23, OppZvo22}.

\smallskip
The authors thank Grzegorz Bobi\'nski  and Grzegorz Zwara for inspiring discussions, especially around Theorem \ref{thm:Brauer} on Brauer graph algebras (cf.~Remark \ref{rem:Antipov}).
\section{Preliminary notions}\label{sec:prelim}

 Our base commutative ring $R$ is either the integer numbers $R=\Z$ or an arbitrary field $R=k$, and $R$-modules (abelian groups or vector spaces) are supposed to have finite bases, that is, they are free and finitely generated. The set of $n \times m$ matrices with coefficients in $R$ is denoted by $\M_{n,m}(R)$; more generally, given two finite sets $X$ and $Y$, we consider the set of $X\times Y$-matrices $\M_{X,Y}(R):=\{f:X\times Y\to R: f\ \text{a function}\}$ and we define addition and multiplication of such matrices in a natural way. Any morphism of $R$-modules with fixed bases is identified with its corresponding matrix.  Given a natural number $n\geq 1$ (resp.~a finite set $X$),  by $\{\bas_i\}_{i=1}^n$ (resp.~$\{\bas_i\}_{i\in X}$) we denote the canonical basis of the $R$-module $R^n$ (resp.~$R^X$), that is, the coefficients of $\bas_i$ are given by $(\bas_i)_j=\delta_{i,j}$ for each $i,j$. We often identify $R^X$ with $R^{|X|}$ under an arbitrarily fixed ordering in $X$ and we view the elements of $R^{|X|}$ as column vectors. Similarly we identify $\M_{X,Y}(R)$ with $\M_{|X|,|Y|}(R)$. Given a matrix $M\in\M_{n,m}(\Z)$, the rank of $M$ is denoted by $\Rnk(M)$.  By $\CRnk(M):=m-\Rnk(M)$  we denote the (right) \textbf{corank} of $M$. Clearly, $\CRnk(M)$ is the rank of the (right) null space $\ker(M)$ of $M$.

\subsection{Integral quadratic forms}\label{subsec:qf}

Let $\langle -,- \rangle:K^2 \to \Z$ be an integral bilinear form where $K$ is a free finitely generated abelian group. If $\beta:=\{\beta_1,\ldots,\beta_n\}$ is a $\Z$-linear basis of $K$, then the \textbf{Gram matrix} of $\langle -,- \rangle$ with respect to $\beta$ is the matrix $M \in \M_{n,n}(\Z)$ given by
\begin{equation}\label{eq:Grambil}
M=M_{\langle -,- \rangle,\beta}=[\langle \beta_i,\beta_j \rangle]_{i,j=1}^n.
\end{equation}
The \textbf{rank}, \textbf{corank} and \textbf{determinant} of $\langle -,- \rangle$ are those of $M$, which do not depend on the choice of basis of $K$. Two integral bilinear forms $\langle -,- \rangle:K^2 \to \Z$ and $\langle -,- \rangle':(K')^2 \to \Z$ are called \textbf{equivalent} (or isometric), if there is an isomorphism of groups $f:K \to K'$ such that $\langle x,y \rangle=\langle f(x),f(y) \rangle'$ for all $x,y$ in $K$. Note that rank, corank and determinant are invariant under equivalence of integral bilinear forms.

\begin{definition}\label{def:qf}
Let $K$ be a free finitely generated abelian group.
\begin{enumerate}[label={\textnormal{(\alph*)}},topsep=4px,parsep=0px]
 \item An \textbf{integral quadratic form} on $K$ is a function $q:K \to \Z$ satisfying $q(x+x)=4q(x)$ for all $x$ in $K$, and such that the mapping $q(x,y):=q(x+y)-q(x)-q(y)$ is bilinear. The (symmetric) integral bilinear form $q(-,-):K^2 \to \Z$ is called \textbf{polarization} of $q$. Note that $q(x)=\frac{1}{2}q(x,x)$ by the first condition, and by the latter we have $q(tx)=t^2q(x)$ for all $t \in \Z$ and $x \in K$. The \textbf{rank} $\Rnk(q)$, \textbf{corank} $\CRnk(q)$ and \textbf{determinant} $\det(q)$ of $q$ are those of its polarization $q(-,-)$, and the \textbf{radical} of $q$ is the set $\rad(q):=\{x \in K \mid q(-,x)\equiv 0\}$. 

 \item  Given two integral quadratic forms $q:K \to \Z$ and $q':K' \to \Z$, by the \textbf{direct sum} of $q$ and $q'$ we mean the quadratic form $q\oplus q':K\oplus K'\to \ZZ$ defined by $(q\oplus q')(x,x'):=q(x)+q'(x')$ for $x\in K$ and $x'\in K'$.

\item  Two integral quadratic forms $q:K \to \Z$ and $q':K' \to \Z$ are \textbf{equivalent}  if so are their polarizations, that is, if there is an isomorphism of groups $f:K \to K'$ such that $q=q' \circ f$. In this case we write $q\sim q'$.

 \item An integral quadratic form $q$ is \textbf{non-negative} (respectively, \textbf{positive}) if $q(x) \geq 0$ for all $x \in K$ (respectively, if $q(x)>0$ for all $0 \neq x \in K$). The form $q$ is \textbf{irreducible} if whenever $q=aq'$ for an integral quadratic form $q'$ and an integer $a \in \Z$ then $a=\pm 1$.

 \item For $d\in \Z$, the set $\roots_q(d):=\{x \in K \mid q(x)=d\}$ is called set of \textbf{$d$-roots} of an integral quadratic form $q:K \to \Z$.
\end{enumerate}
\end{definition}

It can be easily verified that the radical of $q:K\to \Z$ is the (right or left) null space of the (symmetric) Gram matrix of $q(-,-)$ with respect to any basis of $K$. Therefore, the corank of $q$ is the rank of its radical. It can also be shown that if $q$ is non-negative then $\rad(q)=\roots_q(0)$, and $q$ is positive if and only if $\CRnk(q)=0$ (equivalently, $\det(q) \neq 0$), cf.~\cite[Proposition 2.8]{SimFI2011}.
Furthermore, given an integral bilinear form $\langle -,- \rangle:K^2 \to \Z$, the mapping $q:K\to\ZZ$ defined by $q(x):=\langle x,x \rangle$ is an integral quadratic form with $q(x,y)=\langle x,y \rangle+\langle y,x \rangle$ for all $x,y\in K$.

Note that, besides the numerical invariants rank, corank and determinant, the non-negativity (positivity) property of an integral quadratic form is preserved under equivalence (in particular, under change of basis). Below we list some notions of interest of an integral quadratic form $q:K \to \Z$ that depend on a given fixed basis $\beta:=\{\beta_1,\ldots,\beta_n\}$ of the abelian group $K$. We set  $q_{ij}=q_{ij}^\beta:=q(\beta_i,\beta_j)$ and $q_i=q_i^\beta:=q(\beta_i)=\frac{1}{2}q_{ii}$ for all $1\leq i,j\leq n$. Observe that then 
\begin{equation}\label{eq:q}
q\left(\sum_{i}x_i\beta_i\right)=\frac{1}{2}x^\tr Mx= \sum_{i}q_ix_i^2+\sum_{i<j}q_{ij}x_ix_j
\end{equation} for any $x=[x_1,\ldots,x_n]^\tr\in\ZZ^n$, where $M=M_{q(-,-),\beta}=[q_{ij}]_{i,j=1}^n$. 
\begin{itemize}[topsep=4px,parsep=0px]
\item The form $q$ is \textbf{connected} if for any $i \neq j$ in $\{1,\ldots,n\}$ there is a sequence $(i_0,\ldots,i_r)$ with $r>0$, $i_0=i$, $i_r=j$ and such that $q_{{i_{t-1}}{i_t}}\neq 0$ for $t=1,\ldots,r$.
\item The form $q$ is \textbf{unitary} or a \textbf{unit form} (respectively, \textbf{semi-unitary}) if $q_i=1$ (resp.~$q_i\in\{0,1\}$) for all $i$, cf.~\cite{BP}.

 \item The form $q$ is \textbf{Cox-regular} (resp.~\textbf{semi-Cox-regular}) if $q_i>0$ (resp.~$q_i\geq 0$) for each $1\leq i\leq n$, and $\frac{q_{ij}}{q_i}\in\ZZ$,  if $q_i>0$, and $\frac{q_{ij}}{q_j}\in\ZZ$, if $q_j>0$, for all $i\neq j$. Recall that $q$ is Cox-regular if and only if the Gram matrix $M_{q(-,-),\beta}=[q_{ij}]_{i,j=1}^n$ is a (symmetrized) quasi-Cartan matrix~\cite{BGZ}, cf.~\cite{KS, Roiter} and~\cite[Proposition 2.7]{MM.LAA2019}.
\end{itemize}

Observe that (semi-)unitary forms are (semi-)Cox-regular, and that the connectedness between non-negative unitary forms is preserved under equivalence, cf.~\cite[Lemma~2.1(e)]{Jesus.sub2021}. However, this is not the case among Cox-regular quadratic forms. Moreover, equivalences do not preserve neither unitary nor Cox-regular forms, cf.~\cite{JM, MZaj.LAA2022}.
For certain reasons (related mainly to Lie theory) one usually restricts equivalences to certain classes of isomorphisms. We recall the following classification results for certain classes of Cox-regular forms.
Let $q:K\to\ZZ$ be a non-negative connected Cox-regular form, considered in a fixed basis $\beta$ of $K$, and let $r:=\Rnk(q)\geq 1$ and $c:=\CRnk(q)\geq 0$ be the rank and corank of $q$, respectively.
\begin{itemize}[topsep=4px,parsep=0px]
\item[(D1)] If $q$ is a unit form then by~\cite{BP} there exists a unique (simply-laced) Dynkin graph $D_r \in \{\AA_{n_a},\DD_{n_d},\EE_6,\EE_7,\EE_8\}_{n_a\geq 1, n_d\geq 4}$ with $r$ vertices such that $q\weak q_{D_r}\oplus \zeta^c$, where $q_{D_r}:\ZZ^r\to \ZZ$ is the unit form associated to $D_r$ and $\zeta^c:\ZZ^c\to\ZZ$ denotes the zero form, see~\cite{BP} for details, cf.~\cite{Sim2016.1, SimZajLAA2017} and~\cite[Theorem 4.3]{JM}. The Dynkin graph  $D_r$ is called the \textbf{Dynkin type} of $q$ and it is denoted by $\Dyn(q):=D_r$.

\item[(D2)] If $q$ is irreducible, non-unitary, and the coefficients $q_{ij}=q_{ij}^\beta$ and $q_i=q_i^\beta$ of $q$ in the basis $\beta$ satisfy:   $1\leq q_i \leq 2$ and $\frac{q_{ij}}{q_iq_j} \in \ZZ$ for all $i,j$, then
    $q$ is said to have \textbf{Dynkin type} $\CC_r$, written $\Dyn(q)=\CC_r$ (note that $r\geq 2$ in this case). It is shown in~\cite[Theorem 5.19]{JM} that in this case the equivalence $q\sim q_{\DD_r}\oplus\zeta^c$ holds, provided $r\geq 4$. We stress that it should not be confused with  Dynkin type $\DD$ since simply-laced Dynkin types are defined only for unit forms. In particular, the class of Cox-regular forms of Dynkin type $\CC$ is not closed under equivalences. However, it is closed under so-called Gabrielov equivalence $\sim_G$  (of interest in Lie theory and singularity theory) and a complete $\sim_G$-classification of forms of type $\CC$  is provided in~\cite[Section 5]{JM}. We refer to~\cite{KS, MZaj.LAA2022,BKL.JA2006,PR.LAA2019,Sim.AM2022,MZaj2} for more details on Cox-regular forms, their certain other classifications and a broader context.
\end{itemize}

\subsection{Bidirected graphs}\label{subsec:bdg}

For convenience we use the following notion of graph by means of ``glued'' half-edges, cf.~\cite{OPS18}.

\begin{definition}\label{def:graph}
A \textbf{graph} $G=(V,H,\s,\iota)$ consists of sets $V$ and $H$ together with a function $\s:H \to V$ and an involution $\iota:H\to H$ without fixed points. The elements of $V$ and $H$ are called \textbf{vertices} and \textbf{half-edges} of $G$, respectively, and the vertex $\s(h)$ is said to be \textbf{incident} to the half-edge $h$. The set of \textbf{edges} $\edgs{G}$ of $G$ is the set of $\iota$-orbits on $H$, that is, $\edgs{G}:=H/\iota=\{\{h,\iota h\}\}_{h \in H}$. An edge $\{h,\iota h\}$ with $\s(h)=\s(\iota h)$ is called a \textbf{loop}. We also write $\verts{G}:=V$ and $\hedgs{G}:=H$.
\end{definition}

Throughout the paper we will assume that all graphs are finite (that is, so are the sets $V$ and $H$) and that they have no isolated vertices (that is, $\Img(\s)=V$).
Some of the usual concepts related to graphs can be expressed according to this definition as follows.
Two graphs $G=(V,H,\s,\iota)$ and $G'=(V',H',\s',\iota')$ are \textbf{isomorphic} if there are bijections $f_V:V \to V'$ and $f_H:H \to H'$ such that $\s' \circ f_H=f_V \circ \s$ and $\iota' \circ f_H =f_H \circ \iota$. A \textbf{subgraph} of a graph $G=(V,H,\s,\iota)$ consists of subsets $V' \subseteq V$ and $H' \subseteq H$ such that $\s$ and $\iota$ are well defined on their restrictions to $H'$ and $V'$, denoted by $\s':H' \to V'$ and $\iota':H' \to H'$. In that case $G':=(V',H',\s',\iota')$ is itself a graph. A non-empty subgraph $G'$ of $G$ is called \textbf{proper} if $V' \neq V$, \textbf{full} if any $h \in H$ with $\s(h)$ and $\s(\iota h)$ in $V'$ belongs to $H'$, and a \textbf{component} of $G$ if any $h \in H$ with $\s(h)$ or $\s(\iota h)$ in $V'$ belongs to $H'$. A graph is called \textbf{connected} if it contains no proper component. Observe that  $G$ is disconnected if and only if there is a (non-trivial) partition $H=H^1\cup H^2$ with $\iota H^\epsilon=H^\epsilon=\s^{-1}(\s H^\epsilon)$ for $\epsilon=1,2$. 
If $G$ is a connected graph with
\begin{equation}\label{eq:tcycle}
t=\cycles(G):=|\edgs{G}|-|\verts{G}|+1
\end{equation} then $t\geq 0$ and we say that $G$ is a $t$-\textbf{cycle} graph. Observe that $G$ is a $0$-cycle graph if and only if $G$ is a tree.

Consider the set of \textbf{ordered edges} $\Oedgs{G}$ of $G$, that is, the subset of $H \times H$ given by
\[
\Oedgs{G}:=\{(h',h'') \in H \times H \mid h''=\iota h'\}.
\]
For an ordered edge $i=(h',h'') \in \Oedgs{G}$ denote by $\invw{i} := (h'',h')\in \Oedgs{G}$ the \textbf{inverse} of $i$ and by $\vtx{i}=\{h',h''\} \in \edgs{G}$ the corresponding (unordered) edge of $G$. We use the notation $\tarh(i):=h' \in H$ and $\tar(i):=\s(h') \in V$ for the (half-edge and vertex) \textbf{target} of $i$, and $\souh(i):=h'' \in H$ and $\sou(i):=\s(h'') \in V$ for the corresponding \textbf{source} of $i$. Thus we may visualize the ordered edge $i=(h',h'')$ as follows:
\begin{equation}\label{eq:ordedg}
{\xymatrix@C=2.8pc@R=0.05pc{
{\bullet}\ar@{<-}[r]^-{\tarh(i)=h'} &\!\!\!&\ar@{|-}[l]_-{\souh(i)=h''} {\bullet}\\
\mathmiddlescript{\tar(i)=\s(h')}&\!\!\! &\mathmiddlescript{\sou(i)=\s(h'')} \\
 }}
\end{equation}

By \textbf{walk} of $G$ we mean an alternating sequence $$\wlk=(v_0,i_1,v_1,\ldots,v_{\ell-1},i_{\ell},v_{\ell})$$ of vertices $v_t$ and ordered edges $i_t$ of $G$ such that $\tar(i_t)=v_{t-1}$ and $\sou(i_t)=v_t$ for $t=1,\ldots,\ell\geq 0$ (observe that we write walks \lq\lq from right to left''). The integer $\ell(\wlk):=\ell \geq 0$ is called \textbf{length} of $\wlk$, and we use the compact notation $\wlk=i_1\cdots i_{\ell}$ or $\wlk={}^{v_0}i_1\cdots i_{\ell}^{v_{\ell}}$ if $\ell(\wlk)>0$. A walk $\wlk$ of length zero starting at vertex $v$ is called \textbf{trivial} and is denoted by $\triv_{v}:=(v)$. Extend the target (resp. source) functions to non-trivial walks by taking $\tarh(\wlk):=\tarh(i_1)$ and $\tar(\wlk):=\tar(i_1)$ (resp. $\souh(\wlk):=\souh(i_{\ell})$ and $\sou(\wlk):=\sou(i_{\ell})$), and take $\tar(\triv_v)=v=\sou(\triv_v)$. We say that a non-trivial walk $\wlk$ as above is:
\begin{itemize}
\item  \textbf{closed} if $\sou(\wlk)=\tar(\wlk)$,
\vspace{-2pt}
\item \textbf{open} if $\sou(\wlk)\neq\tar(\wlk)$,
\vspace{-2pt}
\item \textbf{reduced} if  $i_t \neq \invw{i}_{t+1}$ for $t=1,\ldots,\ell-1$,
\vspace{-2pt}
\item  \textbf{self-crossing} at $v_t$ if $v_t=v_s$ for some $0\leq t,s\leq \ell$, $s\neq t$,
\vspace{-2pt}
\item and $\wlk$ has \textbf{no repeated edges}, if $\vtx{i}_t\neq \vtx{i}_s$ for all $s\neq t$.
\end{itemize}
    The \textbf{inverse} of the walk $\wlk$ is given by $\invw{\wlk}:=\invw{i}_{\ell}\cdots \invw{i}_1$ if $\ell>1$ and $\invw{\triv}_v:=\triv_v$. A pair of walks $\wlk$ and $\wlk'$ is said to be \textbf{concatenable} if $\sou(\wlk)=\tar(\wlk')$, and their concatenation is denoted by $\wlk\wlk'$ which is a walk of length $\ell(\wlk)+\ell(\wlk')$. If $\wlk$ is a closed walk we denote by $\wlk^r$ the $r$-fold concatenation of $\wlk$ with itself, and say that $\wlk$ is \textbf{primitive} if it is not the $r$-fold concatenation of a walk of strictly smaller length. For an arbitrary walk $\wlk$ and $0 \leq t \leq \ell(\wlk)$ we will often use the decomposition
\[
 \wlk=\wlk_{[t]} \, {}_{[t]}\wlk \qquad \qquad \text{with $\ell(\wlk_{[t]})=t$ and $\ell(_{[t]}\wlk)=\ell(\wlk)-t$.}
 \]

\begin{definition} \label{def:bidi}
By \textbf{bidirected graph} $(G,\sigma)$ we mean a graph $G$ together with a function $\sigma:H\to \{\pm 1\}$ called a \textbf{bidirection} of $G$. If such function satisfies $\sigma(\iota h)=-\sigma(h)$ for all $h\in H$ we say that it is a \textbf{direction} of $G$, and the pair $(G,\sigma)$ is called a \textbf{directed graph}.
By the \textbf{(edge-vertex) incidence matrix} of $(G,\sigma)$ we mean the matrix $\Inc(G,\sigma)\in\M_{\edgs{G},\verts{G}}(\ZZ)$ whose transposed rows are defined by
\begin{equation}\label{eq:inc}
\Inc(G,\sigma)^\tr\bas_i:=\sigma(h)\bas_{\s(h)}+\sigma(h')\bas_{\s(h')}, \; \text{ for $i=\{h,h'\}\in \edgs{G}$}.
\end{equation}
Viewing a graph $G$ as a bidirected graph $(G,\sigma^{+})$ with constant $\sigma^+ \equiv +1$, we define its \textbf{incidence matrix} as $\Inc(G):=\Inc(G,\sigma^+)$.
\end{definition}
A bidirection $\sigma$ of $G$ induces a function $\sigma:\Oedgs{G} \to \{\pm 1\}$ by taking
 \begin{equation}\label{eq:signofedge}
\sigma(i):=(-1)\sigma(\tarh(i))\sigma(\souh(i)).
 \end{equation}
A positive ordered edge $i$ (that is, one with $\sigma(i)=+1$) is called a (direct) \textbf{arrow} if $\sigma(\souh(i))=1$, and an \textbf{inverse arrow} if $\sigma(\souh(i))=-1$. A negative ordered edge $i$ is called a \textbf{bidirected arrow}; in this case $\sigma(\souh(i))=\sigma(\tarh(i))=:\epsilon$, and if $\epsilon=+1$ (resp.~$\epsilon=-1$) then we call $i$ a \textbf{two-tail arrow} (resp.~\textbf{two-head arrow}). The sign of a walk $\wlk$ is given by $\sigma(\wlk):=\prod_{t=1}^{\ell(\wlk)} \sigma(i_t)$, setting $\sigma(\wlk):=+1$ if $\wlk$ is trivial. A bidirected graph $(G,\sigma)$ is {\bf balanced} if all closed walks in $G$ are positive, that is, each of them contains an even number of negative (bidirected) arrows. Otherwise, that is, when $G$ contains a negative closed walk, we say $(G,\sigma)$ is {\bf unbalanced}.   Sometimes, whenever it can be done without ambiguity, we also use letters $i$, $j$ etc. to denote unordered edges, and we apply analogous notation and terminology for unordered edges, that is, the sign of an unordered edge $i=\{h,h'\}\in E(G)$  is $\sigma(i):=-\sigma(h)\sigma(h')$ and we say that $i$ is \textbf{bidirected} if $\sigma(i)=-1$, analogously for \textbf{two-tail} (resp.~\textbf{two-head}) \textbf{edges}.

In representation theory, directed graphs are called \textbf{quivers} and are denoted by $\quiver=(\quiver_0,\quiver_1,\sou,\tar)$ where $\quiver_0$ is the set of vertices and $\quiver_1$ is the set of (direct) arrows of $\quiver$, and $\sou,\tar:\quiver_1 \to \quiver_0$ are the source and target functions as above (cf.~\cite{ASS}). In the sequel we often use (indexed) Greek letters to denote arrows in quivers and natural numbers (or Latin letters $i$, $j$) to denote their vertices. Note that the ordered edges of $\quiver$ are given by $\Oedgs{\quiver}=\quiver_1\cup \invw{\quiver}_1$. We say that a walk $\alpha_1 \cdots \alpha_{\ell}$ of $\quiver$ is a \textbf{path} if $\alpha_t \in \quiver_1$ for $t=1,\ldots,\ell \geq 0$.

\section{Combinatorial aspects of gentle algebras}\label{sec:combgentle}

By $k$-algebra we mean an associative algebra with unity over an arbitrary field $k$. Unless otherwise stated, all $k$-algebras are finite dimensional. For a $k$-algebra $A$ consider the category of finite dimensional left $A$-modules $\mod{A}$, and the full subcategory $\proj\, A$ of $\mod{A}$ determined by projective $A$-modules. We fix a complete set $\{S_i\}_{i=1}^n$ of representatives of isomorphism classes of simple $A$-modules, and take for $i=1,\ldots,n$ a projective cover $P_i$ and an injective envelope $E_i$ of $S_i$, respectively.  By {\bf global dimension} $\gldim(\La)\in\NN\cup\{\infty\}$ of $A$ we mean as usual the maximum of lengths of  (minimal) projective resolutions of simples $\{S_i\}$, cf.~\cite{ASS}.

We denote by $\Db{A}:=\Db{\mod{A}}$ the bounded derived category of (bounded complexes of) finite dimensional $A$-modules, see~\cite{HDb2}. Recall that $\Db{A}$ has a structure of triangulated category with translation functor $[1]:\Db{A}\to \Db{A}$ acting as the usual shift of complexes (to the left). The full triangulated subcategory $\per\,A$ of $\Db{A}$ determined by complexes quasi-isomorphic to bounded complexes whose terms are in $\proj\, A$ is called the \textbf{perfect category} of $\La$, and its objects are called \textbf{perfect complexes}. It is well known that $\per\,A$ (resp.~$\Db{A}$) is triangle equivalent with the homotopy category $\Kb(\proj\,\La)$ (resp.~$\Kmb(\proj\,\La)$) of bounded complexes with terms in $\proj\, A$ (resp.~right bounded complexes with terms in $\proj\, A$ and whose cohomology is bounded).
Two algebras $A$ and $A'$ are \textbf{derived equivalent} (written $A\congder A'$) if their derived categories $\Db{A}$ and $\Db{A'}$ are triangle equivalent. Recall that $\per\,A=\Db{A}$ precisely when $A$ has finite global dimension, and that in general $A\congder A'$ if and only if the categories $\per\,A$ and $\per\,A'$ are triangle equivalent, see~\cite{Rickard}.

The \textbf{path algebra} $k\quiver$ of a quiver $\quiver$ is the possibly infinite dimensional $k$-algebra whose underlying vector space has as basis the paths of $\quiver$, and its product is determined on this basis by the concatenation of concatenable paths, and zero otherwise, see \cite{ASS}. Denote by $R_Q$ the two-sided ideal of $k\quiver$ generated by the arrows of $\quiver$. A pair $(\quiver,I)$ is a \textbf{bound quiver} if $I$ is a two-sided ideal of $k\quiver$ with $I \subseteq R_\quiver^2$. Moreover, if $R^m_\quiver \subseteq I$ for some $m \geq 2$, then $I$ is an \textbf{admissible ideal} and $(\quiver,I)$ is an \textbf{admissible bound quiver}. The quotient $k(\quiver,I):=k\quiver/I$ is a \textbf{bound quiver algebra}, which is finite dimensional if and only if $I$ is an admissible ideal. It is well-known that if $k$ is algebraically closed, then every (finite dimensional) $k$-algebra is Morita equivalent to an admissible bound quiver algebra (Gabriel~\cite{GabrielI}). A \textbf{relation} of the ideal $I$ (or of the bound quiver $(Q,I)$) is an element of $I$ given as a $k$-linear combination of parallel paths (that is, with common source and common target) of length at least two. Observe that every admissible ideal $I$ is generated as bimodule by a finite set of relations. A relation of $I$ is \textbf{monomial} if it is a single path, and $I$ is a \textbf{monomial ideal} if it can be generated by finitely many monomial relations (in this case, the bound quiver $(Q,I)$ and bound quiver algebra $k(Q,I)$ are also called monomial). If $I$ is a monomial ideal, then it has a unique minimal monomial generating set, see~\cite{GreenMonomial}.

Given a bound quiver $(\quiver,I)$, a path of $Q$ not belonging to the ideal $I$ is called \textbf{permitted}, and it is a \textbf{maximal permitted path} if it is not properly contained in any permitted path (note that all paths of length at most $1$ are permitted). Observe that a monomial bound quiver algebra $k(\quiver,I)$ is isomorphic to the $k$-vector space generated by all permitted paths of $(\quiver,I)$.

If $A$ is an admissible bound quiver algebra, $A=k(\quiver,I)$, we call $\gldim(\quiver,I):=\gldim(A)$ the global dimension of $(\quiver,I)$, and take $P_i:=A \triv_i$, $S_i:=P_i/\rad(P_i)$ and $E_i:=D(\triv_iA)$ for each vertex $i \in \quiver_0$, where $D(-):=\Hom_k(-,k)$ is the standard duality. We usually identify the set of vertices $\quiver_0$ with the set $\{1,\ldots,|\quiver_0|\}$. In the paper, all $k$-algebras are bound quiver algebras.

\subsection{Gentle bound quivers}\label{subsec:gentle}

In this subsection we collect definitions and mostly well-known properties of gentle bound quivers. For a quiver $\quiver$ and a vertex $i$ in $\quiver$ we denote by $\deg^+(i)$ (resp. $\deg^-(i)$) the number of arrows $\alpha \in \quiver_1$ such that $\sou(\alpha)=i$ (resp. $\tar(\alpha)=i$), and call $\deg(i):=\deg^+(i)+\deg^-(i)$ the \textbf{degree} of $i$.

\begin{definition} \label{def:gentle}
An admissible bound quiver $(\quiver,I)$ is called a \textbf{gentle bound quiver} if the ideal $I$ is generated by paths of length two, and the following conditions hold:
\begin{enumerate}[label={\textnormal{(\alph*)}},topsep=4px,parsep=0px]
 \item $\deg^{+}(i) \leq 2$ and $\deg^{-}(i) \leq 2$ for any vertex $i$ of $\quiver$;
 \item for any arrow $\beta$ of $\quiver$ there is at most one arrow $\alpha$ such that $\alpha\beta \notin I$, and at most one arrow $\gamma$ such that $\beta\gamma \notin I$;
 \item for any arrow $\beta$ of $\quiver$ there is at most one arrow $\alpha$ such that $0 \neq \alpha\beta \in I$, and at most one arrow $\gamma$ such that $0 \neq \beta\gamma \in I$.
\end{enumerate}
Here we always assume that $\quiver$ is connected and has at least one arrow. By \textbf{gentle algebra} we mean the bound quiver algebra $k(\quiver,I)=k\quiver/I$ of a gentle bound quiver.
\end{definition}

Alternatively, an admissible bound quiver $(\quiver,I)$ with ideal $I$ generated by paths of length two is gentle if there are functions $S,T:\quiver_1 \to \{\pm 1\}$ satisfying the following conditions for any pair of arrows $\alpha, \beta\in Q_1$:
\begin{enumerate}[label={\textnormal{(\alph*)}},topsep=4px,parsep=0px]
 \item[(1)] if $\sou(\alpha)=\sou(\beta)$ and $\alpha \neq \beta$, then $S(\alpha) \neq S(\beta)$;
 \item[(2)] if $\tar(\alpha)=\tar(\beta)$ and $\alpha \neq \beta$, then $T(\alpha) \neq T(\beta)$;
 \item[(3)] if $\tar(\alpha)=\sou(\beta)$, then $\beta \alpha \in I$ if and only if $T(\alpha)=S(\beta)$.
\end{enumerate}
The functions $S,T$ are called \textbf{string functions} for $(\quiver,I)$, cf.~\cite{Bob, GeiRei05}. We  extend them to the set of non-trivial walks of $Q$ by taking $S(\alpha^{-1}):=T(\alpha)$, $T(\alpha^{-1}):=S(\alpha)$ for $\alpha\in Q_1$, and
$S(\wlk):=S(\aoi_\ell)$, $T(\wlk):=T(\aoi_1)$ for a walk $\wlk=\aoi_{1}\cdots \aoi_\ell$ in $Q$ with $\aoi_t \in \quiver_1\cup\invw{\quiver}_1$ for $t=1,\ldots,\ell \geq 1$.

Following~\cite{AAG}, if $(\quiver,I)$ is a gentle bound quiver, then a maximal permitted path is a \textbf{non-trivial permitted thread} of $(\quiver,I)$. Moreover, the path $\triv_i$ is a \textbf{trivial permitted thread} if either $\deg(i)=1$, or $\deg(i)=2$ and there are arrows $\alpha,\beta$ with $\sou(\alpha)=\tar(\beta)=i$ such that $\alpha \beta \notin I$. For the sake of precision, we also consider \textbf{split permitted thread} which is a pair of paths $(\eta'',\eta')$ such that $\eta''\eta'$ is a permitted thread, and the vertex $\sou(\eta'')=\tar(\eta')$ is called the \textbf{center} of $(\eta'',\eta')$. The set of (split) permitted threads of $(Q,I)$ is denoted by $\perm_{\quiver,I}$ (respectively, by $\Sperm_{\quiver,I}$).  A  path of $Q$ is called \textbf{forbidden} if it contains no permitted subpath of length two, and it is called a \textbf{non-trivial forbidden thread} if it is not properly contained in any forbidden path. Moreover, the path $\triv_i$ is a \textbf{trivial forbidden thread} if either $\deg(i)=1$, or $\deg(i)=2$ and there are arrows $\alpha,\beta$ with $\sou(\alpha)=\tar(\beta)=i$ such that $\alpha\beta \in I$. By \textbf{split forbidden thread} we mean a pair of paths $(\theta'',\theta')$ such that $\theta''\theta'$ is a forbidden thread, and the vertex $\sou(\theta'')=\tar(\theta')$ is called the \textbf{center} of $(\theta'',\theta')$. The set of (split) forbidden threads of $(Q,I)$ is denoted by $\forb_{\quiver,I}$ (respectively, by $\Sforb_{\quiver,I}$).  A closed forbidden path $\wlk=\alpha_{1}\cdots\alpha_\ell$ without repeated arrows is called a \textbf{full cycle} if $\alpha_\ell\alpha_1\in I$, and a \textbf{almost-full cycle} at $\sou(\wlk)$ if $\alpha_\ell\alpha_1\notin I$.

The following technical facts are known and follow easily from Definition~\ref{def:gentle}, see for instance~\cite{BekMer} and~\cite{OPS18}. We provide a sketch of proof for an illustration,
see also  Example~\ref{exa:amiot}.

\begin{lemma}\label{lem:iota}
Let $(\quiver,I)$ be a gentle bound quiver.
\begin{enumerate}[label={\textnormal{(\alph*)}},topsep=4px,parsep=0px]
 \item Every non-trivial permitted path of $(\quiver,I)$ can be extended to a unique permitted thread.
 \item Every vertex of $\quiver$ is center of exactly two split permitted threads.
\end{enumerate}
\end{lemma}
\begin{proof}
If $\alpha_1$ is an arrow of $Q$ such that $\alpha_1 \alpha_2$ and $\alpha_1 \alpha'_2$ are permitted paths of $(\quiver,I)$, then $\tar(\alpha_2)=\sou(\alpha_1)=\tar(\alpha'_2)$ and $T(\alpha_2) \neq S(\alpha_1) \neq T(\alpha'_2)$, which implies that $\alpha_2=\alpha'_2$. Similarly, if $\alpha_0\alpha_1$ and $\alpha'_0\alpha_1$ are permitted paths of $(\quiver,I)$ then $\alpha_0=\alpha'_0$. This shows that  $\alpha_1$ is contained in at most one (non-trivial) permitted thread, cf.~\cite[Lemma~3]{BekMer}. And at least one such permitted thread exists since $I$ is admissible. By analogous arguments the claim holds for any permitted path of length greater than 1.

For  (b) recall that it was observed in~\cite[1.3]{OPS18} (cf.~\cite[2.2]{AAG}) that every vertex of $Q$ appears exactly two times along all permitted threads (this includes the situation when some permitted thread $\eta$ passes 2 times through a given vertex $i$, that is, $\eta$ self-crosses at $i$). In our terminology this means precisely that every vertex of $\quiver$ is a center of exactly two split permitted threads. 
\end{proof}

\begin{remark}\label{rem:conveta}
By Lemma~\ref{lem:iota}(a) it follows that every arrow $\alpha$ of $Q$ belongs to a unique permitted thread $\eta$ and the position of $\alpha$ in $\eta$ is unique (since obviously permitted paths have no repeated arrows). 
Hence we will usually index non-trivial permitted threads with Greek letters $\eta^{\alpha},\eta^{\beta},\ldots$ and their constituent arrows will be denoted by the corresponding indexed letters $\eta^{\alpha}=\alpha_1\alpha_2\cdots\alpha_{\ell^{\alpha}}$ and $\eta^{\beta}=\beta_1\beta_2\cdots\beta_{\ell^{\beta}}$. Observe that with this notation all monomial relations of length two in $(\quiver,I)$ are of the form $\alpha_r\beta_s$ or $\alpha_r\alpha_s$ with $s \neq r+1$ (to be precise, these relations form the unique minimal monomial generating set for $I$). 
We indicate these relations with dotted lines in diagrams.
\end{remark}

\begin{example}\label{exa:amiot} (a) Consider the following gentle bound quiver $(Q,I)$ together with (split) permitted and forbidden threads:
\vspace{-0.1cm}
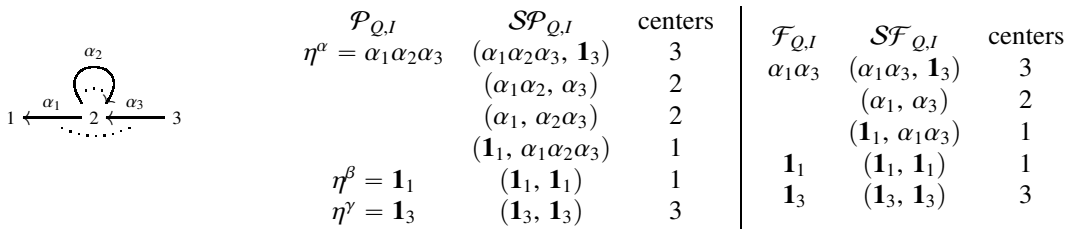
\[
\begin{array}{l l |l}
{\xymatrix@C=1.8pc@R=1.5pc{
\mathmiddlescript{1} &
\mathmiddlescript{2} \ar[l]_-{\alpha_1}="a1" \ar@(ul,ur)^-{\alpha_2}="a2" &
\mathmiddlescript{3}\ar[l]_-{\alpha_3}="a3" & \\
\ar@{.}@/^3pt/"a3"!<-0.5ex,-2ex>;"a1"!<0.8ex,-2ex>
\ar@{.}@/^3pt/"a2"!<-6pt,8pt>;"a2"!<6pt,8pt>
}} & {\begin{array}{ccc}\perm_{\quiver,I}&\Sperm_{\quiver,I}&\text{centers}\\
\eta^\alpha=\alpha_1\alpha_2\alpha_3&(\alpha_1\alpha_2\alpha_3,\,\triv_3)&3\\
&(\alpha_1\alpha_2,\,\alpha_3)&2\\
&(\alpha_1,\,\alpha_2\alpha_3)&2\\
&(\triv_1,\,\alpha_1\alpha_2\alpha_3)&1\\
\eta^\beta=\triv_1&(\triv_1,\,\triv_1)&1\\
\eta^\gamma=\triv_3&(\triv_3,\,\triv_3)&3\\
\end{array}}
& {\begin{array}{ccc}\forb_{\quiver,I}&\Sforb_{\quiver,I}&\text{centers}\\
\alpha_1\alpha_3&(\alpha_1\alpha_3,\,\triv_3)&3\\
&(\alpha_1,\,\alpha_3)&2\\
&(\triv_1,\,\alpha_1\alpha_3)&1\\
\triv_1&(\triv_1,\,\triv_1)&1\\
\triv_3&(\triv_3,\,\triv_3)&3\\
\end{array}}
\\
\end{array}
\]
The ideal $I$ is generated by relations $\alpha_1\alpha_3$ and $\alpha_2^2$. Observe that $\alpha_2$ is a full cycle, and  $(Q,I)$ has infinite global dimension, cf.~Lemma~\ref{lem:fgd}.

(b) For later use, we consider also the following examples of gentle bound quivers with corresponding permitted threads:
\vspace{-0.1cm}
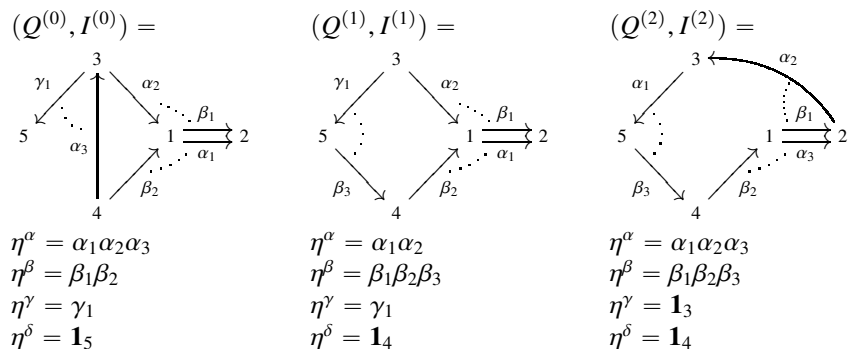
\[
\begin{array}{l l l l l l}
 (Q^{(0)},I^{(0)})= & & (Q^{(1)},I^{(1)})= & & (Q^{(2)},I^{(2)})= & \\
{\xymatrix@C=1.5pc@R=1.5pc{
{} & \mathmiddlescript{3} \ar[ld]_-{\gamma_1}="c1" \ar[rd]^-{\alpha_2}="a2" \\
\mathmiddlescript{5} & & \mathmiddlescript{1} \ar@<.5ex>[r]^-{\beta_1}="b1" \ar@<-.5ex>[r]_-{\alpha_1}="a1" & \mathmiddlescript{2} \\
{} & \mathmiddlescript{4} \ar[ru]_-{\beta_2}="b2" \ar[uu]^(.42){\alpha_3}="a3"
\ar@{.}@/^3pt/"a3"!<1ex,1.5ex>;"c1"!<1.5ex,-2ex>
\ar@{.}@/^3pt/"a2"!<-2ex,-1.5ex>;"b1"!<-1.5ex,-0.6ex>
\ar@{.}@/^4pt/"a1"!<0ex,2.8ex>;"b2"!<0ex,1.5ex>}} &  &
{\xymatrix@C=1.5pc@R=1.5pc{
 {} & \mathmiddlescript{3} \ar[ld]_-{\gamma_1}="c1" \ar[rd]^-{\alpha_2}="a2" \\
\mathmiddlescript{5} \ar[rd]_-{\beta_3}="b3" & &
\mathmiddlescript{1} \ar@<.5ex>[r]^-{\beta_1}="b1" \ar@<-.5ex>[r]_-{\alpha_1}="a1" & \mathmiddlescript{2} \\
{} & \mathmiddlescript{4} \ar[ru]_-{\beta_2}="b2" &&
\ar@{.}@/^3pt/"c1"!<1ex,-2.5ex>;"b3"!<1ex,3ex>
\ar@{.}@/^3pt/"a2"!<-2ex,-1.5ex>;"b1"!<-1.5ex,-0.6ex>
\ar@{.}@/^4pt/"a1"!<0ex,2.8ex>;"b2"!<0ex,1.5ex>}} &  &
{\xymatrix@C=1.5pc@R=1.5pc{
{} & \mathmiddlescript{3} \ar[ld]_-{\alpha_1}="a1" \\
\mathmiddlescript{5} \ar[rd]_-{\beta_3}="b3" & & \mathmiddlescript{1} \ar@<.5ex>[r]^-{\beta_1}="b1" \ar@<-.5ex>[r]_-{\alpha_3}="a3" & \mathmiddlescript{2} \ar@/_10pt/[llu]_-{\alpha_2}="a2" \\
{} & \mathmiddlescript{4} \ar[ru]_-{\beta_2}="b2"
\ar@{.}@/^3pt/"a1"!<1ex,-2.5ex>;"b3"!<1ex,3ex>
\ar@{.}@/^3pt/"b1"!<-1ex,-1.5ex>;"a2"!<0ex,-2ex>
\ar@{.}@/^4pt/"a3"!<0ex,2.8ex>;"b2"!<0ex,1.5ex>}} &  \\
 \eta^{\alpha}=\alpha_1\alpha_2\alpha_3 & & \eta^{\alpha}=\alpha_1\alpha_2 & & \eta^{\alpha}=\alpha_1\alpha_2\alpha_3 \\
 \eta^{\beta}=\beta_1\beta_2 & & \eta^{\beta}=\beta_1\beta_2\beta_3 & & \eta^{\beta}=\beta_1\beta_2\beta_3 \\
 \eta^{\gamma}=\gamma_1 & & \eta^{\gamma}=\gamma_1 & & \eta^{\gamma}=\triv_3 \\
 \eta^{\delta}=\triv_5 & & \eta^{\delta}=\triv_4 & & \eta^{\delta}=\triv_4 \\
\end{array}
\]
\end{example}

Forbidden threads have similar properties to permitted threads in case of finite global dimension.
To see this we recall the following known fact.

\begin{lemma}[{{\cite[Proposition~2.1]{Bob2cycle}}}]\label{lem:fgd}
A gentle bound quiver $(\quiver,I)$ has finite global dimension if and only if it contains no full cycle.
\end{lemma}

\begin{lemma}\label{lem:forbidden}
Let $(\quiver,I)$ be a gentle bound quiver of finite global dimension.
\begin{enumerate}[label={\textnormal{(\alph*)}},topsep=4px,parsep=0px]
 \item Every non-trivial forbidden path of $(\quiver,I)$ can be extended to a unique forbidden thread.
 \item Every vertex of $\quiver$ is center of exactly two split forbidden threads.
 \item For each vertex $i$ of $Q$ there exists at most one almost-full cycle at $i$ in $(Q,I)$ and it has length at least 2.
\end{enumerate}
\end{lemma}
\begin{proof}
Since $(\quiver,I)$ has no full cycle by Lemma~\ref{lem:fgd}, then every non-trivial forbidden path can be extended to a forbidden thread. The uniqueness of such thread follows the lines of the proof of Lemma~\ref{lem:iota}(a). The claim (b)  similarly follows by arguments as in Lemma~\ref{lem:iota}(b). 
To see (c) assume that there are two different almost-full cycles $\wlk=\anp_1\ldots\anp_\ell$ and $\wlk'=\anp'_1\ldots\anp'_{\ell'}$ at $i\in Q_0$, for some arrows $\anp_t,\anp_{t'}\in Q_1$, for $1\leq t\leq \ell$, $1\leq t'\leq \ell'$. 
 By definition $\wlk$ and $\wlk'$ are closed forbidden paths without repeated arrows with $\sou(\wlk)=\tar(\wlk)=\sou(\wlk')=\tar(\wlk')=i$ and $\anp_\ell\anp_1, \anp'_{\ell'}\anp'_1\notin I$. Gentleness of  $(Q,I)$ easily implies that $\{\anp_1,\ldots,\anp_\ell\}\cap\{\anp'_1,\ldots,\anp'_{\ell'}\}=\emptyset$ and $\anp'_{\ell'}\anp_1, \anp_\ell\anp'_1\in I$. This means that $\wlk'\wlk$ is a full cycle -- contradiction with Lemma~\ref{lem:fgd}. Finally, if $\wlk=\anp_1$ is an almost-full cycle of length 1 then it means that $\anp_1^2\notin I$, hence $\alpha^n_1\notin I$ for each $n\geq 1$, which is impossible since $I$ is admissible ($kQ/I$ has finite dimension).
\end{proof}

\subsection{Marked ribbon graphs}\label{subs:mrg}

A graph $G=(V,H,\s,\iota)$ together with a permutation $\permut$ of the set of half-edges $H$ whose orbits coincide with the inverse images $\s^{-1}(\alpha)$ for any vertex $\alpha$ of $G$ is called a \textbf{ribbon graph}, cf.~\cite[Definition~1.2]{OPS18}. Ribbon graphs have a wide range of applications, especially in connection to orientable surfaces. They were used in the analysis of Brauer graph algebras~\cite{MarshSchroll} and in connection to trivial extension of gentle algebras~\cite{Schroll.trivext}, see Subsection~\ref{subsec:applA} below.  The latter connection makes use of a function $\m:V \to H$ satisfying $\s(\m(\alpha))=\alpha$ for each vertex $\alpha \in V$, called a \textbf{marking} of the ribbon graph $(G,\permut)$, cf.~\cite[Definition~1.5]{OPS18} (recall our assumption $V=\Img(\s)$). Observe that the marking $\m$ partitions the set of half-edges $H$ of $G$ into linearly ordered sets (a disjoint union of chains joint by the involution $\iota$), which is a formulation better adjusted to our purposes:

\begin{definition}\label{def:mrg}
Let $G=(V,H,\s,\iota)$ be a graph. A pair $(G, \leq)$ is called a \textbf{marked ribbon graph} if $\leq$ is a partial order on the set of half-edges $H$ of $G$ such that $h \leq h'$ or $h' \leq h$ if and only if $\s(h)=\s(h')$ for $h,h' \in H$. Here we further assume that $G$ is a connected graph having a vertex of degree at least two.
\end{definition}

\begin{remark}\label{rem:correspondence}
If $(G,\permut)$ is a connected ribbon graph with marking $\m$, then for a vertex $\alpha \in \verts{G}$ we take
\[
\s^{-1}(\alpha):=\{h^{\alpha}_0,h^{\alpha}_1,\ldots,h^{\alpha}_{\ell^{\alpha}-1},h^{\alpha}_{\ell^{\alpha}}\} \subseteq \hedgs{G},
\]
where $h^{\alpha}_0$ is the marked half-edge $\m(\alpha)$, and such that $\permut(h^{\alpha}_r)=h^{\alpha}_{r-1}$ for $r=1,\ldots,\ell^{\alpha}$ and $\permut(h^{\alpha}_0)=h^{\alpha}_{\ell^{\alpha}}$. For $h,h' \in \hedgs{G}$, define $h \leq h'$ precisely when there is a vertex $\alpha$ of $G$ such that $h=h^{\alpha}_r$ and $h'=h^{\alpha}_{r'}$ with $0 \leq r' \leq r \leq \ell^{\alpha}$. Then $(G,\leq)$  is a marked ribbon graph, and the set $\m(\verts{G})$ consists of the maximal elements of $\hedgs{G}$ with respect to $\leq$. Clearly, the permutation $\permut$ and the marking $\m$ can be recovered from the partial order $\leq$.
\end{remark}

In the following definition we construct the marked ribbon graph associated with a gentle bound quiver as in~\cite{OPS18}, in a slightly more precise form. We also provide a converse operation below (Definition~\ref{def:IncBoundQuiver}).  These constructions establish a formal foundation for our graph theoretic model.

\begin{definition}\label{def:BoundQuiverInc}
For a gentle bound quiver $(\quiver,I)$ define a pair $\GPC(\quiver,I):=(G,\leq)$ with $G$ a graph and $\leq$ a relation on the set of half-edges of $G$ as follows. The graph $G=(V,H,\s,\iota)$ has as set of vertices the permitted threads of $(\quiver,I)$, $V:=\perm_{\quiver,I}$, as set of half-edges the split permitted threads of $(\quiver,I)$, $H:=\Sperm_{\quiver,I}$, the function $\s$ is given by $\s(\eta'',\eta'):=\eta''\eta'$, and the involution $\iota(\eta'',\eta')$ is the unique element in $\Sperm_{\quiver,I}-\{(\eta'',\eta')\}$ having the same center as $(\eta'',\eta')$, cf.~Lemma~\ref{lem:iota}(b). The relation $\leq$ on $H$ is given by $(\eta''_0,\eta'_0)\leq (\eta'',\eta')$ if and only if $\eta''_0\eta'_0=\eta''\eta'$ and $\ell(\eta'_0) \leq \ell(\eta')$.
\end{definition}
We easily verify that given a (connected) gentle bound quiver $(\quiver,I)$ the pair $\GPC(\quiver,I)$ is a (connected) marked ribbon graph having  a vertex of degree at least two.
Although the above definition is somehow cryptic, the construction is straightforward once we know the permitted threads of $(\quiver,I)$. Indeed, the vertices of $G$ may be viewed as the indices $\alpha,\beta,\ldots$ of the permitted threads, the edges of $G$ are in bijection with the vertices of $\quiver$ (by taking the center of the respective split permitted thread), and the linear orders $\leq$ record the vertices of $\quiver$ as found along each permitted thread of $(\quiver,I)$, compare with~\cite[Definition~1.8]{OPS18}. These vertices are written graphically (from right to left) as $\vBlock{\alpha}{123}$ for a permitted a thread ${}^{\am{1}}\alpha_1^{\am{2}}\alpha_2^{\am{3}}$. Note that $i$ is a loop of $G$ at vertex $\alpha$ exactly when the permitted thread $\eta^{\alpha}$ self-crosses at vertex $i$. Let us consider some examples illustrating these conventions.

\begin{example}\label{ex:small} The following table exhibits three \lq\lq small'' gentle bound quivers together with the associated marked ribbon graphs.
\[
\begin{NiceArray}{w{r}{1.5cm} w{l}{2cm} w{c}{3cm} w{r}{2cm} w{l}{3.5cm} w{c}{10mm}w{r}{3mm}w{l}{5mm}}
\Block{1-2}{\text{\small Gentle bound quiver $(\quiver,I)$}} & & \text{\small Marked ribbon graph $(G,\leq)$} & \Block{1-5}{\text{\small Vertices $\verts{G}$ and edges $\edgs{G}:=\hedgs{G}/\iota$ of $G$}} & & & & \\
\RowStyle[cell-space-limits=5pt]{}
\Block{2-1}{\xymatrix@C=1pc{{}_1 \ar@(ul,ur)^-{\alpha_1}="a" & {}
\ar@{.}@/^3pt/"a"!<-6pt,8pt>;"a"!<6pt,8pt> }}
& \eta^{\alpha}={}^{\am{1}}\alpha_1^{\am{1}}
& \Block{2-1}{\xymatrix{  {\vBlock{\alpha}{1\, 1}} \ar@{-}@(ld,lu)^(.5){1}}} & \verts{G}= &  \{\eta^{\alpha}\} & & &  \\
 & & & H(G)/\iota=&\{(\triv_1,\alpha_1),(\alpha_1,\triv_1)\} & \longleftrightarrow & 1 & =\quiver_0 \\ \\
\RowStyle[cell-space-limits=5pt]{}
\Block{2-1}{\xymatrix{\mathmiddlescript{1} \ar@/^7pt/[r]^-{\alpha_1}="a" & \mathmiddlescript{2} \ar@/^7pt/[l]^-{\alpha_2}="b"
\ar@{.}@/^3pt/"a"!<7pt,-6pt>;"b"!<7pt,7pt> }}
& \eta^{\alpha}={}^{\am{2}}\alpha_1^{\am{1}}\alpha_2^{\am{2}} &
\Block{2-1}{\xymatrix{ {\vBlock{\alpha}{212}}  \ar@{-}@(ld,lu)^(.5){2} \ar@{-}[r]^-{1} & {\vBlock{\beta}{1}} }} & \verts{G} = & \{\eta^{\alpha},\eta^{\beta}\} &  & & \\
  & \eta^{\beta} = \triv_1^{\am{1}} & & \Block{2-1}{H(G)/\iota=} & \{(\alpha_1,\alpha_2),(\triv_1,\triv_1)\} & \Block{2-1}{\longleftrightarrow} & 1 & \Block{2-1}{=\quiver_0} \\
  & & & & \{(\triv_2,\alpha_1\alpha_2),(\alpha_1\alpha_2,\triv_2)\} & & 2 & \\ \\
\RowStyle[cell-space-limits=5pt]{}
\Block{2-1}{\xymatrix{\mathmiddlescript{1} \ar@/^7pt/[r]^-{\alpha_1}="a" & \mathmiddlescript{2} \ar@/^7pt/[l]^-{\beta_1}="b"
\ar@{.}@/^3pt/"a"!<7pt,-6pt>;"b"!<7pt,7pt>
\ar@{.}@/^7pt/"b"!<-3pt,6pt>;"a"!<-3pt,-4pt> }}
& \eta^{\alpha}={}^{\am{2}}\alpha_1^{\am{1}} &
\Block{2-1}{\xymatrix{ {\vBlock{\alpha}{21}} \ar@{-}@<.5ex>[r]^-{2} \ar@{-}@<-.5ex>[r]_-{1} & {\vBlock{\beta}{12}} }} & \verts{G}= & \{\eta^{\alpha},\eta^{\beta}\} & & & \\
  & \eta^{\beta}={}^{\am{1}}\beta_1^{\am{2}} & & \Block{2-1}{H(G)/\iota=} & \{(\alpha_1,\triv_1),(\triv_1,\beta_1)\} & \Block{2-1}{\longleftrightarrow} & 1 &\Block{2-1}{=\quiver_0}\\
  & & & & \{(\triv_2,\alpha_1),(\beta_1,\triv_2)\} & & 2 & \\
\CodeAfter
 \SubMatrix{\{}{3-5}{3-5}{\}}
 \SubMatrix{\{}{3-7}{3-7}{\}}
 \SubMatrix{\{}{6-5}{7-5}{\}}
 \SubMatrix{\{}{6-7}{7-7}{\}}
 \SubMatrix{\{}{10-5}{11-5}{\}}
 \SubMatrix{\{}{10-7}{11-7}{\}}
\end{NiceArray}
\]
\end{example}
\begin{definition}\label{def:IncBoundQuiver}
For a marked ribbon graph $(G,{\leq})$ define a bound quiver $\GBQ(G,\leq):=(\quiver,I)$ as follows. The quiver $\quiver$ has as set of vertices the edges of $G$, $\quiver_0:=\edgs{G}=\hedgs{G}/\iota$ where $\iota$ is the involution of $G$, and as set of arrows $\quiver_1:=\{(h',h) \in \hedgs{G} \times \hedgs{G} \mid \text{ $h$ is direct predecessor of $h'$ with respect to $\leq$}\}$, with $\sou(h',h):=\{h,\iota h\}$ and $\tar(h',h):=\{h',\iota h'\}$. The two-sided ideal $I$ is generated by all paths of length two $(h'_2,h_2)(h'_1,h_1)$ such that $h'_1=\iota h_2$.
\end{definition}
We easily verify that given a (connected) marked ribbon graph $(G,{\leq})$ the pair $\GBQ(G,\leq)$ is a (connected) gentle bound quiver having at least one arrow.
\begin{remark}\label{rem:linear}
For a vertex $\alpha$ of a marked ribbon graph $(G,\leq)$ we have a linear order on the half-edges of $G$ that are incident to $\alpha$,
\begin{equation}\label{eq:linear}
\s^{-1}(\alpha)=\{h^{\alpha}_0 > h^{\alpha}_1 > \ldots > h^{\alpha}_{t-1}>h^{\alpha}_t>\ldots >h^{\alpha}_{\ell^{\alpha}-1}>h^{\alpha}_{\ell^{\alpha}}\},
\end{equation}
with $h^{\alpha}_t$ direct predecessor of $h^{\alpha}_{t-1}$, cf. Remark~\ref{rem:correspondence}. Analogously to the conventions in Remark~\ref{rem:conveta}, we fix the following notation for the arrows and  permitted threads of the gentle bound quiver $\GBQ(G,\leq)$. If $\ell^\alpha>0$, then we denote by $\alpha_t$ the arrow $(h^{\alpha}_{t-1},h^{\alpha}_t)$ for $t=1,\ldots,\ell^\alpha$, and by $\eta^{\alpha}$ the path $\alpha_1\cdots\alpha_{\ell^{\alpha}}$ of $\GBQ(G,\leq)=(Q,I)$. By construction, every non-trivial path $\alpha_t\cdots \alpha_{t'}$ with $1 \leq t \leq t' \leq \ell^\alpha$ is permitted, and $\eta^\alpha$ is a maximal permitted path hence a (non-trivial) permitted thread. If $\ell^\alpha=0$, then we set $\eta^{\alpha}:=\triv_{\{h_0^\alpha,\iota h_0^\alpha\}}$ and we check directly that it is a trivial permitted thread. Note that  every permitted thread in $\GBQ(G,\leq)$ can be found in the way.
\end{remark}

Consider the following examples (cf. Example~\ref{exa:amiot}),\label{ex:Amiotmrg}
\[
\xymatrix@C=1pc@R=.8pc{ {\vBlock{\gamma}{53}} \ar@{-}[rr]^-{3} \ar@{-}[dd]_-{5} && {\vBlock{\alpha}{213}} \\ \\
*++{\vBlock{\beta}{2145}} \ar@{-}[rr]_-{4} \ar@{-}@<.5ex>[rruu]^-{1} \ar@{-}@<-.5ex>[rruu]_-{2} && {\vBlock{\delta}{4}} }
\qquad \qquad
\xymatrix@C=1pc@R=.8pc{ \eta^{\gamma}={}^{\am{5}}\gamma_1^{\am{3}} \ar@{-}[rr]^-{3} \ar@{-}[dd]_-{5} && \eta^{\alpha}={}^{\am{2}}\alpha_1^{\am{1}}\alpha_2^{\am{3}} \\ \\
\eta^{\beta}={}^{\am{2}}\beta_1^{\am{1}}\beta_2^{\am{4}}\beta_3^{\am{5}} \ar@{-}[rr]_-{4} \ar@{-}@<.5ex>[rruu]^-{1} \ar@{-}@<-.5ex>[rruu]_-{2} && \eta^{\delta}=\triv_4^{\am{4}} }
\qquad \qquad
\xymatrix@C=1.2pc@R=1.2pc{
 {} & \mathmiddlescript{3} \ar[ld]_-{\gamma_1}="c1" \ar[rd]^-{\alpha_2}="a2" \\
\mathmiddlescript{5} \ar[rd]_-{\beta_3}="b3" & &
\mathmiddlescript{1} \ar@<.5ex>[r]^-{\beta_1}="b1" \ar@<-.5ex>[r]_-{\alpha_1}="a1" & \mathmiddlescript{2} \\
{} & \mathmiddlescript{4} \ar[ru]_-{\beta_2}="b2" &&
\ar@{.}@/^3pt/"c1"!<1ex,-2.5ex>;"b3"!<1ex,3ex>
\ar@{.}@/^3pt/"a2"!<-2ex,-1.5ex>;"b1"!<-1.5ex,-0.6ex>
\ar@{.}@/^4pt/"a1"!<0ex,2.8ex>;"b2"!<0ex,1.5ex>}
\]
\[
\xymatrix@C=1pc@R=.8pc{ {\vBlock{\gamma}{112}} \ar@{-}@(ld,rd)_(.2){1} \ar@{-}[rr]^-{2} && {\vBlock{\alpha}{23}} \ar@{-}[dd]^-{3} \\ \\
\vBlock{\beta}{545} \ar@{-}[rr]_-{4} \ar@{-}@(lu,ru)^(.2){5} && {\vBlock{\delta}{43}} }
\qquad \qquad
\xymatrix@C=1pc@R=.8pc{ \eta^{\gamma}={}^{\am{1}}\gamma_1^{\am{1}}\gamma_2^{\am{2}} \ar@{-}@(ld,rd)_(.2){1} \ar@{-}[rr]^-{2} && \eta^{\alpha}={}^{\am{2}}\alpha_1^{\am{3}} \ar@{-}[dd]^-{3} \\ \\ \eta^{\beta}={}^{\am{5}}\beta_1^{\am{4}}\beta_2^{\am{5}} \ar@{-}[rr]_-{4} \ar@{-}@(lu,ru)^(.2){5} && \eta^{\delta}={}^{\am{4}}\delta_1^{\am{3}} }
\qquad \qquad
\xymatrix@C=2pc@R=1.2pc{
 {} & \mathmiddlescript{2} \ar[ld]_(.3){\gamma_2}="c2" \ar@{<-}[rd]^(.3){\alpha_1}="a1" \\
\mathmiddlescript{1} \ar@(lu,ru)^-{\gamma_1}="c1" & \mathmiddlescript{5} \ar@/_10pt/[d]_-{\beta_2}="b2" & \mathmiddlescript{3}  \\ {} & \mathmiddlescript{4} \ar[u]_-{\beta_1}="b1" \ar@{<-}[ru]_-{\delta_1}="d1" &&
\ar@{.}@/^3pt/"b1"!<0ex,.5ex>;"b2"!<1.6ex,1ex>
\ar@{.}@/^3pt/"b1"!<-.5ex,1.5ex>;"d1"!<-1.5ex,2ex>
\ar@{.}@/^4pt/"a1"!<-1.8ex,-.8ex>;"c2"!<1.8ex,-.8ex>
\ar@{.}@/^3pt/"c1"!<-6pt,8pt>;"c1"!<6pt,8pt>}
\]

Note that the construction of Definition~\ref{def:IncBoundQuiver} when applied to the graph $G$ with two vertices joint by a single edge, which is not included in the definition of marked ribbon graph (\ref{def:mrg}), produces a quiver $\quiver$ with one vertex and no arrow, which we do not consider to be gentle (Definition~\ref{def:gentle}) mainly because it has no permitted threads.

It follows from \cite{OPS18,APS23} that the constructions of Definitions \ref{def:BoundQuiverInc} and \ref{def:IncBoundQuiver} are inverse to each other, see also \cite[Theorem 4.10]{PPP19}. For convenience, we provide a brief explicit proof.
\begin{lemma}\label{thm:IncGentle}
Let $(G,{\leq})$ be a marked ribbon graph and  $(\quiver,I)$ a  gentle bound quiver. Then
\begin{enumerate}[label={\textnormal{(\alph*)}},topsep=3px,parsep=0px]
\item $(G,\leq)$ and $\GPC(\GBQ(G,\leq))$ are isomorphic as marked ribbon graphs,
\item $(\quiver,I)$ and $\GBQ(\GPC(\quiver,I))$ are isomorphic as bound quivers.
\end{enumerate}
\end{lemma}
\begin{proof} 
 Taking $(G',\leq'):= \GPC(\GBQ(G,\leq))$ and $(\quiver',I') := \GBQ(\GPC(\quiver,I))$ we consider the following mappings:
\[
\xymatrix@R=.1pc{\verts{G} \ar[r]^{f_V} & \verts{G'} \\
\alpha \ar@{|->}[r] & \eta^{\alpha}} \qquad
\xymatrix@R=.1pc{\hedgs{G} \ar[r]^{f_H} & \hedgs{G'} \\
h^{\alpha}_t \ar@{|->}[r] & (\eta^{\alpha}_{[t]},{}_{[t]}\eta^{\alpha}) } \qquad
\xymatrix@R=.1pc{\quiver_0 \ar[r]^{g_0} & \quiver'_0 \\
i \ar@{|->}[r] & \{h_i,\iota h_i \} } \qquad
\xymatrix@R=.1pc{\quiver_1 \ar[r]^{g_1} & \quiver'_1 \\
\alpha_r \ar@{|->}[r] & ((\eta^{\alpha}_{[r-1]},{}_{[r-1]}\eta^{\alpha}),(\eta^{\alpha}_{[r]},{}_{[r]}\eta^{\alpha})), }
\]
where $\{h_i,\iota h_i\}$ is the pair of split permitted threads of $(\quiver,I)$ with center $i$ as in Lemma~\ref{lem:iota}(b). By Remark~\ref{rem:linear} the mappings $f_V$ and $f_H$ are well-defined bijections and a routine check shows that they define a graph isomorphism between $G$ and $G'$ (cf. \ref{subsec:bdg}) such that $f_H$ is an isomorphism between the partially ordered sets $(H(G),\leq)$ and $(H(G'),\leq')$, hence it commutes with marked ribbon graph structures. Similarly, by Lemma~\ref{lem:iota} and Remark~\ref{rem:conveta} the mappings $g_0$ and $g_1$ are well-defined bijections
 and  $g=(g_0,g_1):\quiver \to \quiver'$ is an isomorphism of quivers (that is, $\tar'\circ g_1=g_0\circ \tar$ and $\sou'\circ g_1=g_0\circ \sou$) such that the induced mapping $g:kQ\to kQ'$ sends the generating monomial relations of $I$ to those of $I'$. 
\end{proof}

\begin{lemma}\label{rem:string}
If $(\quiver,I)$ is a gentle bound quiver with marked ribbon graph $(G,\leq)$, then there is a $1-1$ correspondence between string functions of $(\quiver,I)$ and directions of the graph $G$. In particular, there exist precisely $2^{|Q_0|}$ pairs of string functions for $(Q,I)$.
\end{lemma}
\begin{proof}
By Lemma~\ref{thm:IncGentle} we can assume that $(\quiver,I)=\GBQ(G,\leq)$ as in Definition~\ref{def:IncBoundQuiver}. If $S,T$ are string functions of $(\quiver,I)$, define a bidirection $\sigma^{S,T}$ of $G$ by taking, if $\s^{-1}(\alpha)=\{h^{\alpha}_0>h^{\alpha}_1>\ldots >h^{\alpha}_{\ell^{\alpha}}\}$ for a vertex $\alpha$ of $G$,
\begin{equation}\label{eq:stringONE}
\sigma^{S,T}(h^{\alpha}_t):=\left\{
\begin{array}{l l}
S(\alpha_t), & \text{if $0<t\leq \ell^\alpha$}, \\
(-1)T(\alpha_1), & \text{if $0=t<\ell^{\alpha}$}, \\
(-1)\sigma^{S,T}(\iota h^{\alpha}_t), & \text{if $0=t=\ell^{\alpha}$}.
\end{array} \right.
\end{equation}
Recall that the permitted thread corresponding to $\alpha$ is $\eta^\alpha=\alpha_1\cdots\alpha_{\ell^\alpha}$ where $\sou(\alpha_t)=\{h^\alpha_t,\iota h^\alpha_t\}$ and $\tar(\alpha_t)=\{h^\alpha_{t-1},\iota h^\alpha_{t-1}\}$ (see Remark~\ref{rem:linear} and Definition~\ref{def:IncBoundQuiver}). Take $h^{\beta}_s:=\iota h^{\alpha}_t$ for some vertex $\beta$. The bidirection $\sigma^{S,T}$ is well-defined, because $\ell^{\alpha}=0$ implies that the permitted thread $\eta^\beta$ is non-trivial ($\ell^{\beta}>0$), since $G$ is connected and has a vertex of degree at least two. That $\sigma^{S,T}$ is in fact a direction (that is, $\sigma(h^\alpha_t) \neq \sigma(h^\beta_s)$, see Definition~\ref{def:bidi}) follows from definition if $\ell^\alpha\ell^\beta=0$, and from the properties of string functions if $\ell^\alpha\ell^\beta>0$: if $st>0$, then $\sigma^{S,T}(h^\alpha_t)=S(\alpha_t)$ and $\sigma^{S,T}(h^\beta_s)=S(\beta_s)$, and since $\sou(\alpha_t)=\{h^\alpha_t,h^\beta_s\}=\sou(\beta_s)$ we have $S(\alpha_t) \neq S(\beta_s)$. The case $st=0$ can be shown similarly.

Now take an arbitrary direction $\sigma$ of $G$  and define two functions $S^{\sigma}, T^{\sigma}:Q_1\to \{\pm1\}$ as
\begin{equation}\label{eq:stringTWO}
\text{$S^{\sigma}(\alpha_t):=\sigma(h^{\alpha}_{t})$ and $T^{\sigma}(\alpha_t):=-\sigma(h^{\alpha}_{t-1})$ for an arrow $\alpha_t=(h^{\alpha}_{t-1},h^{\alpha}_{t})$ of $\quiver$.}
\end{equation}
To show that $S^{\sigma},T^{\sigma}$ are string functions recall  that $\sou(\alpha_t)=\{\rhalf{\alpha_t},\iota \rhalf{\alpha_t}\}$, $\tar(\alpha_t)=\{\lhalf{\alpha_t},\iota \lhalf{\alpha_t}\}$  where we set
 $\lhalf{\alpha_t}:=h^{\alpha}_{t-1}$ and $\rhalf{\alpha_t}:=h^{\alpha}_{t}$ for any arrow $\alpha_t=(h^{\alpha}_{t-1},h^{\alpha}_{t})$ as above (see Definition~\ref{def:IncBoundQuiver}).   Then  for any pair of  arrows $\alpha_t$ and $\beta_s$ of $\quiver$ we have:
\begin{itemize}[topsep=4px,parsep=0px]
 \item if $\alpha_t\neq \beta_s$ and $\sou(\alpha_t)=\sou(\beta_s)$, then $\rhalf{\alpha_t}=\iota \rhalf{\beta_s}$ and $S^{\sigma}(\alpha_t)=\sigma(\rhalf{\alpha_t}) = \sigma(\iota \rhalf{\beta_s}) = -\sigma(\rhalf{\beta_s}) = -S^{\sigma}(\beta_s)$;
 \item if $\alpha_t\neq \beta_s$ and $\tar(\alpha_t)=\tar(\beta_s)$, then $\lhalf{\alpha_t}=\iota \lhalf{\beta_s}$ and $T^{\sigma}(\alpha_t)=-\sigma(\lhalf{\alpha_t})= -\sigma(\iota \lhalf{\beta_s})= \sigma(\lhalf{\beta_s}) = -T^{\sigma}(\beta_s)$;
 \item if $\sou(\alpha_t)=\tar(\beta_s)$ and $\rhalf{\alpha_t}=\iota \lhalf{\beta_s}$ (that is, $\alpha_t\beta_s\in I$ ), then $S^{\sigma}(\alpha_t)=\sigma(\rhalf{\alpha_t})=\sigma(\iota \lhalf{\beta_s})=-\sigma(\lhalf{\beta_s})=T^{\sigma}(\beta_s)$;
 \item if $\sou(\alpha_t)=\tar(\beta_s)$ and $\rhalf{\alpha_t}=\lhalf{\beta_s}$ (that is, $\alpha_t\beta_s\notin I$ ), then $S^{\sigma}(\alpha_t)=\sigma(\rhalf{\alpha_t})=\sigma(\lhalf{\beta_s})=-T^{\sigma}(\beta_s)$. In this case we necessarily have $\beta=\alpha$ and $s=t+1$, since $\alpha_t=(h^{\alpha}_{t-1},h^{\alpha}_{t})$ and $\beta_s=(h^{\beta}_{s-1},h^{\beta}_{s})$ with $h^\alpha_t=\rhalf{\alpha_t}=\lhalf{\beta_s}=h^\beta_{s-1}$, see~\eqref{eq:linear}.
\end{itemize}
This shows that $S^{\sigma},T^{\sigma}$ are string functions of $(\quiver,I)$.
Clearly, the constructions $(S,T) \mapsto \sigma^{S,T}$  and $\sigma \mapsto (S^{\sigma},T^{\sigma})$ are inverse to each other,
which concludes the proof of the main statement. To see the final claim recall that $|Q_0|=|\edgs{G}|$ and clearly, there exist $2^{|\edgs{G}|}$ directions in $G$.
\end{proof}

For a gentle bound quiver $(\quiver,I)$ of finite global dimension an analogous construction of Definition~\ref{def:BoundQuiverInc}, replacing ``permitted'' by ``forbidden'' notions (cf. Lemma~\ref{lem:forbidden}), yields also a marked ribbon graph together with a bidirection of the graph. We give notation and brief details for convenience.

\begin{definition}\label{rem:forb}
For a gentle bound quiver $(\quiver,I)$ of finite global dimension define a triple $(\fmrgs{G},\fmrgs{\leq},\fmrgs{\sigma})$ with $(\fmrgs{G},\fmrgs{\sigma})$ a bidirected graph and $\fmrgs{\leq}$ a relation on the set of half-edges of $\fmrgs{G}$ as follows. The graph $\fmrgs{G}=(V,H,\s,\iota)$ has as set of vertices the forbidden threads of $(\quiver,I)$, $V:=\forb_{\quiver,I}$, as set of half-edges the split forbidden threads of $(\quiver,I)$, $H:=\Sforb_{\quiver,I}$, the function
$\s$ is given by $\s(\theta'',\theta'):=\theta''\theta'$, and the involution $\iota(\theta'',\theta')$ is the unique element in $\Sforb_{\quiver,I}-\{(\theta'',\theta')\}$ having the same center as $(\theta'',\theta')$, cf.~Lemma~\ref{lem:forbidden}(b). The relation $\fmrgs{\leq}$ on $H$ is given by $(\theta''_0,\theta'_0)\leq (\theta'',\theta')$ if and only if $\theta''_0\theta'_0=\theta''\theta'$ and $\ell(\theta'_0) \leq \ell(\theta')$.  The bidirection $\fmrgs{\sigma}:H \to\{\pm1\}$ is defined, for the linearly ordered half-edges
$\s^{-1}(\alpha)=\{\fmrgs{h}^{\alpha}_0 > \fmrgs{h}^{\alpha}_1 > \ldots >\fmrgs{h}^{\alpha}_{\ell^{\alpha}}\}$ incident to each vertex $\alpha\in V$, by setting $\fmrgs{\sigma}(\fmrgs{h}^{\alpha}_t):=(-1)^{\ell^{\alpha}-t}$ for $t=0,\ldots, \ell^{\alpha}$. We call the triple $(\fmrgs{G},\fmrgs{\leq},\fmrgs{\sigma})$ the \textbf{forbidden marked ribbon graph} of $(\quiver, I)$, and denote it by $\GFC(\quiver,I)$.
\end{definition}

We verify that $(\fmrgs{G},\fmrgs{\leq})$ (called the \textbf{dual marked ribbon graph} of $(Q,I)$) constructed in the first part of Definition \ref{rem:forb} is a well-defined (connected) marked ribbon graph. Recall that in the geometric model of a gentle bound quiver $(Q,I)$, the marked ribbon graph $\GPC(Q,I)$ (as in Definition \ref{def:BoundQuiverInc}) is in some sense uniquely embedded into an orientable surface with boundary, where the edges of the graph determine a dissection of the surface (called admissible dissection in~\cite{OPS18,APS23}, or cellular dissection in~\cite{PPP19}). In the case of finite global dimension, the dual marked ribbon graph $(\fmrgs{G},\fmrgs{\leq})$  corresponds to a so-called lamination or dual dissection of $(G,\leq)$ in the surface, and the gentle bound quiver $\GBQ(\fmrgs{G},\fmrgs{\leq}^{\rm op})$ is isomorphic to the Koszul dual of $(Q,I)$, see~\cite[Definition~4.4]{PPP19} and~\cite[Proposition~1.25]{OPS18}, where $\fmrgs{\leq}^{\rm op}$ denotes the order on $H(\fmrgs{G})$ inverse to $\fmrgs{\leq}$.
In comparison to graphs in the geometric model we consider here the graph $\fmrgs{G}$ with  additional structure -- the bidirection $\fmrgs{\sigma}$. In Sections~\ref{sec:derinv} and~\ref{sec:AR} we will use the bidirected graph $(\fmrgs{G},\fmrgs{\sigma})$ to combinatorially bypass the explicit inversion of Cartan matrices in the analysis of Euler forms. 

\subsection{Walks and indecomposable perfect complexes}\label{subsec:perfcpx}

In this subsection we apply the results of Bekkert-Merklen \cite{BekMer} (see also \cite{Bob,ALP}) to describe the indecomposable perfect complexes of a gentle bound quiver $(\quiver,I)$ in terms of walks of the corresponding marked ribbon graph $(G,\leq)$. Recall that all non-trivial permitted paths of $(\quiver,I):=\GBQ(G,\leq)$ are of the form $\alpha_{s+1}\cdots \alpha_t=(h^{\alpha}_s,h^{\alpha}_{s+1})\cdots (h^{\alpha}_{t-1},h^{\alpha}_t)$ for some vertex $\alpha \in \verts{G}$, see Remark~\ref{rem:linear}. Consider the (unique) ordered edges $i$ and $j$ of $G$ with $\souh(i)=h^{\alpha}_s$ and $\tarh(j)=h^{\alpha}_t$, and observe that the target and source of the path $\alpha_{s+1}\cdots \alpha_t$ are the unordered edges $\vtx{i},\vtx{j}\in \quiver_0:=\edgs{G}$ respectively, and that $j \neq \invw{i}$. Thus the situation around the vertex $\alpha$ of $G$ may be visualized as follows (we use the convention of \eqref{eq:ordedg}, and the red arrows symbolize the corresponding arrows of $Q$):
          $$
          {\xymatrix@C=2.8pc@R=0.03pc{
 &\ar@{-}[rdddd]^-{h_{s+1}^\alpha}="hsp"& \cdots&\ar@{-}[ldddd]_-{h_{t-1}^\alpha}="htm"&\\
 && && \\
 && && \\
  && && \\
\ar@{<-}[r]_-{\iota h_s^\alpha}&& {\alpha}\ar@{|-}[l]^-{h_s^\alpha}="hs"\ar@{<-}[r]_-{h_t^\alpha}="ht" &\!\!\!&\ar@{|-}[l]^-{\iota h_t^\alpha} \\
 &i& &j&\\
 \ar@[red]@/_4pt/"ht"!<0ex,1ex>;"htm"!<0ex,1.5ex>_{\mathmiddlescript{\begin{array}{c} \\ \am{\alpha_t}\end{array}}}
 \ar@[red]@/_4pt/"hsp"!<0ex,1ex>;"hs"!<0ex,1.5ex>_{\mathmiddlescript{\begin{array}{c} \\ \am{\alpha_{s+1}}\end{array}}}
 }}$$
 \vspace{-0.5cm}

 In what follows we will use the notation
\[
\ppath{i}{j}:=\alpha_{s+1}\cdots \alpha_t.
\]
Conversely, a pair of ordered edges $i$ and $j$ of $G$ is concatenable if and only if the half-edges $\souh(i)$ and $\tarh(j)$ are comparable with the partial order $\leq$, that is, $\souh(i)=h^{\alpha}_s$ and $\tarh(j)=h^{\alpha}_t$ for some vertex $\alpha \in \verts{G}$. If $s=t$ then $j=\invw{i}$ and the walk $ij$ is not reduced. If $s<t$ (that is, $h^{\alpha}_s>h^{\alpha}_t$) we take $\deg(ij):=1$ and $\ppath{i}{j}:=\alpha_{s+1}\cdots \alpha_t$ as above; if $s>t$ we take $\deg(ij)=-1$ and $\ppath{i}{j}:=\invw{\ppath{\invw{j}}{\invw{i}}}$. If $\wlk=i_1i_2\cdots i_Li_{L+1}$ is a reduced walk ($L\geq 0$) we take the \textbf{degree} of $\wlk$ as
\begin{equation}\label{eq:degree}
\deg(\wlk):=\sum_{t=1}^L\deg(i_ti_{t+1}).
\end{equation}
In particular,  $\deg(i_1)=0$ and $\deg(\wlk^{-1})=-\deg(\wlk)$.
Moreover, we have the following fact:

\begin{remark}\label{rem:letters}
The construction $\ppath{-}{-}$ provides a bijection between the set of reduced walks of $G$ of length two and the set of non-trivial permitted paths of $(\quiver,I)$ together with their inverses (the so-called homotopy letters of $(\quiver,I)$, see discussion below).
\end{remark}

As usual, the morphism between indecomposable projective (left) modules $\xymatrix{P_{i} \ar[r] & P_{j}}$ given by right multiplication with a path is denoted by same symbol as the path (where we write $P_i$ instead of $P_{\vtx{i}}$, recall that $\quiver_0=\edgs{G}$).

For a reduced walk $\wlk=i_1i_2\cdots i_Li_{L+1}$ of $(G,\leq)$, $L\geq 0$, consider the following ``unfolded diagram'' similarly as in~\cite{ALP},
\[
\xymatrix@C=3pc{P_{i_1} \ar@{-}[r]^-{\ppath{i_1}{i_2}} & P_{i_2} \ar@{-}[r]^-{\ppath{i_2}{i_3}} & \cdots \ar@{-}[r]^-{\ppath{i_{L-1}}{i_L}} & P_{i_L} \ar@{-}[r]^-{\ppath{i_L}{i_{L+1}}} & P_{i_{L+1}}}
\]
where $\xymatrix{P_{i} \ar@{-}[r]^-{\ppath{i}{j}} & P_{j}}$ denotes the morphism $\xymatrix{P_{i} \ar[r]^-{\ppath{i}{j}} & P_{j}}$ if $\deg(ij)=1$ or $\xymatrix{P_{i} \ar@{<-}[r]^-{\invw{\ppath{i}{j}}} & P_{j}}$ if $\deg(ij)=-1$. We refer to \cite[Section 2]{ALP} for description of the transition of an unfolded diagram to the corresponding ``folded'' complex. Here we illustrate the procedure with a couple of representative examples. Consider the following marked ribbon graph $(G,\leq)$ with corresponding gentle bound quiver (to simplify the notation for walks we draw the edges of $G$ together with a fixed order, cf.~\eqref{eq:ordedg}),
\[
\xymatrix@C=1pc@R=.8pc{ {\vBlock{\gamma}{53}} \ar@{<-|}[rr]^-{3} \ar@{|->}[dd]_-{5} && *++[o]{\vBlock{\alpha}{6213}} \ar@{|->}[dd]^-{6} \\ \\
*++[o]{\vBlock{\beta}{2145}} \ar@{<-|}[rr]_-{4} \ar@{<-|}@<.5ex>[rruu]^-{1} \ar@{<-|}@<-.5ex>[rruu]_-{2} && {\vBlock{\delta}{46}} }
\qquad
\xymatrix@C=1.2pc@R=1.2pc{
 {} & \mathmiddlescript{3} \ar[ld]_-{\gamma_1}="c1" \ar[rd]^-{\alpha_3}="a3" \\
\mathmiddlescript{5} \ar[rd]_-{\beta_3}="b3" & &
\mathmiddlescript{1} \ar@<.5ex>[r]^-{\beta_1}="b1" \ar@<-.5ex>[r]_-{\alpha_2}="a2" & \mathmiddlescript{2} \ar[d]^-{\alpha_1}="a1" \\
{} & \mathmiddlescript{4} \ar[ru]^-{\beta_2}="b2" && \mathmiddlescript{6} \ar[ll]^-{\delta_1}="d1"
\ar@{.}@/^3pt/"c1"!<1ex,-2.5ex>;"b3"!<1ex,3ex>
\ar@{.}@/^3pt/"a3"!<-2ex,-1.5ex>;"b1"!<-1.5ex,-0.6ex>
\ar@{.}@/^4pt/"d1"!<3ex,0ex>;"a1"!<-3ex,-1ex>
\ar@{.}@/^4pt/"a2"!<0ex,2.8ex>;"b2"!<3.3ex,-1ex>
\ar@{.}@/^4pt/"b2"!<0ex,-2.5ex>;"d1"!<-1.5ex,1.5ex>
\ar@{.}@/^4pt/"b1"!<0ex,2ex>;"a1"!<.5ex,1ex>  }
\qquad
\xymatrix@!0@C=40pt@R=25pt{ &&& 2_{\succ}\invw{3}_{\pred}\invw{5}_{\pred}4_{\succ}6_{\succ}\invw{2}_{\succ}1 \\ P_2 \ar@{-}[r]^-{\ppath{2}{\invw{3}}} & P_3 \ar@{-}[r]^-{\ppath{\invw{3}}{\invw{5}}} & P_5 \ar@{-}[r]^-{\ppath{\invw{5}}{4}} & P_4 \ar@{-}[r]^-{\ppath{4}{6}} & P_6 \ar@{-}[r]^-{\ppath{6}{\invw{2}}} & P_2 \ar@{-}[r]^-{\ppath{\invw{2}}{1}} & P_1 \\
 P_2 \ar[r]^-{\alpha_2\alpha_3} & P_3 \ar@{<-}[r]^-{\gamma_1} & P_5 \ar@{<-}[r]^-{\beta_3} & P_4 \ar[r]^-{\delta_1} & P_6 \ar[r]^-{\alpha_1} & P_2 \ar[r]^-{\beta_1} & P_1. }
\]
On the right we show a reduced walk $\wlk=2\invw{3}\invw{5}46\invw{2}1$ with its corresponding unfolded diagram, where $2_{\succ}\invw{3}$ means
$\souh(2)\,\succ\, \tarh(\invw{3})$
(equivalently,  $\deg(2\invw{3})=1$). Below we show the corresponding complex,
\[
\xymatrix@!0@C=40pt@R=10pt{ & P_2 \ar[r]^-{\alpha_2\alpha_3} & P_3 & \\ \\ P_4 \ar[r]^{\beta_3} \ar[rdd]_-{\delta_1} & P_5 \ar[ruu]_-{\gamma_1} & & \\ \\ & P_6 \ar[r]^-{\alpha_1} & P_2 \ar[r]^-{\beta_1} & P_1 } \qquad
\xymatrix@!0@C=70pt@R=10pt{ & P_2 & & \\ & \oplus & P_3 & \\ P_4 \ar[r]^-{\left[ \begin{smallmatrix} 0\\\beta_3\\ \delta_1\end{smallmatrix} \right]} & P_5 \ar[r]^-{\left[ \begin{smallmatrix} \alpha_2\alpha_3 & \gamma_1 & 0 \\ 0&0&\alpha_1 \end{smallmatrix} \right]} & \oplus \ar[r]^-{\left[ \begin{smallmatrix} 0&\beta_1\end{smallmatrix} \right]} & P_1. \\
& \oplus & P_2 & \\ & P_6 }
\]
Similarly, with the reduced walk $\wlk'=3_{\pred}\invw{6}_{\pred}\invw{4}_{\succ}5_{\succ}3$ we obtain the following unfolded diagram and corresponding complex,
\[
\xymatrix@!0@C=40pt@R=20pt{  P_3 \ar@{-}[r]^-{\ppath{3}{\invw{6}}} & P_6 \ar@{-}[r]^-{\ppath{\invw{6}}{\invw{4}}} & P_4 \ar@{-}[r]^-{\ppath{\invw{4}}{5}} & P_5 \ar@{-}[r]^-{\ppath{5}{3}} & P_3 \\
 P_3 \ar@{<-}[r]^-{\alpha_1\alpha_2\alpha_3} & P_6 \ar@{<-}[r]^-{\delta_1} & P_4 \ar[r]^-{\beta_3} & P_5 \ar[r]^-{\gamma_1} & P_3 }
\qquad
\xymatrix@!0@C=40pt@R=10pt{ P_4 \ar[rdd]_-{\beta_3} \ar[r]^-{\delta_1} & P_6 \ar[r]^-{\alpha_1\alpha_2\alpha_3} & P_3 \\ \\ & P_5 \ar[r]_-{\gamma_1} & P_3 }
\qquad \qquad
\xymatrix@!0@C=70pt@R=10pt{ & P_6 & P_3 \\
P_4 \ar[r]^-{\left[ \begin{smallmatrix} \delta_1\\ \beta_3 \end{smallmatrix} \right]} & \oplus \ar[r]^-{\left[ \begin{smallmatrix} \alpha_1\alpha_2\alpha_3 &0 \\ 0&\gamma_1 \end{smallmatrix} \right]} & \oplus \\  & P_5  & P_3. }
\]

A reduced walk $\wlk=i_1i_2\cdots i_Li_{L+1}$ of $(G,\leq)$ is a \textbf{belt} if $i_1=i_{L+1}$, $\wlk_{[L]}$ is a non-trivial ($L \geq 1$) primitive closed walk and $\deg(\wlk)=0=\deg(i_Li_1i_2)$. Such walks (as the example $\wlk'$ above) produce an additional family of complexes as follows. For an indecomposable automorphism $\mu$ of a $k$-vector space $K$ consider the complex,
\[
\xymatrix@!0@C=80pt@R=10pt{ P_4 \otimes_k K \ar[rdd]_-{\beta_3 \otimes \unit} \ar[r]^-{\delta_1 \otimes \unit} & P_6 \otimes_k K \ar[r]^-{\alpha_1\alpha_2\alpha_3 \otimes \mu} & P_3 \otimes_k K \ar@{=}[dd] \\ \\ & P_5 \otimes_k K \ar[r]_-{\gamma_1 \otimes \unit} & P_3 \otimes_k K } \qquad
\xymatrix@!0@C=110pt@R=10pt{ & P_6 \otimes_k K &  \\
P_4 \otimes_k K \ar[r]^-{\left[ \begin{smallmatrix} \delta_1 \otimes \unit\\ \beta_3 \otimes \unit \end{smallmatrix} \right]} & \oplus \ar[r]^(.51){\left[ \begin{smallmatrix} \alpha_1\alpha_2\alpha_3 \otimes \mu &\gamma_1 \otimes \unit \end{smallmatrix} \right]} & P_3 \otimes_k K. \\  & P_5 \otimes_k K  & }
\]

The following description of the indecomposable perfect complexes over a  gentle algebra is a reformulation of the classification of Bekkert-Merklen \cite{BekMer} (cf.~\cite{Bob,ALP})  in terms of walks of the corresponding marked
ribbon graph.
\begin{proposition}\label{pro:walkBelt}
Let $(\quiver,I)$ be a gentle bound quiver with marked ribbon graph $(G,\leq)$. Let $A:=kQ/I$.
\begin{enumerate}[label={\textnormal{(\alph*)}},topsep=4px,parsep=0px]
 \item If $x=(m,\wlk)$ with $m \in \Z$ and $\wlk=i_1i_2\cdots i_Li_{L+1}$ a non-trivial reduced walk of $G$, then there is an indecomposable perfect complex $\stringcpx{x}$ over $A$ given by the following unfolded diagram,
\[
\xymatrix@C=3pc{P_{i_1} \ar@{-}[r]^-{\ppath{i_1}{i_2}} & P_{i_2} \ar@{-}[r]^-{\ppath{i_2}{i_3}} & \cdots \ar@{-}[r]^-{\ppath{i_{L-1}}{i_L}} & P_{i_L} \ar@{-}[r]^-{\ppath{i_L}{i_{L+1}}} & P_{i_{L+1}},}
\]
where the module $P_{i_1}$ is positioned at cohomological degree $m$.

 \item If $x=(m,\wlk,\mu)$ with $m \in \Z$, $\wlk=i_1i_2\cdots i_Li_{L+1}$ a belt of $(G,\leq)$ and $\mu$ an indecomposable automorphism of a finite dimensional $k$-vector space $K$, then there is an indecomposable perfect complex $\stringcpx{x}$ over $A$ given by the following unfolded diagram,
\[
\xymatrix@C=3pc{P_{i_1}\otimes_k K \ar@{=}@/_10pt/[rrrr] \ar@{-}[r]^-{\ppath{i_1}{i_2}\otimes \mu} & P_{i_2}\otimes_k K \ar@{-}[r]^-{\ppath{i_2}{i_3}\otimes \unit} & \cdots \ar@{-}[r]^-{\ppath{i_{L-1}}{i_L}\otimes \unit} & P_{i_L}\otimes_k K \ar@{-}[r]^-{\ppath{i_L}{i_{L+1}}\otimes \unit} & P_{i_{L+1}}\otimes_k K,}
\]
where the module $P_{i_1}$ is positioned at cohomological degree $m$.

 \item Every indecomposable perfect complex over $A$ is isomorphic to $\stringcpx{x}$ for some $x$ as in $($a$)$ or $($b$)$.
\end{enumerate}
\end{proposition}

Observe that by construction $\stringcpx{m,\wlk}[1]\cong \stringcpx{m-1,\wlk}$ for any $m,\wlk$ as in (a), and $\stringcpx{m,\wlk,\mu}[1]\cong \stringcpx{m-1,\wlk,\mu}$ for any $m,\wlk,\mu$ as in (b). For convenience, we set $\stringcpx{m,\wlk}$ to be the zero complex for any $m\in\ZZ$ and a trivial walk $\wlk$. Our proof of Proposition~\ref{pro:walkBelt} relies on Bekkert-Merklen Theorem~\cite{BekMer} for which we recall some terminology. If $(\quiver,I)$ is a gentle bound quiver, a non-trivial walk $\wlk$ in $\quiver$ is called a \textbf{direct homotopy letter} if it is a permitted path in $(\quiver,I)$, and an \textbf{inverse homotopy letter} if $\invw{\wlk}$ is a permitted path in $(\quiver,I)$. If $\wlk$ is any non-trivial reduced walk of $Q$, then there is a decomposition $\wlk=\sigma_1(\wlk)\cdots \sigma_L(\wlk)$ with $\sigma_t(\wlk)$ a homotopy letter for $t=1,\ldots,L>0$. Such decomposition is unique if $L$ is minimal, in which case we say that $\sigma_1(\wlk)\cdots \sigma_L(\wlk)$ is the \textbf{homotopy presentation} of $\wlk$. The \textbf{degree} of $\wlk$ is the difference $\deg(\wlk)$ of the number of direct homotopy letters minus the number of inverse letters in the homotopy presentation of $\wlk$. To be precise, given string functions $S,T$ of $(\quiver,I)$ and concatenable homotopy letters $\wlk',\wlk''$ then $\wlk'\wlk''$ is a homotopy presentation if and only if
\[
S(\wlk')T(\wlk'')=\deg(\wlk')\deg(\wlk''),
\]
cf.~\cite[p. 645]{Bob}. Indeed, if $\deg(\wlk')\deg(\wlk'')=-1$ and $S(\wlk')T(\wlk'')=1$ then $\wlk'\wlk''$ in not a reduced walk, and if $\deg(\wlk')\deg(\wlk'')=1$ and $S(\wlk')T(\wlk'')=-1$ then $\wlk'\wlk''$ itself is a homotopy letter. A \textbf{non-trivial homotopy string} of $(\quiver,I)$ is a non-trivial reduced walk of $\quiver$. We consider as \textbf{trivial homotopy strings} the symbols $\triv_{v,\epsilon}$ with $v \in \quiver_0$ and $\epsilon \in \{\pm 1\}$, and take $\deg(\triv_{v,\epsilon})=0$ and $\invw{\triv}_{v,\epsilon}:=\triv_{v,-\epsilon}$. A \textbf{homotopy band} of $(\quiver,I)$ is primitive closed homotopy string $\wlk=\sigma_1(\wlk)\cdots \sigma_L(\wlk)$ such that $\deg(\wlk)=0$ and $S(\wlk)T(\wlk)=-1=\deg(\sigma_1(\wlk))\deg(\sigma_L(\wlk))$ for any string functions $T,S$ of $(\quiver,I)$. Denote by $\walksred_{\geq t}(G)$ the set of reduced walks of a graph $G$ having length at least $t$.

\begin{lemma}\label{lem:walkGQ}
Let $(\quiver,I)$ be a gentle bound quiver with marked ribbon graph $(G,\leq)$.
\begin{enumerate}[label={\textnormal{(\alph*)}},topsep=4px,parsep=0px]
 \item Assume that $i,j$ and $j_0,k$ are two concatenable pairs of ordered edges of $G$ with $j \neq \invw{i}$ and $k \neq \invw{j}_0$.
\begin{enumerate}[label={\textnormal{(\alph*)}},topsep=4px,parsep=0px]
 \item[\rm (a1)] The walks $\ppath{i}{j}$ and $\ppath{j_0}{k}$ of $\quiver$ are concatenable if and only if $j_0=\pmw{j}$.
 \item[\rm (a2)] If $j_0=j^{\epsilon}$ for $\epsilon=\pm$ then $S(\ppath{i}{j})=\epsilon \delta T(\ppath{j^{\epsilon}}{k})$ for any string functions $S,T$ of $(\quiver,I)$, where $\delta=\deg(\ppath{i}{j})\deg(\ppath{j^{\epsilon}}{k})$. That is, $\ppath{i}{j}\ppath{j_0}{k}$ is a homotopy presentation if and only if $j_0=j$.
\end{enumerate}

 \item The mapping $\GBQ:\walksred_{\geq 2}(G) \longrightarrow \walksred_{\geq 1}(Q)$ given by $\wlk:=i_1 \cdots i_{L+1} \mapsto \GBQ(\wlk):=\ppath{i_1}{i_2}\ppath{i_2}{i_3} \cdots \ppath{i_L}{i_{L+1}}$ is a bijection, whose inverse is denoted by $\GPC$.

 \item $\GBQ(\invw{\wlk})=\invw{\GBQ(\wlk)}$ and $\deg(\GBQ(\wlk))=\deg(\wlk)$, for any $\wlk\in \walksred_{\geq 2}(G)$.

 \item $\GBQ(\wlk)$ is a homotopy band of $(\quiver,I)$ if and only if $\wlk$ is a belt of $(G,\leq)$, for any $\wlk\in \walksred_{\geq 2}(G)$.
\end{enumerate}
\end{lemma}
\begin{proof}
Note that the pair $(\ppath{i}{j},\ppath{j_0}{k})$ is concatenable if and only if $\sou(\ppath{i}{j})=\vtx{j}=\vtx{j_0}=\tar(\ppath{j_0}{k})$, that is, if and only if $j_0=\pmw{j}$. To show (a2) let $\sigma$ be a direction of the edges of $G$ with corresponding string functions $S,T$ as in Lemma~\ref{rem:string}. A direct computation yields
\[
S(\ppath{i}{j})=\deg(\ppath{i}{j})\sigma(\tarh(j)) \qquad \text{and} \qquad
T(\ppath{j^{\epsilon}}{k})=-\epsilon \deg(\ppath{j^{\epsilon}}{k})\sigma(\souh(j))=\epsilon \deg(\ppath{j^{\epsilon}}{k})\sigma(\tarh(j)).
\]
Hence we have $S(\ppath{i}{j})=\epsilon \deg(\ppath{i}{j})\deg(\ppath{j^{\epsilon}}{k}) T(\ppath{j^{\epsilon}}{k})$. Due to  (a2) and our description of homotopy letters via $\ppath{-}{-}$ (Remark \ref{rem:letters}),  claim (b) is a reformulation of the existence and uniqueness of homotopy presentations for non-trivial reduced walks of $\quiver$.

Claim (c) follows by construction since $\ppath{\invw{j}}{\invw{i}}=\invw{\ppath{i}{j}}$ and definitions of degree. To show (d) take $\wlk=i_1\cdots i_Li_{L+1}\in\walksred_{\geq 2}(G)$. Assume first that $\GBQ(\wlk)$ is a homotopy band. Then $\deg(\wlk)=0=\deg(i_Li_1i_2)$ by (c), and using (a) we get $i_{L+1}=i_1$. If $\wlk':=\wlk_{[L]}=i_1\cdots i_L$ is not primitive then $\wlk'=\wt{\wlk}^r$ for some $r>1$ and some walk $\Wlk=j_1\cdots j_{\ell}$ with $j_1=i_1=i_{L+1}$ and $j_{\ell}=i_L$. Then
\[
\GBQ(\wlk)=\GBQ(\wlk')\ppath{i_L}{i_{L+1}}=[\GBQ(\wt{\wlk})\ppath{j_{\ell}}{j_1}\GBQ(\wt{\wlk})\cdots \GBQ(\wt{\wlk}) \ppath{j_{\ell}}{j_1} \GBQ(\wt{\wlk})]\ppath{i_L}{i_{L+1}}=[\GBQ(\wt{\wlk})\ppath{j_{\ell}}{j_1}]^r,
\]
in contradiction with $\GBQ(\wlk)$ being primitive. This shows that $\wlk$ is a belt. Conversely, assume that $\wlk$ is a belt and observe that $\GBQ(\wlk)$ is a closed walk with $S(\GBQ(\wlk))=-T(\GBQ(\wlk))$ for any string functions $S,T$ by (a), and that the conditions on the degree in the definition of a band follow using (c). Assume that the closed walk $\GBQ(\wlk)$ is not primitive. That is, using (b), we get a walk $\Wlk=j_1\cdots j_{\ell}$ in $G$ such that $\GBQ(\wlk)=\GBQ(\Wlk)^r$ for some $r>1$. Since $\GBQ(\Wlk)$ is closed then $j_{\ell}=\pmw{j_1}$, and since $T(\GBQ(\wt{\wlk}))=T(\GBQ(\wt{\wlk})^r)=T(\GBQ(\wlk))=-S(\GBQ(\wlk))=-S(\GBQ(\wt{\wlk}))$, then $j_{\ell}=j_1$ and the following is a homotopy presentation of $\GBQ(\wt{\wlk})^r$,
\[
[\ppath{j_1}{j_2}\cdots \ppath{j_{\ell-2}}{j_{\ell-1}}]\ppath{j_{\ell-1}}{j_{\ell}}[\ppath{j_1}{j_2}\cdots \ppath{j_{\ell-2}}{j_{\ell-1}}]\ppath{j_{\ell-1}}{j_{\ell}}\cdots \ppath{j_{\ell-1}}{j_{\ell}} [\ppath{j_1}{j_2}\cdots \ppath{j_{\ell-2}}{j_{\ell-1}}]\ppath{j_{\ell-1}}{j_{\ell}},
\]
that is, $\wlk=\GPC(\GBQ(\wt{\wlk})^r)=[(j_1\cdots j_{\ell-1})\cdots (j_1\cdots j_{\ell-1})]j_{\ell}=\wlk'^rj_{\ell}=\wlk'^ri_{L+1}$ where $\wlk':=j_1\cdots j_{\ell-1}$, in contradiction with $\wlk_{[L]}$ being primitive. This completes the proof.
\end{proof}

\begin{proof}[Proof of Proposition~\ref{pro:walkBelt}]
Fix string functions $S,T$ of $(\quiver,I)$ and a corresponding direction $\sigma$ of $G$ as in Lemma~\ref{rem:string}. We extend the bijection $\GBQ$ of Lemma~\ref{lem:walkGQ}(b) to include trivial strings by taking $\GBQ^{S,T}(i):=\triv_{\vtx{i},-1}$ if $i$ is a direct arrow of $(G,\sigma)$ and $\GBQ^{S,T}(i):=\triv_{\vtx{i},+1}$ if $i$ is an inverse arrow (and $\GBQ^{S,T}(\wlk):=\GBQ(\wlk)$ if $\ell(\wlk)>1$). This establishes a bijection $\GBQ^{S,T}$ (which we sometimes denote also by $\GBQ$) 
between the set of reduced walks $\walksred_{\geq 1}(G)$ and the set of homotopy strings of $(\quiver,I)$. For a homotopy string $\wlk'$ and an integer $m \in \Z$, the description of a corresponding indecomposable perfect complex $\stringCpx{m,\wlk'}$ was given by Bekkert and Merklen in~\cite[Theorem~3]{BekMer}, see details in~\cite[Proposition~3.1]{Bob}. Our folded presentation for a walk $\wlk$ of $(G,\leq)$ coincides with the corresponding presentation given in~\cite{ALP}, so that
\begin{equation}\label{eq:BMcpx}
\xymatrix{(m,\wlk) \ar[rr] && \stringcpx{m,\wlk}=\stringCpx{m,\GBQ^{S,T}(\wlk)}. }
\end{equation}
This shows claim (a). By Lemma~\ref{lem:walkGQ}(d), the restriction of the extended function $\GBQ$ to the belts of $(G,\leq)$ is precisely the set of bands of $(\quiver,I)$, and claim (b) follows similarly. The last claim (c) follows also from~\cite[Theorem~3]{BekMer} by the bijectivity of $\GBQ^{S,T}$.
\end{proof}

The terms homotopy string and band were introduced in~\cite{Bob}. Note that if $L=0$ (that is, $\wlk=i_1$) then $\stringcpx{m,\wlk}$ is the stalk complex  $P_{i_1}$ concentrated in degree $m$. A particular consequence of~\cite[Theorem~3]{BekMer} is that the set of isomorphism classes of indecomposable perfect complexes of a gentle algebra $A=k\quiver/I$, denoted by $\Ind(\per \, A)$, partitions into two subsets
\begin{equation}\label{eq:perstbd}
\Ind(\per \, A)= \IndSt \cup \IndBd,
\end{equation}
where $\IndSt$ denotes the classes of the complexes of Proposition~\ref{pro:walkBelt}(a), called \textbf{string complexes}, and $\IndBd$ denotes the classes of the complexes of Proposition~\ref{pro:walkBelt}(b), called \textbf{band complexes} of $A$. Then there is a bijection
\begin{equation}\label{eq:stnoniso}
\xymatrix{[\Z \times \walksred_{\geq 1}(G)]/\sim \ar@{<->}[rr]^-{1-1} & &  \IndSt,}
\end{equation}
where $(m,\wlk) \sim (m',\wlk')$ when $(m,\wlk)=(m',\wlk')$, or $m'=m+\deg(\wlk)$ and $\wlk'=\invw{\wlk}$.

\def\Faces{\mathrm{Fa}}
\def\Cdeg{\deg^\circ}
\def\VerPer{\xi}
\newcommand{\forbw}{\mathcal{F}}

\subsection{Faces of a marked ribbon graph and AAG-invariants} \label{subsec:antiwalks}

Let $(G,\permut)$ be a ribbon graph with $G=(V,H,\s,\iota)$ a finite graph. Following~\cite{OPS18}, see also~\cite{OppZvo22}, a (not necessarily reduced) primitive closed walk $\wlk=i_1\cdots i_Li_{L+1}$ ($L \geq 0$) is called a \textbf{face} of $(G,\permut)$ if the half-edge $\souh(i_{t})$ is direct successor of the half-edge $\tarh(i_{t+1})$ in the cyclic orderings determined by the permutation $\permut$ (that is, $\souh(i_{L+1})=\permut(\tarh(i_1))$ and $\souh(i_{t})=\permut(\tarh(i_{t+1}))$ for $t=1,\ldots,L$). For convenience, faces are considered up to cyclic permutations.  A marking $\m$ of a ribbon graph $(G,\permut)$ cuts the faces of $(G,\permut)$ into exactly $m:=|\verts{G}|$ walks that can be given as follows in terms of our definition of marked ribbon graph as disjoint union of chains with an involution having no fixed point, see Definition~\ref{def:mrg} and Remark~\ref{rem:correspondence}.

Fix a marked ribbon graph $(G,\leq)$ with corresponding gentle bound quiver $(Q,I):=\GBQ(G,\leq)$ (cf.~Definition~\ref{def:IncBoundQuiver}). A walk $\wlk=i_1\cdots i_Li_{L+1}$ of $G$ with $L\geq 0$ is called an \textbf{anti-walk} of $(G,\leq)$ if the half-edge $\souh(i_t)$ is direct successor of $\tarh(i_{t+1})$ with respect to the partial order $\leq$ for $t=1,\ldots,L$, and $\wlk$ is maximal with this property. One can show directly that:
\begin{itemize}
\item  anti-walks are reduced,
\vspace{-3pt}
\item the mapping $\GBQ$ of Lemma~\ref{lem:walkGQ}(b) gives a bijection between the set of anti-walks of a marked ribbon graph $(G,\leq)$, denoted by $\forbw_{G,\leq}$, and the set of forbidden threads $\forb_{\quiver,I}$ of the gentle quiver $(\quiver,I)=\GBQ(G,\leq)$ (in case $i$ is an anti-walk of length 1 we set $\GBQ(i):=\triv_{\vtx{i}}$, cf.~the proof of Proposition~\ref{pro:walkBelt}; note that in this case $\triv_{\vtx{i}}$ is a forbidden thread),
\vspace{-3pt}
\item the degree (\ref{eq:degree}) of an anti-walk $\wlk=i_1\cdots i_Li_{L+1}$ is
 \begin{equation}\label{eq:lengthAntiWalk}
 \deg(\wlk)=L=\ell(\GBQ(\wlk)),
 \end{equation}\vspace{-7pt}
\item if $\wlk$ is an anti-walk with target vertex $\alpha=\tar(\wlk)$ then the half-edge $\tarh(\wlk)$ is the maximal element $\m(\alpha)$ of $\s^{-1}(\alpha)$, and for every vertex $\alpha \in \verts{G}$ there is a unique anti-walk with $\tar(\wlk)=\alpha$, that we denote by $\OT{\alpha}:=\wlk$,
\vspace{-3pt}
\item similarly, the source half-edge $\souh(\wlk)$ is the minimal element of $\s^{-1}(\beta)$ where $\beta$ is the source vertex of $\wlk$.
\end{itemize}
Then the target and source functions restricted to anti-walks are bijections in the following commutative diagram, 
\begin{equation}\label{eq:bijections}
\xymatrix@R=1pc{ \verts{G} \ar@{=}[d] & \forbw_{G,\leq} \ar[d]_-{\GBQ} \ar[l]_-{\tar} \ar[r]^-{\sou} & \verts{G} \ar@{=}[d] \\
\perm_{\quiver,I} \ar[r]^-{\Phi_1} & \forb_{\quiver,I} \ar[r]^-{\Phi_2} & \perm_{\quiver,I} }
\end{equation}
where $\Phi_1$ and $\Phi_2$ are the bijections defined in~\cite[2.2]{Bob2cycle}. In particular, if $\eta\in\perm_{\quiver,I}$ is non-trivial then $\Phi_1(\eta)=\theta\in \forb_{\quiver,I}$, where $\theta$  is the unique forbidden thread such that $\tar(\theta)=\tar(\eta)$ and either $\ell(\theta) = 0$ or $\ell(\theta)>0$ and the terminating arrows of $\theta$ and $\eta$ differ. Whereas if $\triv_i$ is a trivial permitted thread then in case there is an arrow $\alpha_1$ such that $\tar(\alpha_1) = i$, then $\Phi_1(\triv_i)$  is the (unique) forbidden thread whose terminating arrow is $\alpha_1$. If there is no such $\alpha_1$ then $\Phi_1(\triv_i)=\triv_i$. The bijection $\Phi_2$ acts dually (with respect to the source function). The composition
\begin{equation}\label{Eq:VerPer}
\VerPer:=\sou \circ \tar^{-1}:\verts{G} \to \verts{G}
\end{equation}
in diagram~(\ref{eq:bijections}) is a permutation of the set of vertices of $G$ such that $\OT{\alpha}$ and $\OT{\VerPer(\alpha)}$ are concatenable walks for any vertex $\alpha$. Note that the concatenation of anti-walks $\wlk=\OT{\alpha^0}\OT{\alpha^1}\cdots \OT{\alpha^a}$, with $\alpha^t:=\VerPer^t(\alpha)$ and $a:=\min\{t\geq 0 \mid \alpha^{t+1}=\alpha\}$, is a face of $(G,\leq)$ (by a face of $(G,\leq)$ we mean a face of the underlying ribbon graph). Moreover, the set $\mathcal{O}=\{\GBQ(\OT{\alpha^0}),\GBQ(\OT{\alpha^1}),\ldots, \GBQ(\OT{\alpha^a})\}$ is an orbit of the bijection $\Phi:=\Phi_1\Phi_2:\forb_{\quiver,I} \to \forb_{\quiver,I}$, and it is easy to show that if we take $$\Cdeg(i_1\cdots i_{L+1}):=\sum_{t=1}^L\deg(i_ti_{t+1})+\deg(i_{L+1}i_1),$$ where $\deg(ij)$ is as in~(\ref{eq:degree}) if $j\neq \invw{i}$ and $\deg(ij):=-1$ if $j= \invw{i}$, then using~(\ref{eq:lengthAntiWalk}) we get 
\begin{equation}\label{eq:AAGone}
\ell(\wlk)=\Cdeg(\wlk)+2|\mathcal{O}|=\sum_{\Wlk \in \mathcal{O}}\ell(\Wlk) +|\mathcal{O}|.
\end{equation}
Even though every orbit in $\forb_{\quiver,I}/\Phi$ can be found in this way, in general there might be faces of $(G,\leq)$ that are not concatenation of anti-walks. We call them \textbf{full faces} of $(G,\leq)$, cf.~Example \ref{ex:twosided} below. It can be shown that for any full face $\wlk=i_1\cdots i_{L+1}$, $L\geq 0$:
\begin{itemize}
\item $\wlk$ is reduced,
\vspace{-3pt}
\item  the corresponding walk $\GBQ(i_1\cdots i_{L+1} i_1)$ (see Lemma~\ref{lem:walkGQ}(b)) is a full cycle of $(Q,I)$,
\vspace{-3pt}
\item every full cycle of $(Q,I)$ may be found in the above way,
\vspace{-3pt}
\item  we have
\vspace{-3pt}
\begin{equation}\label{eq:AAGtwo}
\ell(\wlk)=\Cdeg(\wlk)=L+1=\ell(\GBQ(i_1\cdots i_{L+1} i_1)).
\end{equation}
\vspace{-8pt}
\end{itemize}
In particular, the gentle bound quiver $(\quiver,I)$ has finite global dimension if and only if every ordered edge of $G$ can be extended to an anti-walk, see Lemma~\ref{lem:fgd}. We call $\Cdeg(\wlk)$ the \textbf{closed degree} of a closed walk $\wlk$.

The closed degree and the length of a closed walk are invariant under cyclic permutations, and when applied to faces of $(G,\leq)$ these numbers are related to the \textbf{Avella-Alaminos Geiss invariant} (AAG invariant, for short) of the gentle bound quiver $(Q,I)$ as follows. Let $\forb'_{\quiver,I}$ be the set of arrows of $Q$ that are not subpaths of any forbidden thread (that is, those that belong to full cycles) and take $\Phi':\forb'_{\quiver,I} \to \forb'_{\quiver,I}$ as $\Phi'(\alpha_t)=\beta_s$ if $\alpha_t \beta_s \in I$. Then the AAG invariant $\varphi_{Q,I}$ of the gentle bound quiver $(Q,I)$ is given by
\begin{equation}\label{eq:AAGdef}
\varphi_{Q,I}(n,m):=|\{ \mathcal{O} \in \forb_{\quiver,I}/\Phi \cup \forb'_{\quiver,I}/\Phi' \mid (n_\mathcal{O},m_\mathcal{O})=(n,m) \}|, \qquad \text{for $n,m \geq 0$},
\end{equation}
where $(n_\mathcal{O},m_\mathcal{O}):=(|\mathcal{O}|,\sum_{\wlk \in \mathcal{O}}\ell(\wlk))$ if $\mathcal{O} \in \forb_{\quiver,I}/\Phi$, and $(n_\mathcal{O},m_\mathcal{O}):=(0,|\mathcal{O}|)$ if $\mathcal{O} \in \forb'_{\quiver,I}/\Phi'$, see details in~\cite{Bob2cycle,AAG}. Note that equivalently to $\varphi_{Q,I}$, one can consider the finite multi-set $\{(n_\CO, m_\CO) \mid \CO\in\forb_{\quiver,I}/\Phi \cup \forb'_{\quiver,I}/\Phi'\}$. As usual, we also say that $\varphi_{Q,I}$ is the AAG invariant of the gentle algebra $kQ/I$, and we use freely the notation $\varphi_{kQ/I}$.

\begin{theorem}[{{\cite[Theorem A]{AAG}}}]\label{thm:AAG}
If two gentle algebras $A$ and $A'$ are derived equivalent, then $\varphi_A(n,m)=\varphi_{A'}(n,m)$ for each $n,m\geq 0$.
\end{theorem}

By the discussion above, there is a bijection between the set of (chosen representatives of the cyclic permutation equivalence classes of) faces $\Faces(G,\leq)$ of the marked ribbon graph $(G,\leq)$ and the set of orbits $\forb_{\quiver,I}/\Phi \cup \forb'_{\quiver,I}/\Phi'$, and using the identities~(\ref{eq:AAGone}) and~(\ref{eq:AAGtwo}) we get
\begin{equation} \label{eq:AAGnew}
\varphi_{Q,I}(n,m)=\left| \left\{ \wlk \in \Faces(G,\leq) \mid \left( \frac{\ell(\wlk)-\Cdeg(\wlk)}{2},\frac{\ell(\wlk)+\Cdeg(\wlk)}{2} \right)=(n,m) \right\} \right|.
\end{equation} 
We will apply equation~(\ref{eq:AAGnew}) in Corollary~\ref{cor:OZspirit} to reformulate the derived invariance of $\varphi_{Q,I}$ in graph theoretical terms.

Note that the orbits in $\verts{G}$ of the permutation~\eqref{Eq:VerPer}  correspond to the marked points in the geometric model of gentle algebras, as found along the boundary components of the corresponding (oriented) marked surface, see~\cite{OPS18}. In Subsection~\ref{subsec:Cox} we will see how the bijection~\eqref{Eq:VerPer} relates the AAG invariant to the Coxeter polynomial of the gentle bound quiver $(\quiver,I)$.

\subsection{Non-perfect complexes}\label{subs:nonperf}

Fix a gentle bound quiver $(Q,I)$ with marked ribbon graph $(G,\leq)$ and set $A:=kQ/I$. Recall that all indecomposable perfect complexes over $A$  are given by reduced walks and belts of $(G,\leq)$, see Proposition \ref{pro:walkBelt}. Here we provide  analogous description of indecomposable non-perfect complexes, that is,  objects of the class $\ind(\Db{\La})\setminus \ind(\per\,\La)$. This approach may be viewed as a gentle algebra analogue of the techniques of {\it Green walks} in Brauer (ribbon) graphs, see \cite{GreenWalk} and \cite{Schroll.Brauer}. 
Note that via the triangle equivalence $\Db{\La}\cong\Kmb(\proj\,\La)$ indecomposable non-perfect complexes correspond to $\ind(\Kmb(\proj\,\La))\setminus \ind(\Kb(\proj\,\La))$. We apply the classification of \cite{BekMer} following conventions established in \cite{ALP}.

\begin{definition}
Let  $\wlk=i_1i_2\cdots i_Li_{L+1}$ be a reduced walk of $(G,\leq)$, with $L\geq 1$.
\begin{itemize}
\vspace{-3pt}
\item The walk  $\wlk$  is called \textbf{left resolvable} if $i_1$ belongs to a full face of $(G,\leq)$ and $\deg(i_1\cdots i_t)\geq 0$ for all $t=1,\ldots, L+1$.
\vspace{-4pt}
\item Left resolvable walk $\wlk$ is \textbf{primitive} if there is no $s\in\{2,\ldots, L\}$ such that $i_1\cdots i_s$ is a subwalk of a power of a full face, and
$i_{s}\cdots i_{L+1}$ is left resolvable.
\end{itemize}
\end{definition}
\vspace{-3pt}
Fix $m\in\ZZ$ and a primitive left resolvable walk $\wlk=i_1\cdots i_{L+1}$. Then there is a unique 
face $\phi=j_1\cdots j_K$ in $(G,\leq)$  such that $j_1=i_1$ and the composition $\phi\wlk$ is defined. Consider the  complex $\stringcpx{m,-,\wlk}$ given by the following left-infinite unfolded diagram
\begin{equation}\label{eq:linf}
\xymatrix@C=1.7pc{
\cdots P_{j_1} \ar@{->}[r]^-{\ppath{j_1}{j_2}} & P_{j_2} \ar@{->}[r]^-{\ppath{j_2}{j_3}} & \cdots \ar@{->}[r]^-{\ppath{j_{K-1}}{j_K}} & P_{j_K} \ar@{->}[r]^-{\ppath{j_K}{j_{1}}} &
P_{j_1} \ar@{->}[r]^-{\ppath{j_1}{j_2}} & P_{j_2} \ar@{->}[r]^-{\ppath{j_2}{j_3}} & \cdots \ar@{->}[r]^-{\ppath{j_{K-1}}{j_K}} & P_{j_K} \ar@{->}[r]^-{\ppath{j_K}{i_{1}}} &
P_{i_1} \ar@{->}[r]^-{\ppath{i_1}{i_2}} & P_{i_2} \ar@{-}[r]^-{\ppath{i_2}{i_3}} & \cdots \ar@{-}[r]^-{\ppath{i_{L-1}}{i_L}} & P_{i_L} \ar@{-}[r]^-{\ppath{i_L}{i_{L+1}}} & P_{i_{L+1}},}
\end{equation}
with $P_{i_1}$  positioned at cohomological degree $m$. Observe that the homotopy letters $\ppath{j_t}{j_{t+1}}$ and $\ppath{j_K}{j_{1}}$ are direct of length 1 by the properties of a full face. One may view \eqref{eq:linf} as the diagram of the infinite word $\cdots\phi\phi\phi\wlk$, analogously as for finite walks in Subsection \ref{subsec:perfcpx}. We consider the obvious dual definitions of (\textbf{primitive}) \textbf{right resolvable} walk and the associated complex $\stringcpx{m,+,\wlk}$ given by right-infinite unfolded diagram dual to \eqref{eq:linf}. A reduced walk $\wlk$ of $(G,\leq)$ is  called (\textbf{primitive}) \textbf{two-sided resolvable} if it is both (primitive) left resolvable and (primitive) right resolvable  and we consider the associated complex $\stringcpx{m,\mp,\wlk}$ given by two-sided infinite unfolded diagram (compare with~\cite[2.3]{ALP}). In each of the cases we position $P_{i_1}$  at cohomological degree $m$. Denote by  $\leftres(G,\leq)$ $($resp.~$\rightres(G,\leq)$, $\twores(G,\leq):=\leftres(G,\leq)\cap \rightres(G,\leq))$  the set of all primitive left
$($resp.~right, two-sided$)$ resolvable walks of $(G,\leq)$.

\begin{example}\label{ex:twosided} 
Consider the following gentle bound quiver and the associated marked ribbon graph
(we draw its edges together with a fixed order, cf.~\eqref{eq:ordedg}):
$$(Q,I):\ \ {\xymatrix@C=2.2pc{\mathmiddlescript{1} \ar@/^7pt/[r]^-{\alpha_2}="a2"
& \mathmiddlescript{2} \ar@/^7pt/[l]^-{\beta_1}="b1"
\ar@{.}@/^3pt/"a2"!<7pt,-6pt>;"b1"!<7pt,7pt>
\ar@{.}@/^7pt/"b1"!<-2pt,4pt>;"a2"!<-3pt,-3pt>
\ar@/^7pt/[r]^-{\alpha_1}="a1"
& \mathmiddlescript{3} \ar@/^7pt/[l]^-{\beta_2}="b2"
\ar@{.}@/^3pt/"a1"!<7pt,-6pt>;"b2"!<7pt,7pt>
\ar@{.}@/^7pt/"b2"!<-2pt,4pt>;"a1"!<-3pt,-3pt>}}
\qquad (G,\leq):\ \ {\xymatrix@C=3pc{ {\vBlock{\alpha}{321}} \ar@{|->}[r]^-{2} \ar@{|->}@/_10pt/[r]^-{1} \ar@{|->}@/^10pt/[r]^-{3} & {\vBlock{\beta}{123}} }}
$$
The graph $(G,\leq)$ has three faces: $\invw{3}1$ (being the concatenation of anti-walks $\OT{\alpha}=\invw{3}$ and $\OT{\beta}=1$) and two full faces $2\invw{1}$ and $3\invw{2}$. The walk $\wlk:=\invw{1}_{\succ}2_{\pred}\invw{3}$ is primitive two-sided resolvable and the associated two-sided infinite unfolded diagram with the corresponding complex $\stringcpx{m,\mp,\wlk}$ for fixed $m\in\ZZ$ look as follows
$$
\xymatrix@C=1.5pc{
\cdots P_{1} \ar@{->}[r]^-{\beta_1} & P_{2} \ar@{->}[r]^-{\alpha_2} &
P_1 \ar@{->}[r]^-{\beta_1} & P_{2} \ar@{->}[r]^-{\alpha_2} &
P_{1} \ar@{->}[r]^-{\beta_1} & P_{2} &
P_3 \ar@{->}[l]_-{\alpha_1} &
P_2 \ar@{->}[l]_-{\beta_2} &
P_3 \ar@{->}[l]_-{\alpha_1} &
P_2 \ar@{->}[l]_-{\beta_2} &
P_3 \ar@{->}[l]_-{\alpha_1} \cdots
}
$$
$$
\stringcpx{m,\mp,\wlk}: \xymatrix@!0@C=50pt@R=10pt{
 \hspace{-2cm}&P_1 & P_2 & P_1 & P_2 & P_1 & &\\
\cdots \hspace{-2cm}&\oplus \ar[r]^-{\left[ \begin{smallmatrix} \beta_1&0\\0&\alpha_1\end{smallmatrix} \right]} & \oplus \ar[r]^-{\left[ \begin{smallmatrix} \alpha_2&0\\0&\beta_2\end{smallmatrix} \right]}& \oplus \ar[r]^-{\left[ \begin{smallmatrix} \beta_1&0\\0&\alpha_1\end{smallmatrix} \right]} & \oplus \ar[r]^-{\left[ \begin{smallmatrix} \alpha_2&0\\0&\beta_2\end{smallmatrix} \right]}& \oplus \ar[r]^-{\left[ \begin{smallmatrix} \beta_1&\alpha_1\end{smallmatrix} \right]}& P_2\ar[r] &0\ \cdots\\
\hspace{-2cm}&P_3 & P_2 & P_3 & P_2 & P_3 & &\\
 }
$$
where the rightmost non-zero term $P_2$ in $\stringcpx{m,\mp,\wlk}$ is positioned at cohomological degree $m+1$. Note that $\stringcpx{m,\mp,\wlk}$ is the  projective resolution of (the stalk complex of) the simple module $S_2$. The walk $\wlk$ treated as (primitive) left resolvable gives rise to the following complex
$$
\stringcpx{m,-,\wlk}: \xymatrix@!0@C=50pt@R=10pt{
 \hspace{-2cm}& &  &  &  & P_1 & &\\
\cdots \hspace{-2cm}&P_1 \ar[r]^-{\beta_1} & P_2 \ar[r]^-{ \alpha_2}& P_1 \ar[r]^-{\beta_1} & P_2 \ar[r]^-{\left[ \begin{smallmatrix} \alpha_2\\0\end{smallmatrix} \right]}& \oplus \ar[r]^-{\left[ \begin{smallmatrix} \beta_1&\alpha_1\end{smallmatrix} \right]}& P_2\ar[r] &0\ \cdots\\
\hspace{-2cm}& &  &  &  & P_3 & &\\
 }
$$
Observe that here $\stringcpx{m,-,\wlk}\cong\stringcpx{m,+,\invw{\wlk}}$, since $\deg(\wlk)=0$.
\end{example}

The classification of indecomposable non-perfect complexes provided by Bekkert-Merklen \cite{BekMer} can be reformulated in terms of walks in the marked ribbon graph as follows. We give a sketch of the proof leaving the details to the reader.
\begin{proposition}\label{prop:nonper}
Let $(Q,I)$ be a gentle bound quiver and $(G,\leq)$ the associated marked ribbon graph. Then for every $\wlk\in \leftres(G,\leq\nolinebreak)$ $($resp.~$\wlk\in \rightres(G,\leq)$, $\wlk\in \twores(G,\leq))$ the complex
$\stringcpx{m,\epsilon,\wlk}$ for $\epsilon=-$ $($resp.~$\epsilon=+$, $\epsilon=\mp)$ is an indecomposable object of $\Kmb(\proj\,\La)$ for $A=kQ/I$. Moreover, the construction of $\stringcpx{m,\epsilon,\wlk}$ induces the following bijection with the set of the isomorphism classes of indecomposable non-perfect complexes over $A$
\begin{equation}\label{eq:nonper}
\xymatrix@C=1pc{[\ZZ\times\{-\}\times \leftres(G,\leq)]\  \cup \ ([\Z \times\{\mp\}\times \twores(G,\leq)]/\sim) \ \ \ar@{<->}[rr]^-{1-1} \ \ & & \ \  \Ind(\Kmb(\proj\,\La))\setminus \Ind(\Kb(\proj\,\La)),}
\end{equation}
where $(m,\mp,\wlk) \sim (m',\mp,\wlk')$ when $(m,\wlk)=(m',\wlk')$, or $m'=m+\deg(\wlk)$ and $\wlk'=\invw{\wlk}$.
\end{proposition}

\begin{proof} 
By applying Lemma \ref{lem:walkGQ} and the properties of a full face from Subsection \ref{subsec:antiwalks}  we check directly that a reduced walk $\wlk$ of $(G,\leq)$ is primitive left (resp.~right, two-sided) resolvable if and only if the walk $\GBQ(\wlk)$ of $Q$ is a primitive  left (resp.~right, two-sided) resolvable homotopy string in the sense of \cite[Definition 2.6]{ALP}. Moreover,
for every $\wlk\in \leftres(G,\leq\nolinebreak)$ $($resp.~$\wlk\in \rightres(G,\leq)$, $\wlk\in \twores(G,\leq))$ the complex
$\stringcpx{m,\epsilon,\wlk}$ for $\epsilon=-$ $($resp.~$\epsilon=+$, $\epsilon=\mp)$ is precisely the complex associated with the (respectively shifted) \lq\lq infinite homotopy string'' denoted by ${}^\infty\GBQ(\wlk)$ (resp.~$\GBQ(\wlk)^\infty$, ${}^\infty\GBQ(\wlk)^\infty$) in \cite[2.3]{ALP}.
Thus the hypothesis follows from \cite[Theorem 3]{BekMer}, see also \cite[Theorem 2.9]{ALP}. With regard to the parametrization \eqref{eq:nonper} we only note that $\wlk$ is right resolvable if and only if $\wlk^{-1}$ is left resolvable thus it is enough to take primitive left resolvable walks  as representatives, cf.~Example \ref{ex:twosided}. Moreover, it easily follows by construction that $\stringcpx{m,\mp,\wlk}\cong \stringcpx{m+\deg(\wlk),\mp,\wlk^{-1}}$, for each $m\in\ZZ$ and $\wlk\in\twores(G,\leq)$. 
\end{proof}

\begin{remark}\label{rem:mor}As seen in Propositions~\ref{pro:walkBelt} and~\ref{prop:nonper}, the reduced walks and belts of a marked ribbon graph $(G,\leq)$ are compact expressions of the indecomposable complexes in the bounded derived category of the gentle algebra $A=kQ/I$ with bound quiver $(Q,I)=\GBQ(G,\leq)$, cf.~Lemma~\ref{thm:IncGentle}. A description of the morphisms between two such indecomposable objects was given by Arnesen-Laking-Pauksztello~\cite{ALP} for algebraically closed fields. Translating such description to our graphical setting yields certain maxi\-mally common subwalks, or ``overlaps'', which simplify the string approach of~\cite{ALP} by unifying so-called graph/quasi-graph maps and single/double maps into non-trivial and trivial ``overlaps'', respectively. We believe, however, that a precise formulation and necessary details for these claims deserve a separate paper, especially in the context of tilting and silting theory. For the main goals of the present paper it is more important to provide a description of Auslander-Reiten triangles, see Proposition \ref{thm:ARtriangle} below.
\end{remark}

\section{Derived invariants}\label{sec:derinv}

In this section we formulate and prove the first main result of the paper, Theorem \ref{thm:mrgvsCar} (Theorem A from the introduction) and provide several applications, see Subsection \ref{subsec:applA}. We start by recalling basic notions and facts related Grothendieck groups, homological bilinear and quadratic forms  and other derived invariants.

\subsection{Grothendieck groups and the Euler form}\label{subsec:eulf}

Let  $\La$ be an arbitrary (finite dimensional) $k$-algebra and let $\{S_i\}_{i=1}^n$, $\{P_i\}_{i=1}^n$ and $\{E_i\}_{i=1}^n$ be complete sets of representatives of isomorphism classes of simple, indecomposable projective and indecomposable injective left $A$-modules respectively, with $P_i$ a projective cover of $S_i$ and $E_i$ and injective envelope of $S_i$ (cf. Section~\ref{sec:combgentle}).  By $K_0(\mod{A})$ and $K_0(\Db{\La})$ we denote the Grothendieck groups of the category $\mod{A}$ of finite dimensional $\La$-modules and the derived category $\Db{\La}=\Db{\mod{\La}}$ of $\mod{A}$, defined by means of short exact sequences and the distinguished triangles, respectively (see~\cite{ASS} and~\cite{HDb2} for the details). Recall that there is an isomorphism  $K_0(\mod{A})\overset{\cong}{\rightarrow}K_0(\Db{\La})$ of abelian groups induced by the canonical embedding of $\mod{A}$ into $\Db{\La}$. The inverse mapping is given by
\begin{equation}\label{eq:alternating}[X^\bullet] \ \,\mapsto \ \, \sum_{i\in\ZZ}(-1)^i[X^i],
\end{equation} where $[X^\bullet]$ denotes the class in $K_0(\Db{\La})$ of a complex $X^\bullet$ from $\Db{\La}$ and $[X^i]$, the class in $K_0(\mod{A})$ of the $\La$-module $X^i$ of $X^\bullet$ of cohomological degree $i$. The groups $K_0(\mod{A})$ and $K_0(\Db{\La})$ are freely generated by the corresponding classes $\{[S_i]\}_{i=1}^n$ of simple modules. Given an $A$-module $X$ (resp.~a complex $X^\bullet$), the coordinate vector in $\ZZ^n$ of the class $[X]\in K_0(\mod{A})$ (resp.~$[X^\bullet]\in K_0(\Db{\La})$) in the basis $\{[S_i]\}_{i=1}^n$ is  called the \textbf{dimension vector} of $X$ (resp.~of $X^\bullet$) and it is denoted by $\vdim{X}\in\ZZ^n$ (resp.~$\vdim{X^\bullet}\in\ZZ^n$). Observe that $\vdim{X}$ belongs to $\NN^n$, while $\vdim{X^\bullet}$ may have negative coefficients, cf.~\eqref{eq:alternating} and~\cite{HDb2}. 

 Similarly, we consider the Grothendieck groups $K_0(\proj\,\La)$ and $K_0(\per\,\La)$   of the categories $\proj\,\La$ of projective $\La$-modules and $\per\,\La$ of perfect complexes, respectively (cf. Section~\ref{sec:combgentle}). Also in this case there is a natural isomorphism  $K_0(\proj\,\La)\overset{\cong}{\rightarrow}K_0(\per\,\La)$  induced by the canonical embedding of $\proj\,\La$ into $\Kb(\proj\,\La)$. The groups $K_0(\proj\,\La)\cong K_0(\per\,\La)$ are freely generated by the corresponding classes $\{[P_i]\}_{i}$ of indecomposable projective modules. Given an $A$-module $X$ in $\proj\,\La$ (resp.~a perfect complex $X^\bullet$), the coordinate vector in $\ZZ^n$ of the class $[X]\in K_0(\proj\,\La)$ (resp.~$[X^\bullet]\in K_0(\per\,\La)$) in the basis $\{[P_i]\}_{i=1}^n$ is   denoted by $\vdimp{X}\in\ZZ^n$ (resp.~$\vdimp{X^\bullet}\in\ZZ^n$).

Recall that the mapping $P\mapsto P/ \rad(P)$ induces an isomorphism \begin{equation}\label{eq:sillyiso}
K_0(\proj\,\La)\overset{\cong}{\rightarrow}K_0(\mod{\La})
\end{equation} which sends the basis elements $[P_i]$ to $[S_i]$ for each $i$. However, recall that the categories $\per\,A\cong \Kb(\proj\,\La)$ and $\Db{\La}$ are not equivalent in general.

In what follows we use the bilinear and quadratic forms, which are  usually defined on $K_0(\mod{A})$ for algebras of finite global dimension (cf.~\cite{ASS, HDb2}). For our purposes it is useful to consider the following slightly more general notions (cf.~\cite{Ladkani.2017.cateulf,BoSk}).

\begin{definition}\label{def:euler}
Let $A$ be a $k$-algebra. The {\bf homological bilinear form} $\langle -,- \rangle_A: K_0(\per\,A)^2\to\ZZ$ of $A$ is given by
\[
\langle [X^{\bullet}],[Y^{\bullet}] \rangle_A=\sum_{i \in \Z} (-1)^i \dim_k\Hom_{\Db{A}}(X^{\bullet},Y^{\bullet}[i]).
\]
Its associated quadratic form $q_A(x):=\langle x,x \rangle_A$ is called the {\bf homological quadratic form} of $A$, or simply {\bf Euler form}.
\end{definition}

\begin{remark}\label{rem:euler}
The Gram matrix of the bilinear form $\langle-,- \rangle_A$ in the basis $\projbas=\{[P_i]\}_{i=1}^n$ of indecomposable projectives is the \textbf{Cartan matrix} $\Car_A:=M_{\langle -,- \rangle_A,\projbas}=[\dim_k\Hom_{A}(P_i,P_j)]_{i,j}$ of a $k$-algebra $A$ (cf.~\eqref{eq:Grambil}). Note that for bound quiver algebras we have $\Car_A\bas_i=\vdim{P_i}$ and $\Car_A^\tr\bas_i=\vdim{E_i}$ for each $i=1,\ldots,n$, see for instance~\cite{Barot2014IntroductionTT}. 
 If two algebras $A$ and $A'$ are derived equivalent then their bilinear forms, resp.~Euler forms are equivalent; in particular, the rank $\Rnk(q_A)$, the corank $\CRnk(q_A)$ and the determinant of the Cartan matrix are derived equivalence invariants, see~\cite{HDb2} and~\cite[Proposition 1.5]{BoSk}. Moreover, if $A$ has finite global dimension $\gldim(A)$, then it follows from~\cite{HDb2} that:
\begin{enumerate}[label={\textnormal{(\alph*)}},topsep=4px,parsep=0px]
\item There is a triangle equivalence of categories $\Kb(\proj\,\La)$ and $\Db{\La}$, that is, $\per\,\La=\Db{\La}$. In particular, the identity mapping $P\mapsto P$ induces an isomorphism of groups $K_0(\proj\,\La)\overset{\cong}{\rightarrow}K_0(\mod{\La})$ (alternative to~\eqref{eq:sillyiso}). 
\item The Cartan matrix $\Car=\Car_A$ is $\ZZ$-invertible and it is the matrix of the isomorphism $K_0(\proj\,\La)\overset{\cong}{\rightarrow}K_0(\mod{\La})$ from (a) in the bases of indecomposable projectives and simples, respectively.
\item The (perfect) bilinear form $\langle -,- \rangle_\La$ is the (classical) Euler characteristic of $A$ and its Gram matrix in the basis of simple $A$-modules is $(\Car^{-1})^\tr\cdot\Car\cdot \Car^{-1}=(\Car^{-1})^\tr$, cf.~\cite[III.3]{ASS}. 
\end{enumerate}
\end{remark}

\begin{lemma}\label{lem:eulfconnected}
Let $(Q,I)$ be an admissible bound quiver with $Q$ connected. Then the Euler form $q_A:K_0(\per\, A)\to\ZZ$ of the algebra $A:=kQ/I$ in the basis of indecomposable projectives is connected.
\end{lemma}

\begin{proof} 
Recall that the indecomposable projective $A$-modules are of the form $P_i:=A \triv_i$, for $i\in Q_0=\{1,\ldots,n\}$ (cf.~the beginning of Section~\ref{sec:combgentle}).
Observe that for each arrow  $\alpha\in Q_1$ with $j:=\tar(\alpha)$ and $i:=\sou(\alpha)$ there is a non-zero $A$-homomorphism $\bar{\alpha}:P_{j}\to P_{i}$  induced by the right composition by $\alpha$ (cf.~\cite[p.~644]{Bob}). This means that the corresponding coefficient of $q=q_A$ in the basis of indecomposable projectives $q_{ji}=q_{ij}=q([P_i],[P_j])=\langle [P_i],[P_j] \rangle_A+\langle [P_j],[P_i] \rangle_A=\dim_k\Hom_{A}(P_i,P_j)+\dim_k\Hom_{A}(P_j,P_i)$ is non-zero (cf.~\eqref{eq:q}). Now take arbitrary $i'\neq j'$ in $\{1,\ldots,n\}$. Since $Q$ is connected then there exists a walk $\wlk=\aoi_{1}\cdots \aoi_\ell$ in $Q$ with $\aoi_t \in \quiver_1\cup\invw{\quiver}_1$ for $t=1,\ldots,\ell \geq 1$ such that $\sou(\wlk)=i'$ and $\tar(\wlk)=j'$. Thus by the above observations the walk $\wlk$ induces the sequence of non-zero coefficients $q_{i_{t-1}i_t}\neq 0$ for $i=1,\ldots,\ell$, where $i_t=\sou(\aoi_t)$, $i_\ell=i'$ and $i_0=j'$. This shows that $q_A$ in the basis of projectives is connected, see Subsection~\ref{subsec:qf}.
\end{proof}
It should be noted that the Euler form $q_A$ of the algebra $A=kQ/I$ of a connected admissible bound quiver $(Q,I)$ may be disconnected in the basis of simple $A$-modules.

\subsection{Main result}\label{subsec:ThmA}

To formulate the  main result we need to establish some notation.
Fix a gentle bound quiver $(\quiver,I)$.  Given a closed walk $\wlk=\aoi_{1}\cdots \aoi_\ell$ in $Q$ with $\aoi_t \in \quiver_1\cup\invw{\quiver}_1$ for $t=1,\ldots,\ell \geq 1$, assume that $\wlk$ has no {repeated edges}, that is $\aoi_i\neq \aoi_j$ and $\aoi_i\neq \invw{\aoi}_j$ for all $i\neq j$. In particular, this means that $\wlk$ is reduced. By $\nrp{\wlk}$ we denote the number of monomial relations of length two $\rho\in I$ that are either a subpath of $\wlk$, or of the form $\rho=\aoi_\ell\aoi_1$ with $\aoi_1,\aoi_\ell\in \quiver_1$ (we call them \textbf{clockwise relations} on $\wlk$, cf.~\cite{AS}). Similarly, denote by $\nrm{\wlk}$ the number of monomial relations of length two $\rho\in I$ that are either a subpath of $\invw{\wlk}$ or of the form $\rho=\invw{\aoi}_1\invw{\aoi}_\ell$ with $\invw{\aoi}_1,\invw{\aoi}_\ell\in \quiver_1$ (we call them \textbf{counterclockwise relations} on $\wlk$). We set $\nrr{\wlk}:=\nrp{\wlk}+\nrm{\wlk}$ and we say that $\wlk$ contains $\nrr{\wlk}$ relations. Note that  $\nrr{\wlk}=\nrr{\invw{\wlk}}=\nrr{\wlk''\wlk'}$   for each decomposition $\wlk=\wlk'\wlk''$ for subwalks $\wlk'$, $\wlk''$. 

\begin{definition}\label{def:mcc}
Let $(\quiver,I)$ be a gentle bound quiver and $\La:=k(\quiver,I)=kQ/I$ the associated gentle algebra.  We say that $(\quiver,I)$ (or  $\La$) satisfies the {\bf  multi-clock condition}, if for each non-trivial closed walk $\cyc$ in $Q$ without repeated edges, the number $\nrr{\cyc}$ of relations in $\wlk$
is even. If $\La$ satisfies the multi-clock condition then we set $\gcc_\La:=1$, otherwise we set $\gcc_\La:=0$.
\end{definition}

\begin{remark}\label{rem:tcycle}
 We say that an admissible bound quiver $(Q,I)$ or the associated algebra $kQ/I$ is $t$-cycle if so is the underlying graph of $Q$ (we set $\cycles(Q,I)=\cycles(Q)=t$, cf.~\eqref{eq:tcycle}).
If $(\quiver,I)$ is a 1-cycle gentle bound quiver satisfying the {\em clock condition} in the sense of~\cite{AS} (that is, $\nrp{\wlk}=\nrm{\wlk}$ for the unique up to inverse and rotation non-trivial closed walk $\wlk$ of $Q$ without repeated edges)
then $(\quiver,I)$ satisfies the multi-clock condition. Obviously, each 0-cycle gentle bound quiver also satisfies the multi-clock condition, cf.~\cite{AH}.
\end{remark}

The following main result of this section is a more precise version of Theorem A. We use the notation and the terminology of Section~\ref{sec:prelim} and Definitions~\ref{def:BoundQuiverInc} and~\ref{rem:forb}. Recall that a graph $G$ is \textbf{bipartite} if there is a partition $\verts{G}=V' \cup V''$ such that the end-vertices of each edge of $G$ belong to different parts of the partition. Alternatively, a graph is bipartite if and only if all its closed walks have even length, if and only if all its cycles have even length (a cycle is a non-trivial reduced closed walk in which only the first and last vertices are equal). 

\begin{theorem}\label{thm:mrgvsCar}
Let $(Q,I)$ be a gentle bound quiver,  $(G,\leq)=\GPC(\quiver,I)$  the associated marked ribbon graph and $(\fmrgs{G},\fmrgs{\leq},\fmrgs{\sigma})=\GFC(\quiver,I)$ the associated forbidden marked ribbon graph $($the latter in case $\gldim(Q,I)<\infty)$. Then the Cartan matrix $\Car_\La$ of  the gentle algebra $\La:=kQ/I$ satisfies
\begin{equation}\label{eq:mrbvsCar}
\begin{array}{rll}\Car_\La+\Car_\La^\tr&=&\Inc(G)\Inc(G)^\tr\quad\text{and}\smallskip\\ \Car_\La^{-\tr}+\Car_\La^{-1}&=&\Inc(\fmrgs{G},\fmrgs{\sigma})\Inc(\fmrgs{G},\fmrgs{\sigma})^\tr\ \ (\text{if}\ \  \gldim(\La)<\infty).
\end{array}
\end{equation}
 In particular, the Euler form $q_\La$ of $\La$ is non-negative. Moreover:
\begin{enumerate}[label={\textnormal{(\alph*)}},topsep=3px,parsep=0px]
\item The following conditions are equivalent.
\begin{enumerate}[label={\textnormal{(\roman*)}},topsep=3px,parsep=0px]
\item[{\rm (a1)}] $(Q,I)$ satisfies the {multi-clock condition}.
\item[{\rm (a2)}] There exist  string functions $S, T:Q_1\to\{\pm 1\}$ such that $S \equiv -T$,
\item[{\rm (a3)}] The graph $G$ is bipartite.
\item[{\rm (a4)}] The bidirected graph $(\fmrgs{G},\fmrgs{\sigma})$ is balanced.
\end{enumerate}
$($The condition {\rm (a4)} is well-defined only in case  $\gldim(\La)<\infty)$.
\item The rank and corank of the Euler form are given by the formulae
\begin{equation}\label{eq:mrbvsCarcrk}
\begin{array}{l}
r=r_A:=\Rnk(q_\La) \ \ =|\verts{G}|-\gcc_\La\ =\ 2|Q_0|-|Q_1|-\gcc_\La\ =\ |\verts{\fmrgs{G}}|-\gcc_\La,\smallskip\\
c=c_A:=\CRnk(q_\La)=|\edgs{G}|-|\verts{G}|+\gcc_\La\ =\ |Q_1|-|Q_0|+\gcc_\La\ =\ |\edgs{\fmrgs{G}}|-|\verts{\fmrgs{G}}|+\gcc_\La
\end{array}
\end{equation}
    $($the last equalities in both rows hold only in case  $\gldim(\La)<\infty)$. In particular, the binary parameter $\gcc_\La$ of Definition~\ref{def:mcc} is a derived equivalence invariant.  

\item The Euler form $q=q_\La$ in the basis of indecomposable projectives $\projbas:=\{[P_i]\}_{i\in Q_0}$ is connected, Cox-regular and:
\begin{itemize}
\item[{\rm (c1)}] The  diagonal coefficients of $q$ are given by:
\begin{equation}\label{eq:mrbvsCarqisproj}
q_{i}=q_i^\projbas = \left\{
\begin{array}{l l}
2, & \text{if there is a permitted path self-crossing at $i$}, \\
1, & \text{otherwise}, \\
\end{array} \right.
\end{equation}
for each vertex $i\in Q_0$. In particular, $q$ is a unit form if and only if  $(Q,I)$ has no self-crossing permitted path.
\item[{\rm (c2)}] If $q$ is a unit form then
\begin{equation}\label{eq:mrbvsCarqDynproj}
\Dyn(q)=\left\{\begin{array}{ll}
\DD_{r},&\text{if }\  \gcc_\La=0\   \text{ and }\  2|Q_0|-|Q_1|\geq 4,\\
\AA_{r},&\text{otherwise}.\\
\end{array}
\right.
\end{equation}

\smallskip

 \item[{\rm (c3)}] If $q$ is not unitary then  $\gcc_\La=0$ and
 \begin{equation}\label{eq:mrbvsCarqCproj}
\left\{\begin{array}{ll}
\Dyn(q)=\CC_r,&\text{if } \  2|Q_0|-|Q_1|\geq 2,\\
\Dyn(\frac{1}{2}q)=\AA_1,&\text{if } \  2|Q_0|-|Q_1|= 1,\\
\end{array}
\right.
\end{equation}
where $\frac{1}{2}q$ in the latter denotes the $($unit$)$ form given by $x\mapsto \frac{1}{2}q(x)$.
\end{itemize}

\item If $\gldim(A)<\infty$ then the Euler form $q=q_\La$ in the basis of simples $\simplebas:=\{[S_i]\}_{i\in Q_0}$ is semi-Cox-regular and:
\begin{itemize}
\item[{\rm (d1)}] The  diagonal coefficients of $q$ are given by:
\begin{equation}\label{eq:mrbvsCarqis}
q_{i}=q_i^\simplebas = \left\{
\begin{array}{l l}
0, & \text{if there is a $($unique$)$ almost-full cycle  at $i$ of odd length}, \\
2, & \text{if there is a $($unique$)$ almost-full cycle  at $i$ of even length}, \\
1, & \text{otherwise}, \\
\end{array} \right.
\end{equation}
for each vertex $i\in Q_0$. In particular, $q$ is a unit form if and only if  $(Q,I)$ has no
almost-full cycles. Moreover,
\begin{itemize}
\item[$\bullet$] if $q$ is connected then there is no almost-full cycle of odd length,
\item[$\bullet$] if there is no almost-full cycle of odd length and $2|Q_0|-|Q_1|\geq 3$ then $q$ is connected.
\end{itemize}
\item[{\rm (d2)}] If $q$ is connected and unitary then the formula~\eqref{eq:mrbvsCarqDynproj} of $($c$2)$ for the Dynkin type  holds.
 \item[{\rm (d3)}] If $q$ is connected and not unitary then  $\gcc_\La=0$ and   the formula~\eqref{eq:mrbvsCarqCproj} of $($c$3)$ for the Dynkin type  holds.
\end{itemize}
\end{enumerate}
\end{theorem}

Observe that even though the Dynkin types of the Euler form $q$ in both bases $P$ and $S$ are described by the same formula, it may happen that $q$ is unitary in the basis $P$ and non-unitary in $S$ and vice-versa. In particular, it may happen that $q$ in the basis $P$  has Dynkin type $\DD$ while in the basis $S$ it has Dynkin type $\CC$ etc., cf.~comments (D2) in Subsection~\ref{subsec:qf}. Moreover, it should be noted that the Euler form of an admissible bound quiver algebra in general does not need to be (semi-)Cox-regular and it is quite surprising that this property is fulfilled for gentle algebras  (see claims (c) and (d)). Cox-regularity of the Euler form of a gentle algebra seems to be a reason of some of the good properties of its roots discussed in Section~\ref{sec:roots}.

The proof of the theorem will be given in Subsection~\ref{subsec:proofA}. In the next subsection we discuss some preparatory  facts required for the proof.


\subsection{Incidence quadratic forms}\label{subsec:incqf}

Let  $(G,\sigma)$ be a bidirected graph.
In~\cite{JM} we studied the following integral quadratic form associated with $(G,\sigma)$.
\begin{definition}\label{def:incqf}
The \textbf{incidence (quadratic) form} of $(G,\sigma)$ is the mapping $q=q_{G,\sigma}:\Z^{\edgs{G}} \to \Z$ defined as
\begin{equation}\label{eq:qinc}
q(x):=\frac{1}{2}x^\tr \Incsh\Incsh^{\tr}x=\frac{1}{2}||\Incsh^{\tr}x||^{2},
\end{equation}
for $x \in \Z^{\edgs{G}}$, where $\Incsh=\Inc(G,\sigma)\in\M_{\edgs{G},\verts{G}}(\ZZ)$ is the incidence matrix of $(G,\sigma)$ as in Definition~\ref{def:bidi}.
\end{definition}
We also consider  the incidence form of a graph $G$, given by  $q_G:=q_{G,\sigma^{+}}$,  where $\sigma^{+}$ is the constant bidirection $\sigma^+\equiv +1$. Recall that the incidence matrix of $G$ is defined as $\Inc(G)=\Inc(G,\sigma^+)$ (cf. Definition~\ref{def:bidi}).

\begin{lemma}\label{lem:Null} Let $G$ $($resp.~$(G,\sigma))$ be a connected graph $($resp.~bidirected graph) and let $\balanced_G:=\Null(\Inc(G))$ $($resp.~$\balanced_{G,\sigma}:=\Null(\Inc(G,\sigma)))$ be the corank of the incidence matrix of $G$ $($resp.~of $(G,\sigma))$.  Then
\begin{enumerate}[label={\textnormal{(\alph*)}},topsep=4px,parsep=0px]
\item $\balanced_{G,\sigma} \in \{0,1\}$ and  $\balanced_{G,\sigma}=1$ if and only if $(G,\sigma)$ is balanced,
\item $\balanced_G \in \{0,1\}$ and  $\balanced_G=1$ if and only if $G$ is bipartite.
 \end{enumerate}
\end{lemma}
\begin{proof} 
Claim (a) follows from~\cite[Proposition 3.12]{JM}. To show (b) note that all edges of $G$ are negative with respect to the constant bidirection $\sigma^+$ (cf. \ref{subsec:bdg}). Therefore the bidirected graph $(G,\sigma^+)$ is balanced if and only if all closed walks in $G$ are of even length. This holds precisely when $G$  is bipartite. Thus (b) follows from (a).
\end{proof}

One of the main results of~\cite{JM} is a description of incidence quadratic forms that we summarize as follows.

\begin{theorem}\label{thm:JM1}
Let $(G,\sigma)$ be a bidirected graph. Then the incidence  form $q_{{G,\sigma}}:\Z^{\edgs{G}} \to \Z$ is a semi-Cox-regular non-negative integral quadratic form of rank $\Rnk(q_{{G,\sigma}})=|\verts{G}|-\balanced_{{G,\sigma}}$ and corank $\CRnk(q_{{G,\sigma}})=|\edgs{G}|-|\verts{G}|+\balanced_{{G,\sigma}}$.
The following additional properties are expressed in terms of the canonical basis.
\begin{enumerate}[label={\textnormal{(\alph*)}},topsep=4px,parsep=0px]

 \item A connected unit  form $q:\Z^n \to \Z$ is non-negative of Dynkin type $\AA$ or $\DD$ if and only if $q=q_{G',\sigma'}$ for some $($connected$)$ bidirected graph $(G',\sigma')$ without loops. Moreover, in this case $\Dyn(q)=\DD$ if and only if $|\verts{G'}| \geq 4$ and ${(G',\sigma')}$ is unbalanced.

 \item A connected irreducible integral quadratic form $q:\Z^n \to \Z$ is non-negative of Dynkin type $\CC$ if and only if $q=q_{G',\sigma'}$ for some $($connected$)$ bidirected graph $(G',\sigma')$ having a bidirected loop.
 \end{enumerate}
\end{theorem}

For a more detailed discussion on the basic properties (connectedness, coefficients) of the incidence forms and some examples we refer to~\cite[Section 2]{JM}.

\subsection{{\it Proof of Theorem~\ref{thm:mrgvsCar}}}\label{subsec:proofA}

Fix a gentle bound quiver $(Q,I)$, and let  $(G,\leq)=\GPC(\quiver,I)$ be the associated marked ribbon graph  and $\Car_\La$, the Cartan matrix  of  the gentle algebra $\La:=kQ/I$. Note that since $I$ is a monomial ideal, there is an isomorphism of $k$-vector spaces
\begin{equation}\label{eq:algebradecomp}
A \cong \bigoplus_{\text{$\eta$ a permitted path}} k \eta.
\end{equation}
Recall that for a vertex $\alpha \in \verts{G}$, the corresponding permitted thread $\eta^{\alpha}=\alpha_1\cdots\alpha_{\ell^{\alpha}}\in\perm_{\quiver,I}$  is a string in the sense of Butler-Ringel~\cite{ButlerRingel}. Consider the set
$R_{\eta^{\alpha}}:=\{\text{permitted paths of $(Q,I)$}\}\,\setminus\, \{{}_{[t]}\eta^{\alpha}\}_{t=0}^{\ell_\alpha}$.
It is clear that the $k$-linear space $\langle R_{\eta^{\alpha}} \rangle$ generated by $R_{\eta^{\alpha}}$ is a left ideal of $A$ and for the left $A$-module  $\stringmod{\alpha}:=A/\langle R_{\eta^{\alpha}} \rangle$ there is an  isomorphism of $k$-vector spaces
\[
M_{\alpha}\cong\bigoplus_{t=0}^{\ell^{\alpha}}k\, {}_{[t]}\eta^{\alpha}=k\alpha_1\cdots\alpha_{\ell^{\alpha}} \oplus k\alpha_2\cdots\alpha_{\ell^{\alpha}} \oplus \ldots \oplus k\alpha_{\ell^{\alpha}}\oplus k\triv_{\sou(\eta^{\alpha})}
\]
by~\eqref{eq:algebradecomp}.
Note that $M_{\alpha}$ is a left $\La$-module isomorphic to the string module associated to the string $\eta^{\alpha}$ as defined in~\cite{ButlerRingel}. In particular, the dimension vector of $M_{\alpha}$ encodes the (half-)edges of $G$ that are incident to vertex $\alpha$, that is, $\alpha$-th column of the incidence matrix $\Inc(G)$ of $G$ has the shape: 
\begin{equation}\label{eq:IncString}
\Inc(G)\bas_{\alpha}=\vdim{M_{\alpha}}, \qquad \text{for every vertex $\alpha \in\verts{G}$}
\end{equation}
(see Definition~\ref{def:bidi}).
Considering a vertex $i$ of $\quiver$ as an edge of $G$, then its end-vertices $\alpha,\beta \in \verts{G}$ correspond to permitted threads $\eta^{\alpha},\eta^{\beta}$ that intersect at vertex $i$ (to be precise, the two split permitted threads with center $i$ are $(\eta^{\alpha}_{[r]},{}_{[r]}\eta^{\alpha})$ and $(\eta^{\beta}_{[s]},{}_{[s]}\eta^{\beta})$ for some integers $r$ and $s$, see Lemma~\ref{lem:iota}(b)).
One can directly show that the following is a short exact sequence of left $\La$-modules,
\begin{equation}\label{eq:ses}
\xymatrix{0 \ar[r] & P_i \ar[r]^-{\left[\begin{smallmatrix} {}_{[r]}\eta^{\alpha} \\ {}_{[s]}\eta^{\beta} \end{smallmatrix} \right]} & M_{\alpha} \oplus M_{\beta} \ar[r]^-{\left[\begin{smallmatrix} a
 & -b \end{smallmatrix} \right]} & E_i \ar[r] & 0, }
\end{equation}
where ${}_{[r]}\eta^{\alpha}$ denotes the right multiplication by the path ${}_{[r]}\eta^{\alpha}$, and similarly for ${}_{[s]}\eta^{\beta}$,  $E_i=D(\triv_iA)=D(\bigoplus k \eta)$ is the corresponding left indecomposable injective module, and $a:M_{\alpha} \to E_i$ is given by ${}_{[t]}\eta^{\alpha} \mapsto 0$ if $t<r$ and ${}_{[t]}\eta^{\alpha} \mapsto ({}_{[r]}\eta^{\alpha})_{[t-r]}^*$ if $t\geq r$ (where the direct sum in the description of $E_i$ is taken over all permitted paths $\eta$ with $\tar(\eta)=i$ and $\eta^*$ denotes the usual dual basic element), and similarly for $b$. Together with the definition of incidence matrix, equations~(\ref{eq:IncString}) and~(\ref{eq:ses}) imply
\begin{equation}\label{eq:CC}
\Inc(G)\Inc(G)^{\tr}\bas_i=\Inc(G)(\bas_{\alpha}+\bas_{\beta})=\vdim{M_{\alpha}}+ \vdim{M_{\beta}}=
\vdim{P_i} + \vdim{E_i}=(\Car_\La+\Car_\La^\tr)\bas_i.
\end{equation}
This happens for every $i\in Q_0$ which means that  the first equality $\Inc(G)\Inc(G)^{\tr}=\Car_\La+\Car_\La^\tr$ of~\eqref{eq:mrbvsCar} in Theorem ~\ref{thm:mrgvsCar} holds. Moreover, this shows that the Euler form $q_A$ in the basis of indecomposable projectives is the incidence form $q_G$ of $G$  and hence it is non-negative (see Definition~\ref{def:incqf} and Subsection~\ref{subsec:eulf}). 

To show the second equality of~\eqref{eq:mrbvsCar} we apply the following observation.

\smallskip

\begin{claim}\label{claimone}
 {If $\wlk$ is an anti-walk of $(G,\leq)$ having as source the vertex $\sou(\wlk)=\beta$, then the string complex $\stringcpx{0,\invw{\wlk}}$ is a projective resolution of the string module $M_{\beta}$. In particular, if $\wlk=i_1\ldots i_Li_{L+1}$ then the resolution $\stringcpx{0,\invw{\wlk}}=\stringcpx{-L,\wlk}$ induces the following exact sequence of $\La$-modules
\begin{equation*}\label{eq:resolution}
\xymatrix{0 \ar[r]& P_{i_{1}} \ar[r] & P_{i_{2}} \ar[r] & \cdots \ar[r] & P_{i_L} \ar[r] & P_{i_{L+1}}\ar[r]&M_\beta\ar[r]&0.}
\end{equation*}}
\end{claim}

This result was noted in~\cite[Lemma~2.4]{Bob2cycle} and its proof, see also~\cite[Lemma~7.21]{JM}, cf.~\eqref{eq:lengthAntiWalk}, \eqref{eq:bijections} and Subsection~\ref{subsec:perfcpx}. Similarly, the anti-walk $\wlk$ with target vertex $\tar(\wlk)=\alpha$ contains the information of an injective  resolution of the string module $M_{\alpha}$ (see Claim \ref{claimtwo} below).

\smallskip
As consequence of Claim \ref{claimone}, for an anti-walk $\wlk=i_1\ldots i_Li_{L+1}$ with $\sou(\wlk)=\beta$  we have
\begin{equation}\label{eq:alterresolution}
\vdim{M_{\beta}}=\sum_{t=1}^{L+1}(-1)^{L+1-t}\vdim{P_{i_t}}=\Car_\La\sum_{t=1}^{L+1}(-1)^{L+1-t}\bas_{\vtx{i_t}},
\end{equation} cf. Remark~\ref{rem:euler} (recall that $\quiver_0=\edgs{G}$).
Assume now that the gentle bound quiver  $(\quiver,I)$ has finite global dimension and let
$(\fmrgs{G},\fmrgs{\leq},\fmrgs{\sigma}):=\GFC(\quiver,I)$ be the forbidden marked ribbon graph as in Definition~\ref{rem:forb}. Observe that by construction, the $\theta$-th column of the  incidence matrix of the bidirected graph $(\fmrgs{G},\fmrgs{\sigma})$ has the shape
\begin{equation}\label{eq:incantistring}
\Inc(\fmrgs{G},\fmrgs{\sigma})\bas_{\theta}=\sum_{t=1}^{L+1}(-1)^{L+1-t}\bas_{\vtx{i_t}},
\end{equation} 
where $\theta:={\GBQ}(\wlk)\in  \forb_{\quiver,I}$ is the corresponding forbidden thread, see~\eqref{eq:bijections} and Definition~\ref{def:bidi} (recall that $\verts{\fmrgs{G}}=\forb_{\quiver,I}$ and $\edgs{\fmrgs{G}}=\quiver_0$).
Now combining the formulae~\eqref{eq:IncString}, \eqref{eq:alterresolution} and~\eqref{eq:incantistring} we get
\begin{equation}\label{eq:inccarincf}
\Inc(G)=\Car_\La\cdot\Inc(\fmrgs{G},\fmrgs{\sigma}),
\end{equation}
via the bijection $\Phi_2=\sou\circ\GBQ^{-1}:\forb_{Q,I}\to\perm_{Q,I}$ of~\eqref{eq:bijections}.
Thus, since $\Car_A$ is $\ZZ$-invertible in this case,  we have
\begin{equation}\label{eq:CCbdg}
\Inc(\fmrgs{G},\fmrgs{\sigma})\Inc(\fmrgs{G},\fmrgs{\sigma})^{\tr}=\Car_\La^{-1}\Inc(G)\Inc(G)^{\tr}\Car_\La^{-\tr}=\Car_\La^{-1}(\Car_\La+\Car_\La^{\tr})\Car_\La^{-\tr}=\Car_\La^{-1}+\Car_\La^{-\tr}.
\end{equation}
This shows the second equality of~\eqref{eq:mrbvsCar}.
Note that this means that the Euler form $q_A$ in the basis of simples is the incidence form $q_{\fmrgs{G},\fmrgs{\sigma}}$ of the bidirected graph $(\fmrgs{G},\fmrgs{\sigma})$ in the case of $\gldim(A)<\infty$.
\smallskip

(a) By Lemma~\ref{thm:IncGentle} we can assume that $(\quiver,I)=\GBQ(G,\leq)$. To show the equivalence of (a2) and (a3) let $S,T$ be the string functions associated to a direction $\sigma:\hedgs{G} \to \{\pm 1\}$ of $G$ as in Lemma~\ref{rem:string}
and observe that
\begin{itemize}
\item[(i)] $S=-T$ if and only if $\sigma$ is constant on all half-edges incident to a vertex of $G$, that is, $\sigma(h)=\sigma(h')$ if $\s(h)=\s(h')$.

Indeed, assume first that $S=-T$ and recall that any arrow of $Q$ is contained in a permitted thread $\eta^\alpha=\alpha_1\cdots \alpha_{\ell^{\alpha}}$ with $\ell^\alpha \geq 1$, for some $\alpha\in\verts{G}$ with $\s^{-1}(\alpha)=\{h^{\alpha}_0 > h^{\alpha}_1 > \ldots >h^{\alpha}_{\ell^{\alpha}}\}$, where $\alpha_t=(h^{\alpha}_{t-1},h^{\alpha}_{t})$, see Remark~\ref{rem:linear} and Lemma~\ref{lem:iota}(a).
 By definition of $S,T$ and the condition $S=-T$ we have $S(\alpha_t)=-T(\alpha_{t+1})=S(\alpha_{t+1})$ for $t=1,\ldots,\ell^\alpha-1$.  Using the definition of $\sigma=\sigma^{S,T}$, see~(\ref{eq:stringONE}), we have $\sigma(h^{\alpha}_t)=S(\alpha_t)=S(\alpha_{t+1})=\sigma(h^{\alpha}_{t+1})$ for $t=1,\ldots,\ell^\alpha-1$, and $\sigma(h^{\alpha}_0)=-T(\alpha_1)=S(\alpha_1)=\sigma(h^{\alpha}_1)$, that is, $\sigma$ is constant on $\s^{-1}(\alpha)$. Conversely, if $\sigma$ is constant on every $\s^{-1}(\alpha)$, by the definition of $S=S^\sigma$ and $T=T^\sigma$, see~(\ref{eq:stringTWO}), we have $S(\alpha_t)=-T(\alpha_t)$ for $t=1,\ldots,\ell^\alpha$. Since every arrow of $(\quiver,I)=\GBQ(G,\leq)$ is of the form $\alpha_t=(h^{\alpha}_{t-1},h^{\alpha}_{t})$ as above, this shows that $S=-T$, as wanted.
\end{itemize}

If $G$ is bipartite with partition $\verts{G}=V' \cup V''$ we can define a direction $\sigma$ by taking $\sigma(h')=1$ for all $h' \in \s^{-1}(V')$ and $\sigma(h'')=-1$ for all $h'' \in \s^{-1}(V'')$; then by (i) the corresponding string functions satisfy $S^{\sigma}=-T^{\sigma}$. Conversely, if there are string functions $S=-T$ of $(\quiver,I)$ then by (i) the associated direction $\sigma^{S,T}$ of $G$ is constant around each vertex, which defines a partition $\verts{G}=V^+ \cup V^-$ where $V^\epsilon:=\{\alpha \in \verts{G} \mid \text{$\sigma^{S,T}(h)=\epsilon$ for any half-edge $h$ with $\s(h)=\alpha$} \}$. Since $\sigma^{S,T}$ is a direction, the end-vertices of each edge of $G$ belong to different parts of this partition, which shows that $G$ is bipartite.

To show the equivalence of (a1) and (a3), note that any closed walk $\wlk$ of $\quiver$ without repeated edges is a homotopy string of $(\quiver,I)$. Since we are interested in $\wlk$ up to cyclic permutations, we may assume that $\wlk=\sigma_1(\wlk)\cdots \sigma_L(\wlk)$ is a homotopy presentation such that:
\[
\text{If both $\sigma_1(\wlk), \sigma_L(\wlk)$ are direct (resp. inverse) homotopy letters, then $\sigma_L(\wlk)\sigma_1(\wlk) \in I$ (resp. $\invw{(\sigma_L(\wlk)\sigma_1(\wlk))} \in I$).} \tag{$*$}
\]
In this situation we have
\begin{itemize}
\item[(ii)] $\deg(\wlk)=\nrp{\wlk}-\nrm{\wlk}$, where $\nrr{\wlk}=\nrp{\wlk}+\nrm{\wlk}$ is the number of relations in $\wlk$ as in Definition \ref{def:mcc}.

Recall that $\deg(\wlk)$ is the number of direct minus the number of inverse homotopy letters in the decomposition $\wlk=\sigma_1(\wlk)\cdots \sigma_L(\wlk)$.  Considering the letters $\sigma_t(\wlk)$ as arrows in a cycle, say $\widetilde{\mathbb{A}}_L$, a simple combinatorial argument shows that the number of clockwise oriented arrows minus the number of counterclockwise oriented arrows of the cycle (that is, $\deg(\wlk)$) coincides with the number of clockwise oriented paths of length two minus the number of counterclockwise oriented paths of length two of the cycle (that is, $\nrp{\wlk}-\nrm{\wlk}$ by the above condition $(*)$ on the homotopy presentation of the closed walk $\wlk$).
\end{itemize}
Moreover, note that for any reduced walk $\Wlk$ of $G$ we have that
\begin{itemize}
\item[(iii)] the integers $\deg(\Wlk)$ and $\ell(\Wlk)$ have different parity.

Since $|\deg(ij)|=1$ for any reduced walk $ij$ of length two, if $\Wlk=i_1 \cdots i_Li_{L+1}$ we have $\deg(\Wlk)=\sum_{t=1}^L\deg(i_ti_{t+1}) \equiv L (\mod 2)$, that is, $\deg(\Wlk)-\ell(\Wlk) \equiv 1 (\mod 2)$.
\end{itemize}

Assume first that $(\quiver,I)$ satisfies the multi-clock condition, let $i_1\cdots i_L$ be a cycle of $G$ and take $\Wlk:=i_1\cdots i_Li_{L+1}$ with $i_{L+1}:=i_1$. It is easy to see that $\Wlk$ is a reduced walk and $\wlk:=\GBQ(\Wlk)$ is a closed walk without repeated edges, see Lemma~\ref{lem:walkGQ}, satisfying the condition $(*)$ above. Since $\nrr{\wlk}=\nrp{\wlk}+\nrm{\wlk}$ is even by assumption, then so is $\deg(\wlk)=\deg(\Wlk)$ by (ii), see Lemma~\ref{lem:walkGQ}(c). Using (iii), the length $L+1$ of $\Wlk$ is odd, that is, the length of original cycle $i_1\cdots i_L$ is even. Since this holds for any cycle then $G$ is a bipartite graph. Conversely, assume that $G$ is a bipartite graph, let $\wlk$ be a closed walk of $\quiver$ without repeated edges satisfying condition $(*)$ above, and take $\Wlk:=\GPC(\wlk)=i_1\cdots i_Li_{L+1}$. Since $i_1\cdots i_L$ is a closed walk (by Lemma~\ref{lem:walkGQ}(a) and $(*)$) and $G$ is bipartite, then $L$ is even and by (iii) the degree $\deg(\Wlk)=\deg(\wlk)$ is also even (since $\ell(\Wlk)=L+1$ is odd). By (ii) the difference $\deg(\wlk)=\nrp{\wlk}-\nrm{\wlk}$ and the sum $\nrr{\wlk}=\nrp{\wlk}+\nrm{\wlk}$ are even integers. This shows that $(\quiver,I)$ satisfies the multi-clock condition.

To show the equivalence of (a3) and (a4), assume that $\gldim(\La)<\infty$ and observe using equation~\eqref{eq:inccarincf} that the coranks of the incidence matrices $\Inc(G)$ and $\Inc(\fmrgs{G},\fmrgs{\sigma})$ coincide (since in this case $\Car_\La$ is $\Z$-invertible). By Lemma~\ref{lem:Null}, it follows that $G$ is bipartite if and only if $(\fmrgs{G},\fmrgs{\sigma})$ is balanced, as claimed.

\medskip

(b) 
Recall that the Euler form $q_A$ of $A$ in the basis of indecomposable projectives is the incidence form $q_G$ of (the underlying graph of) the associated marked ribbon graph $\GPC(\quiver,I)=(G,\leq)$, see~\eqref{eq:CC}. So by Theorem~\ref{thm:JM1}
the corank of $q_A$ is given by $\CRnk(q_{A})=|\edgs{G}|-|\verts{G}|+\balanced_{{G}}$, where $\balanced_{G}=\balanced_{G,\sigma^+}=\Null(\Inc(G))$. Moreover, by  Lemma~\ref{lem:Null}(b) and the equivalence \lq\lq (a1) $\Leftrightarrow$ (a3)'' proved above, the
value of $\Null(\Inc(G))\in\{0,1\}$ coincides with the \lq\lq multi-clock characteristic'' parameter $\gcc_\La$ of Definition~\ref{def:mcc}. Thus we have proved the first equality $\CRnk(q_{A})=|\edgs{G}|-|\verts{G}|+\gcc_\La$ of~\eqref{eq:mrbvsCarcrk} for the corank. This gives us also the last equality
$\CRnk(q_{A})=|\edgs{\fmrgs{G}}|-|\verts{\fmrgs{G}}|+\gcc_\La$ in the finite global dimension case, since \begin{equation}\label{eq:edgesverticesThmA}
|\edgs{G}|=|Q_0|=|\edgs{\fmrgs{G}}|\quad \text{and}\quad |\verts{G}|=|\perm_{Q,I}|=|\forb_{Q,I}|=|\verts{\fmrgs{G}}|,
\end{equation} see Definitions~\ref{def:BoundQuiverInc} and~\ref{rem:forb} and the bijections of the diagram~\eqref{eq:bijections}, cf. Lemmas~\ref{lem:iota}(b) and~\ref{lem:forbidden}(b). (Note that one can alternatively apply Theorem~\ref{thm:JM1} and the fact that $q_A$ in the basis of simples is the incidence form $q_{\fmrgs{G},\fmrgs{\sigma}}$, cf. Lemma~\ref{lem:Null}(a)). To show the remaining equality $\CRnk(q_\La)=|Q_1|-|Q_0|+\gcc_\La$ first note that given a vertex $\alpha\in V(G)$, its degree $\deg(\alpha)$ (that is, the number of incident half-edges)  equals to $\ell(\eta^\alpha)+1$, where $\ell(\eta^\alpha)$ is the length of the corresponding permitted thread $\eta^\alpha\in\perm_{Q,I}$, cf. Definition~\ref{def:BoundQuiverInc}. Hence applying~\eqref{eq:edgesverticesThmA} and elementary graph theoretic property  we get
\begin{equation}\label{eq:edgesdegreesThmA}
2|Q_0|=2|E(G)|=\sum_{\alpha\in V(G)}\deg(\alpha)=\sum_{\alpha\in V(G)}(\ell(\eta^\alpha)+1)
=\sum_{\eta^\alpha\in \perm_{Q,I}}\ell(\eta^\alpha)+|V(G)|=|Q_1|+|V(G)|
\end{equation}
(the latter equality follows since every arrow of $Q$ is contained in exactly one non-trivial permitted thread, by Lemma~\ref{lem:iota}(a)). In particular, we have $|V(G)|=2|Q_0|-|Q_1|$ and this gives us
\begin{equation}\label{eq:centraleqThmA}
\CRnk(q_\La)=|\edgs{G}|-|\verts{G}|+\gcc_\La=|Q_0|-(2|Q_0|-|Q_1|)+\gcc_\La=|Q_1|-|Q_0|+\gcc_\La.
\end{equation}
The formulae for the rank $\Rnk(q_A)$ follow from those for $\CRnk(q_A)$ since $\Rnk(q_A)=|Q_0|-\CRnk(q_A)$, cf.~\eqref{eq:edgesverticesThmA} and Subsection~\ref{subsec:qf}.
To finish the proof of claim (b) it remains to conclude that the binary parameter $\gcc_A$ is a derived equivalence invariant. But this follows from the proved formula for the corank of $q_A$. Indeed, by~\eqref{eq:centraleqThmA} we have that $\gcc_\La=|Q_0|-|Q_1|+\CRnk(q_\La)$ and every summand of the right-hand side of this equality is a derived invariant, see Remark~\ref{rem:euler}, and~\cite[Proposition B]{AAG} for the invariance of  $|Q_1|$.

\smallskip

(c) First observe that given an arbitrary bidirected graph $(G',\sigma')$, the diagonal coefficients of its incidence form $q'=q_{G',\sigma'}:\ZZ^{E(G')}\to\ZZ$ satisfy:
\begin{equation}\label{eq:coeffsincform}
q'_{i} = \left\{
\begin{array}{l l}
2, & \text{if $i$ is a negative (i.e., bidirected) loop}, \\
1, & \text{if $i$ is not a loop}, \\
0, & \text{if $i$ is a positive loop},
\end{array} \right.
\end{equation}
for $i\in E(G')$, see Definitions~\ref{def:bidi},\ref{def:incqf}, cf.~\cite[Lemma 2.11]{JM}. In the case of  our marked ribbon graph $\GPC(\quiver,I)=(G,\leq)$, every edge of $G$
 is negative (with respect to the constant bidirection $\sigma^+\equiv+1$). So since  the Euler form $q=q_A$ in the basis $\projbas=\{[P_i]\}_{i\in Q_0}$ is the incidence form $q_G=q_{G,\sigma^+}$ of $G$, by~\eqref{eq:coeffsincform}
 we have $q_i=q_i^\projbas\in\{1,2\}$ for each $i\in \edgs{G}=Q_0$. In particular, $q$ is Cox-regular by the first statement of Theorem~\ref{thm:JM1}, and $q$ is connected by Lemma~\ref{lem:eulfconnected}. Moreover, by Definition~\ref{def:BoundQuiverInc} it follows that $G$ has a (negative) loop $i\in E(G)$ at vertex $\alpha\in V(G)$ if and only if the corresponding permitted thread $\eta^\alpha$ self-crosses at $i$ as a vertex in $Q_0$. This proves the formula~\eqref{eq:mrbvsCarqisproj} of (c1), cf. Lemma~\ref{lem:iota}. 

Assume now that  $q$ is unitary. By the arguments above, this holds precisely when $G$ has no loops, equivalently, when  $(Q,I)$ has no self-crossing permitted path. Then by Theorem~\ref{thm:JM1}(a), $q$ as the incidence form of $G$, has Dynkin type $\AA$ or $\DD$, and $\Dyn(q)=\DD$ if and only if $|V(G)|\geq 4$ and $(G,\sigma^+)$ is unbalanced. Note that $(G,\sigma^+)$ is unbalanced precisely when $G$ is not bipartite, equivalently, when $\gcc_A=0$ by the equivalence \lq\lq (a1) $\Leftrightarrow$ (a3)'' shown before. This proves (c2), since $|V(G)|=2|Q_0|-|Q_1|$ by~\eqref{eq:edgesdegreesThmA} (note that if $\gcc_A=0$ then $r_A=2|Q_0|-|Q_1|$, cf.~\eqref{eq:mrbvsCarcrk}).

 If $q$ is not unitary, then by the arguments above $G$ has a loop, which means that $G$ is not bipartite. Thus $\gcc_A=0$. Assume first that $2|Q_0|-|Q_1|\geq 2$. It means that $r_A=|V(G)|\geq 2$, so $q$ as the incidence form of $G$ having at least 2 vertices is irreducible by~\cite[Lemma 2.14(b)]{JM}. So we can apply Theorem~\ref{thm:JM1}(b) which implies  that $\Dyn(q)=\CC_{r_A}$. In the remaining case $r_A=|V(G)|=2|Q_0|-|Q_1|=1$ the graph $G$ consists of one vertex and $n:=|E(G)|$ two-tail loops, thus
 \begin{equation}\label{eq:onevertex}
 \Inc(G)=[2,2,2,\ldots,2]^\tr\in\MM_{n,1}(\ZZ)\quad\text{and}\quad
 \Car_\La+\Car_\La^\tr=\Inc(G)\Inc(G)^{\tr}=4\cdot\jeden\in\MM_{n,n}(\ZZ),\end{equation}
  where $\jeden$ denotes the $n\times n$ matrix whose all entries equal 1, cf. Definitions~\ref{def:bidi},\ref{def:incqf} and~\eqref{eq:CC}. This means that the Euler form $q=q_A$ in the basis of projectives is given as $q(x)=2x^\tr J x$, for $x\in\ZZ^n\cong K_0(\per\, A)$. In particular,
  $\frac{1}{2}q(x)=x^\tr J x$ is a non-negative unit form of corank $n-1$ and Dynkin type $\AA_1$, cf.~(D1) in Subsection~\ref{subsec:qf} and~\cite[Example 6.9]{JM}. 

\smallskip

(d) 
Recall that in the case $\gldim(A)<\infty$ the Euler form $q=q_A$ in the basis $\simplebas=\{[S_i]\}_{i\in Q_0}$ of simples is the incidence form $q_{\fmrgs{G},\fmrgs{\sigma}}$ of the bidirected graph $(\fmrgs{G},\fmrgs{\sigma})$, see
\eqref{eq:CCbdg}. In particular,  $q$ is semi-Cox-regular  by Theorem~\ref{thm:JM1}. Note that by construction, $\fmrgs{G}$ has a loop $i\in E(\fmrgs{G})$ at vertex $\alpha\in V(\fmrgs{G})$ if and only if the forbidden thread $\theta^\alpha$ corresponding to $\alpha$ self-crosses at $i$ as a vertex in $Q_0$, see Definition~\ref{rem:forb}. Note that the latter is equivalent with the existence of a (unique) almost-full cycle $\wlk$ at $i$ in $Q$, cf. Lemmas~\ref{lem:fgd} and~\ref{lem:forbidden} (take $\wlk$  as the unique subpath of $\theta^\alpha$ with $\sou(\wlk)=\tar(\wlk)=i$). Moreover, by the definition of the bidirection $\fmrgs{\sigma}$ in Definition~\ref{rem:forb} it follows that if  $i=\{h, \iota h\}$ is a loop in $\fmrgs{G}$ then $\fmrgs{\sigma}(h)=-\fmrgs{\sigma}(\iota h)$ (resp.~$\fmrgs{\sigma}(h)=\fmrgs{\sigma}(\iota h)=+1$), that is, $i$ is a positive (resp.~negative) loop,  provided the corresponding almost-full cycle $\wlk$ has odd (resp.~even) length, cf.~\eqref{eq:signofedge}. Combining these observations with  the formula~\eqref{eq:coeffsincform} we get the proof of~\eqref{eq:mrbvsCarqis}. 
In particular, $q$ is a unit form if and only if  $(Q,I)$ has no
almost-full cycles. Moreover, if there is an almost-full cycle of odd length at $i\in Q_0$ then by the above we have $0=q_i=q([S_i])$. Since $q$ is non-negative then $\rad(q)=\roots_q(0)=q^{-1}(0)$ thus $q_{ij}=q([S_i],[S_j])=0$ for each $j\in Q_0$, which means that $q$ is disconnected, see Subsection~\ref{subsec:qf}. Finally, assume that there is no almost-full cycle of odd length and $2|Q_0|-|Q_1|\geq 3$. This means that $(\fmrgs{G},\fmrgs{\sigma})$ has no positive loop and $|V(\fmrgs{G})|=2|Q_0|-|Q_1|\geq 3$, cf.~\eqref{eq:edgesdegreesThmA} and~\eqref{eq:edgesverticesThmA}. This implies that $q$, as the incidence form of $(\fmrgs{G},\fmrgs{\sigma})$, is connected by~\cite[Lemma 2.14(c)]{JM}.
This finishes the proof of (d1).

To show (d2)-(d3), assume that $q$ is connected. 
We proceed similarly as in (c2)-(c3). If  $q$ is unitary then $(\fmrgs{G},\fmrgs{\sigma})$ has no loops and by Theorem~\ref{thm:JM1}(a), $q$ has Dynkin type $\AA$ or $\DD$. Moreover, $\Dyn(q)=\DD$ if and only if $|V(\fmrgs{G})|\geq 4$ and $(\fmrgs{G},\fmrgs{\sigma})$ is unbalanced. Equivalently,  when $\gcc_A=0$ and $r_A=2|Q_0|-|Q_1|\geq 4$ by \lq\lq (a1) $\Leftrightarrow$ (a4)'' and~\eqref{eq:mrbvsCarcrk} shown before. This proves (d2). If $q$ is not unitary, then $(\fmrgs{G},\fmrgs{\sigma})$ has a negative (bidirected) loop, which means it is unbalanced, that is, $\gcc_A=0$. Now, if  $2|Q_0|-|Q_1|\geq 2$ then $r_A=|V(\fmrgs{G})|\geq 2$, so $q$ is irreducible by~\cite[Lemma 2.14(b)]{JM}. Thus $\Dyn(q)=\CC_{r_A}$ by Theorem~\ref{thm:JM1}(b). In the remaining case $r_A=|V(\fmrgs{G})|=2|Q_0|-|Q_1|=1$ the graph $\fmrgs{G}$ consists of one vertex and $n:=|E(\fmrgs{G})|$ bidirected loops, with $s$ two-tail and $n-s$ two-head loops for some $0\leq s\leq n$.
Thus, up to renumbering, we have:
 \begin{equation}\label{eq:onevertexbidi}
 \Inc(\fmrgs{G},\fmrgs{\sigma})=[2,\ldots,2,-2,\ldots,-2]^\tr\in\MM_{n,1}(\ZZ)\quad\text{and}\quad
 \Car_\La^{-\tr}+\Car_\La^{-1}=\Inc(\fmrgs{G},\fmrgs{\sigma})\Inc(\fmrgs{G},\fmrgs{\sigma})^{\tr}=4\left[
 \begin{array}{cc}
 \!J_{s,s}\!&\!-J_{s,n-s}\!\\
 \!-J_{n-s,s}\!&\!J_{n-s,n-s}\!
 \end{array}
 \right]\in\MM_{n,n}(\ZZ),\end{equation}
  where $\Inc(\fmrgs{G},\fmrgs{\sigma})$ has $s$ entries $2$ and the remaining entries are $-2$, and
  $\jeden_{t, t'}\in\MM_{t,t'}(\ZZ)$ denotes the matrix whose all entries equal 1, cf.\eqref{eq:onevertex}. Similarly as in (c3), this means that
  $\frac{1}{2}q$ is a non-negative unit form of rank $1$ and corank $n-1$, hence it has  Dynkin type $\AA_1$, cf.~\cite[Example 6.9]{JM}. This finishes the proof of (d3) and  of the theorem.
\hfill$\Box$

\subsection{Applications}\label{subsec:applA}

In this subsection we show some consequences of Theorem~\ref{thm:mrgvsCar} and its proof. We start with the following immediate conclusion which provides a collection of graph theoretic invariants for gentle algebras in the spirit of recent results for Brauer graph algebras~\cite[Theorem 1.2]{AZ22.Brauer} and~\cite[Theorem~A]{OppZvo22}.

\begin{corollary}\label{cor:OZspirit}
Let $(Q,I)$ and $(Q',I')$ be two gentle bound quivers with marked ribbon graphs $(G,\leq)=\GPC(\quiver,I)$ and $(G',\leq')=\GPC(\quiver',I')$. If the gentle algebras $A=kQ/I$ and $A'=kQ'/I'$ are derived equivalent then:
\begin{enumerate}[label={\textnormal{(\alph*)}},topsep=3px,parsep=0px]
\item $G$ and $G'$ have the same number of vertices, edges and faces,
\item the multi-sets of lengths and closed degrees of faces $\{(\ell(\wlk),\Cdeg(\wlk)) \mid\ \wlk\in \Faces(G,\leq)\}$ and $\{(\ell(\wlk),\Cdeg(\wlk)) \mid\ \wlk\in \Faces(G',\leq')\}$  coincide,
 \item either both or none of $G$ and $G'$ are bipartite.
 \end{enumerate}
\end{corollary} 
\begin{proof}
Recall that $|\edgs{G}|=|Q_0|$ hence  derived invariance of $|\edgs{G}|$ is clear. For $|\verts{G}|$ observe that if $A\congder A'$ then by Theorem~\ref{thm:mrgvsCar}(b) we have $\gcc_A=\gcc_{A'}$ and $|\verts{G}|-\gcc_{\La}=\Rnk(q_\La)=\Rnk(q_{\La'})=|\verts{G}|-\gcc_{\La'}$ thus $|\verts{G}|=|\verts{G'}|$. Claim (b) is a direct consequence of Theorem~\ref{thm:AAG} and the formula~\eqref{eq:AAGnew}. In particular, (b) implies that $|\Faces(G,\leq)|=|\Faces(G',\leq')|$. The remaining claim (c) follows from  derived invariance of the multi-clock condition (see Theorem~\ref{thm:mrgvsCar}(b)) and the equivalences in Theorem ~\ref{thm:mrgvsCar}(a).
\end{proof}

The following example shows that sometimes we can easily disprove derived equivalences by a quick glance at the quiver and relations.

\begin{example}\label{ex:Amiotder} Consider the three gentle bound quivers $(Q^{(i)},I^{(i)})$, $i=0,1,2$, from Example~\ref{exa:amiot}(b). Let $A^{(i)}:=kQ^{(i)}/I^{(i)}$ be the corresponding gentle algebras. Observe that all three algebras have the same AAG invariant $\{(4,6)\}$ and the Coxeter polynomial $z^5-z^4-z+1$ (cf.~Subsection \ref{subsec:antiwalks} and Example~\ref{ex:Amiotcoxp}). However, we easily see that $\cyccond_{A^{(0)}}=\cyccond_{A^{(2)}}=1$ (that is, $A^{(0)}$ and $A^{(2)}$ satisfy the multi-clock condition) and $\cyccond_{A^{(1)}}=0$, thus $A^{(0)}\not\congder A^{(1)}\not\congder A^{(2)}$ by Theorem~\ref{thm:mrgvsCar}(b). Recall that Amiot showed that also $A^{(0)}\not\congder A^{(2)}$ by using non-trivial geometric arguments, see~\cite[3.6]{Amiot3}. Moreover, the coranks and Dynkin types of the Euler forms are as follows: $\CRnk(q_{A^{(0)}})=2=\CRnk(q_{A^{(2)}})$, $\CRnk(q_{A^{(1)}})=1$ and $\Dyn(q_{A^{(0)}})=\AA_3=\Dyn(q_{A^{(2)}})$, $\Dyn(q_{A^{(1)}})=\DD_4$, cf.~Theorem~\ref{thm:mrgvsCar}(c). On the other hand, we verify directly that each of the three  marked ribbon graphs $(G^{(i)},\leq^{(i)}):=\GPC(\quiver^{(i)},I^{(i)})$ for $i=1,2,3$, has 4 vertices, 5 edges and 1 face of length 10 and closed degree 2 (cf.~example on p.~\pageref{ex:Amiotmrg}). But $G^{(0)}$ and $G^{(2)}$ are bipartite and $G^{(1)}$ is not, cf.~Theorem~\ref{thm:mrgvsCar}(a). This also shows that the collection of graph theoretic parameters in Corollary~\ref{cor:OZspirit} is not a complete derived invariant. 
\end{example}

 The subsequent conclusion  of Theorem~\ref{thm:mrgvsCar} may be viewed as a refinement of ~\cite[Theorem C(ii)]{BGSdisc}.
\begin{corollary}\label{cor:BGS} Let $\La=kQ/I$ be a 1-cycle gentle algebra $($of arbitrary $\gldim(A))$. Then $q_{\La}$ is non-negative of $\CRnk(q_\La)=\gcc_\La\in\{0,1\}$.
\end{corollary}
\begin{proof}
By Theorem~\ref{thm:mrgvsCar}, $q_\La$ is non-negative and $\CRnk(q_\La)=|Q_1|-|Q_0|+\gcc_\La=\gcc_\La$.
\end{proof}

\smallskip

Recall that an algebra $\La$ over an algebraically closed field is {\bf derived discrete} if for each $v=(v_i)\in\NN^{(\ZZ)}$ there are only finitely many isoclasses of indecomposable objects in $\Db{\La}$ of a given homology dimension vector $v$, see~\cite{VossieckDerDisc}. In particular, if $\La$ is {\bf derived finite}, that is, $\Db{\La}$ has only finitely many indecomposables up to shift and isomorphism, then $\La$ is derived discrete.

\begin{corollary}\label{cor:derdiscsimple} Let $\La$ be a derived discrete algebra over an algebraically closed field. Then the Euler form $q_\La$ is non-negative of
$\CRnk(q_\La)\in\{0,1\}$.  If $\La$ is derived finite then  $q_\La$ is  positive.
\end{corollary}
\begin{proof}
Recall from~\cite{VossieckDerDisc}  $\La$ is derived discrete if and only if  $\La$ is derived hereditary of Dynkin type, or $\La$ is a gentle 1-cycle algebra not satisfying the clock condition (see Remark \ref{rem:tcycle}). Note that $\La$ is derived hereditary of Dynkin type if and only if $\La$ is derived finite. Now the hypothesis follows easily by Corollary~\ref{cor:BGS} and the known fact that the Euler form of a derived hereditary algebra of Dynkin type is positive, cf.~\cite{HDb2}. 
\end{proof}

The classification of Dynkin types of the Euler forms of derived discrete algebras follows by applying Theorem~\ref{thm:mrgvsCar}(c). In particular, we can recover also some other claims on the Euler form from~\cite{BGSdisc}.

Yet another conclusion  of Theorem~\ref{thm:mrgvsCar} (and its proof) provides few nice combinatorial limitations for gentle quivers and the associated graphs.

\begin{corollary}\label{cor:limits}
Let $(Q,I)$ be a gentle bound quiver and take $\GPC(\quiver,I)=(G,\leq)$ and $\GFC(\quiver,I)=(\fmrgs{G},\fmrgs{\leq},\fmrgs{\sigma})$ $($the latter in case $\gldim(Q,I)<\infty)$. Then
\begin{enumerate}[label={\textnormal{(\alph*)}},topsep=3px,parsep=0px]
\item $\cycles(Q)=\cycles(G)=\cycles(\fmrgs{G})$, cf.~\eqref{eq:tcycle},
\item $|Q_0|-1\,\leq\, |Q_1|\,\leq\, 2|Q_0|-1$,
\item $0\,\leq\, \cycles(Q)\,\leq\, |Q_0|$,
\item the following four equalities hold simultanously: $|Q_1|= 2|Q_0|-1$,\  $\cycles(Q)= |Q_0|$, \  $|V(G)|=1$, and $|V(\fmrgs{G})|=1$.
\end{enumerate}
\end{corollary}
\begin{proof} 
Claim (a) is a direct consequence of the corank formula~\eqref{eq:mrbvsCarcrk} of Theorem~\ref{thm:mrgvsCar}(b). First inequality in (b) is obvious, since $|Q_0|-1= |Q_1|$ if and only if $Q$ is a tree, cf.~\eqref{eq:tcycle}. To see the second one, recall that $V(G)=2|Q_0|-|Q_1|$ (see~\eqref{eq:edgesdegreesThmA}). Thus $1\leq 2|Q_0|-|Q_1|$, so $|Q_1|\,\leq\, 2|Q_0|-1$. Inequalities in (c) are a simple reformulation of these in (b), cf.~\eqref{eq:tcycle}. The remaining claim (d) follows from the arguments above and~\eqref{eq:edgesverticesThmA}.
\end{proof}

Let us consider another interesting class of algebras, playing an important role  in modular representation theory of finite groups. Recall that a \textbf{Brauer graph} is a triple $\Gamma=(G,\permut,\mult)$, where $(G,\permut)$ is a ribbon graph (cf.~Subsection~\ref{subs:mrg}) and $\mult:\verts{G}\to \NN$ is a function, called the multiplicity function. We say that $\Gamma$ has trivial multiplicities if $\mult\equiv 1$. With a Brauer graph $\Gamma$ one associates a finite dimensional (bound quiver) $k$-algebra  $B=B(\Gamma)$ called the \textbf{Brauer graph algebra}. We recall that these algebras are symmetric, thus self-injective, so they typically have infinite global dimension, cf.~\cite{Schroll.trivext}. We skip the precise definition, for the details as well as some historical remarks we refer   to~\cite{Schroll.Brauer, Schroll.trivext}, cf.~\cite{AZ22.Brauer}.

\begin{theorem}\label{thm:Brauer}
Let $\Gamma$ be a connected Brauer graph with trivial multiplicities and let $\Bralg=\Bralg(\Brgr)$ be the associated Brauer graph algebra over an algebraically closed field $k$. Then the Euler form $q_\Bralg$ is non-negative. Moreover, the following conditions are equivalent:
\begin{enumerate}[label={\textnormal{(\alph*)}},topsep=3px,parsep=0px]
\item The Euler form $q_\Bralg$ is positive.
\item The Cartan matrix $\Car_{\Bralg}$ of $\Bralg$ is non-singular.
\item $\Brgr$ is a tree or $\Brgr$ is a 1-cycle graph and the unique cycle  has  odd length.
\end{enumerate}
\end{theorem}
\begin{proof} 
First assume that $\Gamma$ consists of a single vertex (resp.~a single non-loop edge). Then $B\cong k$ (resp.~$B\cong k[x]/(x^2)$), cf.~\cite[p.~189]{Schroll.trivext}. In these cases $\Car_B=[1]$
(resp.~$\Car_B=[2]$) and all the statements hold easily. Hence since now we may assume that $\Gamma$ has a vertex of degree at least two as in Definition~\ref{def:mrg}.

By~\cite[Theorem 1.3]{Schroll.trivext} there exists a gentle algebra $\La$ such that $\Bralg\cong T(\La)=\La \lJoin D(\La)$ is the trivial extension of $\La$ by the $\La$-$\La$-bimodule $D(\La)={\rm Hom}_k(\La, k)$. Recall that $T(\La)=\La \oplus D(\La)$ as a vector space and the multiplication is given by $(a, f)\cdot(b,g)=(ab, ag+fb)$. In particular, $1_{T(\La)}=(1_\La,0)$ is the identity of $T(\La)$, and if $1_\La=\idem_1+\ldots+ \idem_n$ is a decomposition of $1_\La$ into the sum of primitive orthogonal idempotents then so is  $1_{T(\La)}=(\idem_1,0)+\ldots+ (\idem_n,0)$ for $1_{T(\La)}$. Moreover, for any $1\leq i,j\leq n$ we have $(\idem_j,0)T(\La)(\idem_i,0)\cong\idem_j\La\idem_i\oplus D(\idem_i\La\idem_j)$. Thus
$(\Car_{T(\La)})_{i,j}=\dimk((\idem_j,0)T(\La)(\idem_i,0))=\dimk(\idem_j\La\idem_i)+\dimk(\idem_i\La\idem_j)=
(\Car_\La)_{i,j}+(\Car_{\La})_{j,i}$. This means that
 $\Car_{T(\La)}=\Car_\La+\Car_\La^\tr$. In particular, the Euler forms of $T(A)$ and $A$ are related as follows
  \begin{equation}\label{eq:eulfBrauer}
  q_{T(A)}=2q_A,
  \end{equation}
  under the identifications $K_0(\per\,T(\La))\cong\ZZ^n\cong K_0(\per\,\La)$ (via the bases of corresponding indecomposable projectives). Thus $q_\Bralg$ is non-negative by Theorem~\ref{thm:mrgvsCar}.

\smallskip

Equivalence of (a) and (b): 
  it was shown above that $\Car_\Bralg$ is a symmetric positive semidefinite (integer) matrix. Therefore, if $\Car_\Bralg$ is not positive definite then $\Car_\Bralg$ has eigenvalue 0. This  means that $\Car_\Bralg$ is singular. The opposite implication follows obviously from Sylvester's criterion.

Equivalence of (a) and (c): if $\Gamma=(G',\permut,\mult)$ with $\mult\equiv 1$, then by~\cite[Theorems 1.2, 1.3]{Schroll.trivext} we may assume that the gentle algebra $A$ with $T(A)\cong B$ as above, has the form $A=kQ/I$, where the associated marked ribbon graph $\GPC(\quiver,I)=(G,\leq)$ has its underlying graph $G$ isomorphic to $G'$; more precisely, $(G,\leq)$ arises, up to isomorphism, from $(G',\permut)$ by choosing a marking $\m:V(G') \to H(G')$ corresponding to so-called admissible cut, cf.~Remark~\ref{rem:correspondence} and~\cite[pp.195-196]{Schroll.trivext} (note that since $\Brgr$ has trivial multiplicities then $\Brgr$ is unique for the isoclass of $B$, see~\cite[Lemma 3.1]{AZ22.Brauer}). In particular, $n=|E(G)|=|E(G')|$.
 Now observe that  $q_{B}$ is positive if and only if so is $q_A$ by~\eqref{eq:eulfBrauer}. And the latter is positive if and only if
 $\CRnk(q_A)=0$. On the other hand, by Theorem~\ref{thm:mrgvsCar}(b) we get that $\CRnk(q_A)=n-m+\gcc_A$, where $m=|V(G)|=|V(G')|$,  so $\CRnk(q_A)=0$ is equivalent with $n=m-\gcc_A$.
    The latter equality holds precisely when $G'$ is a tree or $G'$ is not bipartite 1-cycle graph by Theorem~\ref{thm:mrgvsCar}(a). Equivalently, (c) holds.
\end{proof}

As a consequence we obtain the following analogue of classical results for algebras of finite global dimension interrelating the definiteness of the Euler (or Tits) form with the representation type, cf.~\cite{GabrielI,DlabRingel,BPS}. Recall that Brauer graph algebras are of tame representation type. For more details as well as  brief definitions of representation types we refer to~\cite{Schroll.Brauer}.
\begin{corollary}\label{cor:BrauerDomestic}
Let $\Gamma$ be a connected Brauer graph with trivial multiplicities and let $\Bralg=\Bralg(\Brgr)$ be the associated Brauer graph algebra over an algebraically closed field $k$. Then the Euler form $q_\Bralg$ is positive if and only if $\Bralg$ is of finite or 1-domestic representation type.
\end{corollary}
\begin{proof} It is well-known that $\Bralg$ is representation-finite if and only if $\Brgr$ is a tree, see~\cite[Corollary 2.9]{Schroll.Brauer}. Moreover, from the characterization of domestic Brauer graph algebras due to Bocian-Skowro\'nski~\cite{BoSkDomestic} it follows that, in case of trivial multiplicities, $\Bralg$ is 1-domestic if and only if $\Gamma$ a 1-cycle graph and the unique cycle  has  odd length, cf.~\cite[Theorem 5.1]{Schroll.Brauer}. Hence the claim follows by Theorem~\ref{thm:Brauer}.
\end{proof}

Recall that the (non)singularity of the Cartan matrix may reflect certain interesting properties of an algebra, cf.~\cite{BSY,BessHolm2008}.
As an exemplary consequence of Theorem~\ref{thm:Brauer} we present the following observation.
\begin{corollary}\label{cor:Brauer}
Let $\Brgr$ be a connected Brauer graph with trivial multiplicities and let $\Bralg=\Bralg(\Brgr)$ be the associated Brauer graph algebra over an algebraically closed field. Assume that $\Brgr$ is a tree or $\Brgr$ is a 1-cycle graph and the unique cycle  has  odd length. Then the Auslander-Reiten quiver  of $\mod{\Bralg}$ does not admit a generalized standard tube.
\end{corollary}
\begin{proof} Recall that $\Bralg$ is a symmetric algebra. So in case the AR quiver of $\mod{\Bralg}$ contains a generalized standard tube then by~\cite[Theorem 2.1]{BSY} the Cartan matrix $\Car_\Bralg$ is singular, a contradiction with Theorem~\ref{thm:Brauer}(b).
\end{proof}

At the final stages of preparing the manuscript,  discussions with Bobi\'nski and Zwara pointed us to the following observations.

\begin{remark}\label{rem:Antipov} 
(a) Generalization of Theorem \ref{thm:Brauer} to Brauer graph algebras with arbitrary multiplicities may be retrieved from the results of Antipov \cite{Antipov}. Namely, by \cite[Proposition 1]{Antipov} the Cartan matrix $\Car_\Bralg$ of the algebra $\Bralg=\Bralg(\Gamma)$ of a Brauer graph $\Gamma=(G,\permut,\mult)$ (with $G$ connected and having at least one edge) decomposes as a sum of certain symmetric matrices, which may be reformulated using our conventions as
\def\vert{\alpha}
\def\wkol{c}
\begin{equation}\label{eq:antipov}
\Car_\Bralg=\sum_{\vert\in \verts{G}}\mult(\vert)\wkol_\vert \wkol_\vert^\tr
\end{equation}
where $\wkol_\vert:=\Inc(G)\bas_\vert$ is the $\vert$-th column of the incidence matrix $\Inc(G)$ of $G$. This shows that $\Car_\Bralg$ is positive semidefinite (thus the Euler form $q_B$ is non-negative). Moreover, it follows  that the positive definiteness of $C_B$ does not depend on the multiplicity function $\mult$ (since, by  \eqref{eq:antipov}, $\Car_\Bralg$ is positive definite if and only if for every $0\neq x\in\ZZ^{\edgs{G}}$ there exists $\alpha\in\verts{G}$ such that $x^\tr\wkol_\vert \wkol_\vert^\tr x>0$). In particular, the equivalence of points (a,b,c) in Theorem \ref{thm:Brauer} holds also in the general case. However, since the representation type of a Brauer graph algebra heavily depends on multiplicities, one cannot extend Corollary \ref{cor:BrauerDomestic} for $\Bralg$ with non-trivial $\mult$. Actually, only one implication of the corollary holds in this case: {\em if $\Bralg$ is of finite or 1-domestic representation type, then $q_\Bralg$ is positive} (for instance, if $\Brgr$ is a tree with at least three non-trivial multiplicities then $q_{\Bralg}$ is positive but $\Bralg$ is not domestic, cf.~\cite{BoSkDomestic,Schroll.Brauer}).
We note that  the characterization of Brauer graph algebras having non-singular Cartan matrices, as in Theorem \ref{thm:Brauer}(b,c),
may be deduced from \cite[Main Theorem]{Antipov}, which was proved by different techniques.

(b) 
Recall that for the algebra $\Bralg=\Bralg(\Gamma)$ of a Brauer graph $\Gamma=(G,\permut,\mult)$ with $\mult\equiv 1$ we have $\Car_\Bralg=\Inc(G)\Inc(G)^\tr$ (cf.~Theorem \ref{thm:mrgvsCar} and the proof of Theorem \ref{thm:Brauer}). The formula \eqref{eq:antipov} implies that for $\Bralg=\Bralg(\Gamma)$ with arbitrary $\mult$ we have:
\begin{equation}\label{eq:sqrtinc}
    \Car_\Bralg=\Inc(G,\mult)\Inc(G,\mult)^\tr
\end{equation}
where $\Inc(G,\mult)\in\M_{\edgs{G},\verts{G}}(\ZZ)$ is the matrix defined by $\Inc(G,\mult)\bas_\vert=\sqrt{\mult(\vert)}\Inc(G)\bas_\vert$ for each $\vert\in\verts{G}$. One may view $\Inc(G,\mult)$ as the \lq\lq weighted'' incidence matrix of $(G,\mult)$. It is interesting to note that $C_B$ does not depend on the permutation $\permut$.
\end{remark}

\section{Roots and indecomposable complexes}\label{sec:roots}

The following main result of this section is a more precise version of Theorem B from the introduction. We use the notation and the terminology of Subsections~\ref{subsec:qf} and~\ref{subsec:perfcpx}.

\begin{theorem}\label{thm:rootsind}
Let $(Q,I)$ be a gentle bound quiver and let $q=q_\La:K_0({\rm per}\,\La)\to\ZZ$ denote the  Euler form of the gentle algebra $A=kQ/I$. Then the following inclusions hold:
\begin{equation}\label{eq:rootsind}
\roots_q(0) \cup \roots_q(1)\  \subseteq  \  \perdims(A) \  \subseteq \ \roots_q(0) \cup \roots_q(1) \cup \roots_q(2),
\end{equation}
where $\perdims=\perdims(A):=[\ind({\rm per}\,\La)]\subseteq K_0({\rm per}\,\La)$ denotes the subset of all the classes of indecomposable perfect complexes in the Grothendieck group $K_0({\rm per}\,\La)$ of all perfect complexes. Moreover, if  $(G,\leq)=\GPC(\quiver,I)$ is the associated marked ribbon graph then
\begin{enumerate}[label={\textnormal{(\alph*)}},topsep=3px,parsep=1px]
\item $\roots_q(0) = \{[\stringcpx{m,\wlk}] \mid \ m\in\ZZ\ \text{\,and\,} \, \wlk\ \text{is a reduced closed walk in $G$ of even length}\}$,
\item $\roots_q(1) = \{[\stringcpx{m,\wlk}] \mid \  m\in\ZZ\ \text{\,and\,}\, \wlk\ \text{is a reduced  open walk in $G$}\}$,
\item $\perdims\cap\roots_q(2)=  \{[\stringcpx{m,\wlk}] \mid \  m\in\ZZ\ \text{\,and\,}\, \wlk\ \text{is a reduced closed walk in $G$ of odd length}\}$,
\item $\roots_q(0) \cup \roots_q(1) = \perdims$ \ if and only if\ \ $(Q,I)$ satisfies the multi-clock condition,
\item $[Y^\bullet]\in  \roots_q(0)$ for every indecomposable band complex $Y^\bullet$.
\end{enumerate}
\end{theorem}

Slightly abusing our conventions,  in this section any trivial walk $\wlk$ is considered  as a reduced closed walk of even length and in this case the associated zero complex $\stringcpx{m,\wlk}$ is treated as an indecomposable string complex. Observe that in contrast to the classical results providing a classification of indecomposable modules in terms of  positive roots of the Euler (or Tits) form (cf.~\cite{GabrielI, BongartzTits, DlabRingel}) we do not restrict in Theorem~\ref{thm:rootsind} to \lq\lq positive'' roots since we deal with objects of the derived category. The proof of the theorem will be given in Subsection~\ref{subsec:proofB}. In the next subsection we discuss some preparatory facts required for the proof.

\subsection{Incidence vectors}\label{subsec:incvects}

In~\cite[Section 6]{JM} we studied so-called incidence vectors associated to walks in an arbitrary bidirected graph. Here we recall the special case of this construction, for walks in a graph $G$ (viewed as the bidirected graph $(G,\sigma^{+})$ with constant bidirection $\sigma^+ \equiv +1$, cf.~Definition~\ref{def:bidi}).
\begin{definition}\label{def:incvect} Given a graph $G$ and a walk $\wlk=i_1\cdots i_Li_{L+1}$ of $G$, for $L\geq 0$ and $i_t\in\Oedgs{G}$, $t=1,\ldots,L+1$, we set
\begin{equation}\label{eq:defincvect}
\inc(\wlk)=\inc_G(\wlk):=\sum_{t=0}^{L}(-1)^t\bas_{i_{t+1}}\in\ZZ^{\edgs{G}}
\end{equation}
where $\bas_{i}:=\bas_{\vtx{i}}$ for $i\in \Oedgs{G}$, and we call $\inc(\wlk)$ the \textbf{incidence vector} of $\wlk$. For a trivial walk $\wlk$ we set $\inc(\wlk):=0$.
\end{definition} 

We recall from~\cite[Lemmas~6.2-6.3]{JM} the following basic properties of incidence vectors.
\begin{lemma}\label{lem:incvects}
Let $G$ be a graph  and $\wlk$, $\wlk'$ a pair of  concatenable walks in $G$. Then
\begin{enumerate}[label={\textnormal{(\alph*)}},topsep=3px,parsep=0px]
 \item $\inc(\wlk \wlk')=\inc(\wlk)+(-1)^{\ell(\wlk)}\inc(\wlk')$.
 \item $\inc(\invw{\wlk})=(-1)^{\ell(\wlk)+1}\inc(\wlk)$,
 \item if $\wlk$ is a closed walk of even length, then $\inc(\wlk^k)=k\,\inc(\wlk)$ for each $k\in\ZZ$,
\item if $\wlk$ is a closed walk of odd length,  then $\inc(\wlk^k)=(k \,{\rm mod}\, 2)\inc(\wlk)$ for each $k\in\ZZ$,
 \item $\Inc(G)^{\tr}\inc(\wlk)=\bas_{\tar(\wlk)}-(-1)^{\ell(\wlk)}\bas_{\sou(\wlk)}\in \ZZ^{\verts{G}}$, where $\Inc(G)\in\M_{\edgs{G},\verts{G}}(\ZZ)$ is the incidence matrix of $G$ as in Definition~\ref{def:bidi}.
\end{enumerate}
\end{lemma} 

The following one of the main results of~\cite{JM} is crucial in our study.
\begin{theorem}[{{\cite[Theorem C]{JM}}}]\label{thm:JM1B}
Given a connected graph $G$, the set of vectors
\begin{equation}\label{eq:JM1Bincroots}
\CRinc_G:=\{\pm\inc(\wlk) \mid \ \wlk \text{ a walk of } G\}\subseteq \Z^{\edgs{G}}
\end{equation}
satisfies
\begin{equation}\label{eq:JM1Bsubsets}
\roots_q(0) \cup \roots_q(1)\  \subseteq  \ \CRinc_G \ \subseteq \ \roots_q(0) \cup \roots_q(1) \cup \roots_q(2),
\end{equation}
where $q=q_G=q_{G,\sigma^+}:\Z^{\edgs{G}} \to \Z$ is the incidence form of $G$ as in Definition~\ref{def:incqf}. Moreover,
\begin{enumerate}[label={\textnormal{(\alph*)}},topsep=4px,parsep=0px]
\item $\roots_q(0) = \{\pm \inc(\wlk) \mid \  \, \wlk\ \text{is a closed walk of even length}\}$,
\item $\roots_q(1) = \{\pm \inc(\wlk) \mid \  \, \wlk\ \text{is an open walk}\}$,
\item $\CRinc_G\cap\roots_q(2)=  \{\pm \inc(\wlk) \mid \  \, \wlk\ \text{is a closed walk of odd length}\}$.
\end{enumerate}
\end{theorem}
Vectors from $\CRinc_G$ are called \textbf{incidence roots} of $G$ (or of $q_G$). Observe that by Lemma~\ref{lem:incvects}(a)-(b)  we may restrict to reduced walks in the definition of $\CRinc_G$, that is,
$\CRinc_G=\{\pm\inc(\wlk) \mid \ \wlk \text{ a reduced walk of } G\}$, and similarly for the sets in claims  (a)-(c) of Theorem~\ref{thm:JM1B}, cf.~\cite[Remark 7.1]{JM}. 
 The following observations provide a  bridge from incidence roots to the study of the present paper. We use the notation and terminology of Subsections~\ref{subsec:perfcpx} and~\ref{subsec:eulf}.

\begin{lemma}\label{lem:incvectsvsdim}
Let $(Q,I)$ be a gentle bound quiver and  $(G,\leq)=\GPC(\quiver,I)$  the associated marked ribbon graph. Then
\begin{enumerate}[label={\textnormal{(\alph*)}},topsep=4px,parsep=1px]
 \item $\vdimp{\stringcpx{m,\wlk}}=(-1)^m\inc_G(\wlk)$ for each  $m\in\ZZ$ and a non-trivial reduced walk $\wlk$ of $G$,
\item  $\vdimp{\stringcpx{m,\wlk,\mu}}=(-1)^md\,\inc_G(\wlk_{[\ell(\wlk)-1]})$ for each $m \in \Z$, a belt $\wlk$  of $(G,\leq)$ and  an indecomposable automorphism $\mu$ of a $k$-vector space $K$ of dimension $d={\rm dim}_k(K)\geq 1$.
\end{enumerate}
\end{lemma}
\begin{proof} 
 We show (a) by induction on $L\geq 0$, for reduced walks $\wlk=i_1\cdots i_Li_{L+1}$. If $L=0$ then $\wlk=i_1$ and
$\stringcpx{m,\wlk}=\ldots 0\to P_{i_1}\to 0\ldots$ is a stalk complex concentrated in cohomological degree $m$, so $\vdimp{\stringcpx{m,\wlk}}=(-1)^m\bas_{i_1}$, cf.~the formula~\eqref{eq:alternating}. Since $\inc(i_1)=\bas_{i_1}$ we get the claim. Now, given a reduced walk $\wlk'=i_1\cdots i_Li_{L+1}i_{L+2}$, by construction in Subsection~\ref{subsec:perfcpx} it is clear that $\vdimp{\stringcpx{m,\wlk'}}=\vdimp{\stringcpx{m,\wlk}}+(-1)^{s+1}\bas_{i_{L+2}}$, independently on the degree $\deg(i_{L+1}i_{L+2})$, where $\wlk=i_1\cdots i_Li_{L+1}$ and  $s$ is the cohomological degree of the summand $P_{i_{L+1}}$  in $\stringcpx{m,\wlk}$, cf.~\eqref{eq:alternating}. Thus by the inductive assumption we get that $(-1)^s=(-1)^{m+L}$ and $$\vdimp{\stringcpx{m,\wlk'}}=(-1)^m\inc(\wlk)+(-1)^{m+L+1}\bas_{i_{L+2}}
=(-1)^m\left(\sum_{t=0}^{L}(-1)^t\bas_{i_{t+1}}+(-1)^{L+1}\bas_{i_{L+2}}\right)=(-1)^m\inc(\wlk'),$$
which shows (a). Now take  $m\in\ZZ$, a belt $\wlk$ and an automorphism $\mu:K\to K$ as in (b). Then directly from Proposition~\ref{pro:walkBelt} it follows that $\vdimp{\stringcpx{m,\wlk,\mu}}=d\,\vdimp{\stringcpx{m,\wlk_{[\ell(\wlk)-1]}}}$  for $d={\rm dim}_k(K)$. Thus by (a) we get $\vdimp{\stringcpx{m,\wlk,\mu}}=(-1)^md\,\inc_G(\wlk_{[\ell(\wlk)-1]})$.
\end{proof}

\subsection{{\it Proof of Theorem~\ref{thm:rootsind}}}\label{subsec:proofB} 
Fix  a gentle bound quiver $(Q,I)$, and let  $(G,\leq)=\GPC(\quiver,I)$ be the associated marked ribbon graph  and $q=q_\La:K_0({\rm per}\,\La)\to\ZZ$, the Euler form  of  the gentle algebra $\La:=kQ/I$. Recall that by Theorem~\ref{thm:mrgvsCar} the  form $q_A$ in the basis of indecomposable projectives $\projbas:=\{[P_i]\}_{i\in Q_0}$ is the incidence form $q_G:\Z^{\edgs{G}} \to \Z$ of $G$  (see Definition~\ref{def:incqf}). In particular, 
\begin{equation}\label{eq:proofBqAqG}
q_A([X^\bullet])=q_G(\vdimp{X^\bullet})
\end{equation}
for every perfect complex $X^\bullet$ (the set of edges $\edgs{G}$ of $G$ is identified with the set of vertices $Q_0$ of $Q$).
By Proposition~\ref{pro:walkBelt} the set $\perdims=\perdims(A)=[\ind({\rm per}\,\La)]\subseteq K_0({\rm per}\,\La)$ is the sum
\begin{equation}\label{eq:proofBPsplit}
\perdims=\perdims_{\rm st}\cup\perdims_{\rm bd}\quad\text{for}\quad
\begin{array}{l}\perdims_{\rm st}:=\{[\stringcpx{m,\wlk}] \mid \ m\in\ZZ\ \text{\,and\,} \, \wlk\ \text{is a reduced walk in $G$}\,\},\medskip\\
\perdims_{\rm bd}:=\{[\stringcpx{m,\wlk,\mu}] \mid \ m\in\ZZ, \, \wlk\ \text{is
a belt of $(G,\leq)$ and $\mu$ is an indecomposable}\\
\,\hspace{2.5cm}\text{ automorphism of a finite dimensional $k$-vector space $K$}\},
\end{array}
\end{equation}
 cf.~\eqref{eq:perstbd}. By transferring the inclusions~\eqref{eq:JM1Bsubsets} of Theorem~\ref{thm:JM1B} via~\eqref{eq:proofBqAqG} and Lemma~\ref{lem:incvectsvsdim}(a) we obtain  
 \begin{equation}\label{eq:proofBincluPst}
\roots_{q_A}(0) \cup \roots_{q_A}(1)\  \subseteq  \  \perdims_{\rm st} \  \subseteq \ \roots_{q_A}(0) \cup \roots_{q_A}(1) \cup \roots_{q_A}(2),
\end{equation}
and in the same way the equalities  (a) and (b) of Theorem~\ref{thm:rootsind} follow by Theorem~\ref{thm:JM1B}(a)-(b) (note that $[\stringcpx{m+1,\wlk}]=-[\stringcpx{m,\wlk}]$).
To show the remaining claims take an indecomposable band complex $Y^\bullet$. By Proposition~\ref{pro:walkBelt} there exists $m\in\ZZ$,
a belt $\wlk=i_1i_2\cdots i_Li_{L+1}$ of $(G,\leq)$ and an indecomposable automorphism $\mu$ of a finite dimensional $k$-vector space $K$ such that $Y^\bullet\cong \stringcpx{m,\wlk,\mu}$. By the definition of
 a belt it follows that $\wlk_1:=\wlk_{[L]}=i_1i_2\cdots i_L$ is a closed walk and $\deg(\wlk)=0$. The latter condition implies that $L=\ell(\wlk_1)>0$ is even. 
 Hence by Lemma~\ref{lem:incvectsvsdim} and Lemma~\ref{lem:incvects}(c) we get that
 \begin{equation}\label{eq:proofBbeltvsstring}
 \vdimp{Y^\bullet}=\vdimp{\stringcpx{m,\wlk,\mu}}=(-1)^md\,\inc_G(\wlk_1)=(-1)^m\,\inc_G(\wlk_1^d)
 =\vdimp{\stringcpx{m,\wlk_1^d}}
 \end{equation}
 where $d={\rm dim}_k(K)$. In particular, $[Y^\bullet]=[\stringcpx{m,\wlk_1^d}]$ belongs to $\roots_{q_A}(0)$ by (a), since $\wlk_1^d$ is a closed walk of even length (note that $\wlk_1^d$ is reduced since $i_Li_1=i_Li_{L+1}$ is reduced by the definition of a belt). This shows claim (e). This also shows that $\perdims_{\rm bd}\subseteq \perdims_{\rm st}$ hence $\perdims=\perdims_{\rm st}$. Thus by applying~\eqref{eq:proofBincluPst} we obtain the inclusions~\eqref{eq:rootsind}. Similarly, we get claim (c) by applying Theorem~\ref{thm:JM1B}(c) via Lemma~\ref{lem:incvectsvsdim}(a). 
To see the remaining claim (d) observe that by~\eqref{eq:rootsind}, the equality $\roots_q(0) \cup \roots_q(1) = \perdims$ is equivalent with $\perdims\cap\roots_q(2)=\emptyset$. By (c), the latter holds exactly if $G$ has no closed walks of odd length, that is, if $G$ is bipartite. And this holds if and only if  $(Q,I)$ satisfies the multi-clock condition by Theorem~\ref{thm:mrgvsCar}(a). 
\hfill$\Box$

\subsection{Root systems and applications}\label{subs:rootsys}

In this subsection we show some exemplary applications and consequences of Theorem~\ref{thm:rootsind}.

\begin{corollary}\label{cor:2roots}
Let $(Q,I)$  be a gentle bound quiver and let $q=q_\La:K_0({\rm per}\,\La)\to\ZZ$ be the Euler form of the gentle algebra $A=kQ/I$. Let $\perdims=\perdims(A)=[\ind({\rm per}\,\La)]\subseteq K_0({\rm per}\,\La)$ be the subset  as in Theorem~\ref{thm:rootsind}.
Then for any $x\in \perdims\cap \roots_q(2)$ there exist infinitely many $($up to shift and isomorphism$)$ indecomposable string complexes $X^\bullet$ in $\Db{A}$ such that $[X^\bullet]=x$. In particular, if $(Q,I)$ does not satisfy the multi-clock condition, then $A$ is of derived infinite type.
\end{corollary}
\begin{proof} 
 Take $x\in \perdims\cap \roots_q(2)$. By Theorem~\ref{thm:rootsind}(c) there exists $m\in\ZZ$ and a reduced closed walk $\wlk$ in $G$ of odd length such that $x=[\stringcpx{m,\wlk}]$, where $(G,\leq)=\GPC(\quiver,I)$. So by Lemma~\ref{lem:incvectsvsdim}(a) and Lemma~\ref{lem:incvects}(d) for every odd $t\in\ZZ$ we have
$$
\vdimp{\stringcpx{m,\wlk^t}}=(-1)^m\inc_G(\wlk^t)=(-1)^m\inc_G(\wlk)=\vdimp{\stringcpx{m,\wlk}},
$$
thus $[\stringcpx{m,\wlk^t}]=[\stringcpx{m,\wlk}]=x$. Observe that the string complexes $\{\stringcpx{m,\wlk^t}\}_{t\in\ZZ}$ are pairwise not isomorphic (even up to shift), see~\eqref{eq:stnoniso}, cf.~\cite{BekMer}. Thus the first  assertion holds. Now, if  $(Q,I)$ does not satisfy the multi-clock condition then by Theorem~\ref{thm:rootsind}(d) the set $\perdims\cap \roots_q(2)$ is non-empty. And by the arguments above, every $x\in\perdims\cap \roots_q(2)$ induces an infinite family of pairwise not isomorphic (up to shift) indecomposable complexes in $\Db{A}$.
\end{proof}

\begin{example}\label{ex:extremal} 
Let $(Q,I)$  be a gentle bound quiver satisfying the equivalent properties of Corollary~\ref{cor:limits}(d). In particular, the associated marked ribbon graph $(G,\leq)=\GPC(\quiver,I)$ has one vertex and $n:=|Q_0|=|\edgs{G}|\geq 1$ loops, and $(Q,I)$ does not satisfy the multi-clock condition. Observe that the condition $|\verts{G}|=1$ is precisely the case when the Euler form $q=q_\La:K_0({\rm per}\,\La)\to\ZZ$ of $A=kQ/I$ is not irreducible, and in this case $\Dyn(\frac{1}{2}q)=\AA_1$, see Theorem~\ref{thm:mrgvsCar}(c), cf.~\eqref{eq:onevertex}. Therefore the set $\roots_{q}(1)\subseteq K_0({\rm per}\,\La)$ of $1$-roots of $q$ is empty. Alternatively, this follows by Theorem~\ref{thm:rootsind}(b) since  there are no open walks in $G$. Moreover, the sets $\roots_{q}(0)$ and $\perdims(A)\cap\roots_q(2)$ are non-empty by Theorem~\ref{thm:rootsind}(a,c).

Let us consider the case $n=1$ in more detail. 
In this case $(Q,I)$ is the one loop gentle bound quiver {\xymatrix@C=-1pt{{}_1 \ar@(dl,ul)^-{\alpha_1}="a" & {}
\ar@{.}@/^3pt/"a"!<-6pt,-5pt>;"a"!<-6pt,5pt> }} of Example~\ref{ex:small}, $A=kQ/I$ is the algebra of \lq\lq dual numbers'' (cf.~\cite{VossieckDerDisc}),  and $G={\xymatrix{  {{\alpha}} \ar@{-}@(ld,lu)^(.5){1}}}$. Note that  the Grothendieck group $K_0({\rm per}\,\La)$ has rank 1 and it is generated by the class $[P_1]$ of the only indecomposable projective $A$-module $P_1=A$.
Moreover, every non-trivial reduced walk of $G$ is the  $\ell$-fold concatenation $1^\ell=11\ldots 1$ or $1^{-\ell}:=(1^{-1})^\ell=1^{-1}1^{-1}\ldots 1^{-1}$, for $\ell\geq 1$ (here without loss of generality we treat the loop $1$ as an ordered edge such that $\deg(11)=1$, cf.~Subsection~\ref{subsec:perfcpx}). Observe that $\deg(1^\ell)=\ell-1$ and $\deg(1^{-\ell})=1-\ell$ for every $\ell\geq 1$, hence $G$ has  no belts. 
 Therefore by Proposition~\ref{pro:walkBelt}, up to isomorphism every (non-zero) indecomposable perfect complex is the string complex $\stringcpx{m,1^\ell}=\stringcpx{m+l-1,1^{-\ell}}$ of the form
 \begin{equation*}\label{eq:dualnumberscpx}
\xymatrix{...\ar[r]&0 \ar[r]& P_{1} \ar[r]^{\alpha_1} & P_{1} \ar[r]^{\alpha_1} & \cdots \ar[r]^{\alpha_1} & P_{1} \ar[r] \ar[r]&0\ar[r]&...}
\end{equation*}
 with $\ell$ copies of $P_1$ starting at cohomological degree $m$, for some $m\in\ZZ$ and $\ell\geq 1$. To get the non-perfect complexes observe that the walk $1$ is a full face and $11$ is the unique primitive left resolvable walk of $(G,\leq)$. Thus by Proposition \ref{prop:nonper} every indecomposable non-perfect complex has the form $\stringcpx{m,-,11}$ which is the shifted (infinite) projective resolution of the simple module $S_1$. In this way we easily recovered the known
classification of indecomposable objects in $\Db{A}$, cf.~\cite{VossieckDerDisc},  see also Example \ref{ex:oneloopAR}. Moreover, since $1^\ell$ is a closed walk, by Theorem~\ref{thm:rootsind}(a,c) we infer that $[\stringcpx{m,1^\ell}]=0\in \roots_q(0)$ if $\ell$ is even, and $[\stringcpx{m,1^\ell}]=(-1)^m[P_1]\in \roots_q(2)$ if $\ell$ is odd,  cf.~\eqref{eq:alternating}. 
 Note that here $\roots_q(2)=\{[P_1],-[P_1]\}$ since the Cartan matrix has the form $C_A=[2]\in\MM_{1,1}(\ZZ)$. Thus $\perdims(A)=\roots_q(0)\cup\roots_q(1)\cup \roots_q(2)=\roots_q(0)\cup\roots_q(2)=\{0,[P_1],-[P_1]\}$, cf.~\eqref{eq:rootsind}.
\end{example}

In the theorem below we show that in case the Euler form of a gentle algebra $A$ is positive, the set $\perdims=[\ind({\rm per}\,\La)]$ is finite and it carries an additional structure of  a root system in the sense of Bourbaki~\cite{Bourbaki}. Recall that a root system is  a finite system of vectors in a Euclidean space satisfying axioms defined in~\cite[III.9]{HumphreysLie}. Root systems  are classified via their \textbf{Cartan-Dynkin types} $\AA_n, \BB_n,\CC_n,\DD_n,\EE_6,\EE_7,\EE_8,$ $\FF_4,\GG_2$, see~\cite[Chapter III]{HumphreysLie} for the details. Recall also that this classification played a fundamental role in the well-known classification of complex semisimple Lie algebras, see~\cite{Bourbaki,HumphreysLie}.

\begin{theorem}\label{thm:rootsys}
Let $(Q,I)$  be a gentle bound quiver and let $q=q_\La:K_0({\rm per}\,\La)\to\ZZ$ be the  Euler form of the gentle algebra $A=kQ/I$. Let $n:=|Q_0|\geq 1$ and let $\perdims=\perdims(A)=[\ind({\rm per}\,\La)]\subseteq K_0({\rm per}\,\La)$ be the subset  as in Theorem~\ref{thm:rootsind}.
 If $q_\La$ is positive then  the set $\perdims^\star:=\perdims\setminus\{0\}$ is an irreducible $($finite$)$ root system in the Euclidean space $(K_0({\rm per}\,\La)\otimes \R, q(-,-))$. Moreover,
     \begin{enumerate}[label={\textnormal{(\alph*)}},topsep=4px,parsep=1px]
\item if $(Q,I)$ satisfies the multi-clock condition then the root system $\perdims^\star$  has $n^2+n$ elements and its Cartan-Dynkin type is $\AA_n$,
 \item if $(Q,I)$ does not satisfy the multi-clock condition then the root system $\perdims^\star$  has $2n^2$ elements and its Cartan-Dynkin type is $\CC_n$ $($we set $\CC_1:=\AA_1$ here$)$;
moreover, the set $\roots_q(1)\subsetneq\perdims^\star$ has $2(n^2-n)$ elements and for $n\geq 4$ it is an irreducible root system of the Cartan-Dynkin type $\DD_n$.
    \end{enumerate}
\end{theorem}

\begin{proof} Let $(G,\leq)=\GPC(\quiver,I)$ be the associated marked ribbon graph. Assume that $q_\La$ is positive, equivalently, that $\CRnk(q_A)=0$. Recall  that $\CRnk(q_A)=|\edgs{G}|-|\verts{G}|+\gcc_\La\ =\ |Q_1|-|Q_0|+\gcc_\La$ and $|\edgs{G}|=|Q_0|$, see Theorem~\ref{thm:mrgvsCar}(b) and~\eqref{eq:edgesverticesThmA}. Thus positivity of $q_\La$  implies that one of the following conditions holds (cf.~the arguments in the proof of Theorem~\ref{thm:Brauer}): 
\begin{enumerate}[label={\textnormal{(\roman*)}},topsep=4px,parsep=1px]
\item $G$ is a tree (equivalently, $Q$ is a tree),
\item $G$ is not bipartite 1-cycle graph (equivalently, $(Q,I)$ is 1-cycle not satisfying the multi-clock condition),
\end{enumerate}
cf.~Theorem~\ref{thm:mrgvsCar}(a) and Corollary~\ref{cor:limits}(a).

Assume first that $n=|\edgs{G}|=1$. 
Since by our conventions $G$ has a vertex of degree at least two (see Definitions~\ref{def:mrg} and \ref{def:BoundQuiverInc}), the assumption $n=1$ implies that the case (ii) holds and  $G$ consists of one vertex and one loop. Thus $(Q,I)$ is a one loop gentle bound quiver. As we showed in  Example~\ref{ex:extremal}, in this case $\roots_{q_A}(1)=\emptyset$ and $\perdims^\star=\roots_{q_A}(2)=\{[P_1],-[P_1]\}\subseteq K_0({\rm per}\,\La)\cong \ZZ[P_1]$ hence $\perdims^\star$ is an irreducible root system in $K_0({\rm per}\,\La)\otimes \R\cong \R$ of Cartan-Dynkin type $\AA_1(=\CC_1)$, see~\cite[9.3]{HumphreysLie}.

Now assume that $n=|\edgs{G}|\geq 2$. For this  case we may apply~\cite[Proposition 7.15 and Remark 7.16]{JM}, where   by a direct verification of the root system axioms and a careful counting of walks in $G$ it was shown that the incidence roots of the incidence form $q_G$ of $G$ satisfy the following properties.
\begin{itemize}[topsep=4px,parsep=1px]
\item In both cases (i) and (ii), the set $\roots^\star:=\CRinc_G\setminus\{0\}$ is an irreducible root system in the Euclidean space $(\R^n, q_G(-,-))$.
\item In case (i) we have $|\roots^\star|=n^2+n$ and $\CDyn(\roots^\star)=\AA_n$.
\item In case (ii) we have $|\roots^\star|=2n^2$ and $\CDyn(\roots^\star)=\CC_n$.
\item In case (ii) we have $|\roots_{q_G}(1)|=2(n^2-n)$ and $\roots_{q_G}(1)\subsetneq \roots^\star$. Moreover,   the set $\roots_{q_G}(1)$ is an irreducible root system with  $\CDyn(\roots_{q_G}(1))=\DD_n$, provided $n\geq 4$. 
\end{itemize}
(Here  $\CDyn(-)$ denotes  the Cartan-Dynkin type of a root system.) Now recall that by Theorem~\ref{thm:mrgvsCar} the Euler form $q_A$ in the   basis of indecomposable projectives $\projbas:=\{[P_i]\}_{i\in Q_0}$ is the incidence form $q_G:\Z^{\edgs{G}} \to \Z$ of $G$, and the set $\perdims=\perdims_{\rm st}\subseteq K_0({\rm per}\,\La)$ corresponds bijectively via
$\vdimp{-}$ to the set $\CRinc_G\subseteq \ZZ^n$ (see Lemma~\ref{lem:incvectsvsdim}(a) and the proof of Theorem~\ref{thm:rootsind}). Thus  the properties of the root systems $\roots^\star$ and $\roots_{q_G}(1)$ listed above transfer to the sets $\perdims^\star$ and  $\roots_{q_A}(1)$, respectively, and in this way we get the claims (a) and (b).
\end{proof}

\begin{remark}\label{rem:longroots}
Let $(Q,I)$ be a gentle bound quiver $(Q,I)$  with positive Euler form $q=q_A$.

(a) Since by Theorem~\ref{thm:rootsys}, the set $\perdims^\star\subseteq K_0({\rm per}\,\La)$ is an irreducible root system, there is a simple complex Lie algebra having $\perdims^\star$ as its root system, cf.~\cite[18.4]{HumphreysLie}. It is interesting to ask what happens if $q$ is not positive. In this case the set $\perdims^\star$ is infinite and it seems to us that (at least in some cases) $\perdims^\star$ or its subsets may be viewed as infinite root systems like these related to Kac-Moody Lie algebras, cf.~\cite{Kac,DlabRingel, Kac80}, compare also with recent results on Lie algebra structures related with gentle algebras  in \cite{Chaparro2018OnTL,CY23}. 

\smallskip

(b) If $n=|Q_0|\geq 4$ and  $(Q,I)$ does not satisfy the multi-clock condition, then by Theorem~\ref{thm:rootsys}(b) there are two root systems: $\roots_q(1)\subsetneq\perdims^\star$ and $\perdims^\star$ in $(K_0({\rm per}\,\La)\otimes \R, q(-,-))$ of Cartan-Dynkin types $\DD_n$ and $\CC_n$, respectively. Observe that the set $\perdims':=\perdims^\star\setminus \roots_q(1)$ has the cardinality $2n$ and it is contained in $\roots_q(2)$ (cf.~Theorem~\ref{thm:rootsind}). Thus, using the terminology of~\cite{HumphreysLie},  the elements of $\roots_q(1)$ (resp.~of $\perdims'$) are \lq\lq short'' (resp.~\lq\lq long'') roots in the root system $\perdims^\star$. For more detailed discussion on  root systems induced by  positive Cox-regular quadratic forms (of arbitrary Dynkin type) and their interrelations with Weyl groups and Weyl roots we refer to~\cite{MZaj.LAA2022}, cf.~\cite{Roiter}.
\end{remark}

\section{Auslander-Reiten triangles and the Coxeter transformation}\label{sec:AR}

Let $\La$ be a finite dimensional $k$-algebra with bounded derived category $\Db{\La}$. It was shown in~\cite{HDb3} that if $Z^\bullet$ is an indecomposable complex in $\Db{\La}$ then there is an Auslander-Reiten triangle $X^\bullet \to Y^\bullet \to Z^\bullet \to X^\bullet[1]$ in $\Db{\La}$ if and only if $Z^\bullet$ is a perfect complex, in which case $X^\bullet$ is indecomposable and the triangle is unique up to isomorphism, see also~\cite{HDb2}. Moreover, if $\La$ is a Gorenstein algebra, then the whole triangle is Auslander-Reiten in the category $\per\,\La$, see, for instance,~\cite{HaKeRe08}. Since gentle algebras are Gorenstein (see~\cite{GeiRei05} and~\cite{APS23}), using Proposition~\ref{pro:walkBelt} and following~\cite{Bob} we produce all Auslander-Reiten triangles in the derived category  of a gentle algebra in terms of walks of the associated marked ribbon graph (Proposition~\ref{thm:ARtriangle} and Remark \ref{rem:beltARtriangle}). Afterwards, we apply this construction to describe the Coxeter transformation (Corollary \ref{cor:CoxTra}) and the Coxeter polynomial of a gentle algebra (Theorem \ref{cor:CoxPol}, which is Theorem C from the introduction).

The main construction of  Auslander-Reiten triangles involves the concatenation of walks followed by an appropriate ``reduction'', which we briefly analyze below.

\subsection{Reduced concatenation}\label{subsec:redCon}

For any concatenable walks $\wlk_1,\wlk_2$ of a graph $G$ there is a maximal $\ell \geq 0$ such that $\wlk_1=\wlk'_1\wlk$ and $\wlk_2=\invw{\wlk} \wlk''_2$ for a walk $\wlk$ of length $\ell$. Note that $\sou(\wlk'_1)=\tar(\wlk)=\sou(\invw{\wlk})=\tar(\wlk''_2)$, and that if $\wlk_1$ and $\wlk_2$ are reduced walks then $\wlk'_1\wlk''_2$ is reduced by the maximality of $\ell$. In this situation, in what follows we use the notation,
\[
\ell(\wlk_1,\wlk_2):=\ell \qquad \text{and} \qquad \wlk_1 \cdot \wlk_2 := \wlk'_1\wlk''_2,
\]
and call $\wlk_1 \cdot \wlk_2$ the \textbf{reduced concatenation} of the (reduced) walks $\wlk_1$ and $\wlk_2$.

\begin{lemma}\label{lem:associative}
If $\wlk_1$, $\wlk_2$, $\wlk_3$ are concatenable reduced walks of a graph $G$ then $\wlk_1 \cdot (\wlk_2 \cdot \wlk_3) = (\wlk_1 \cdot \wlk_2) \cdot \wlk_3$.
\end{lemma}
\begin{proof}
We distinguish three cases taking $\ell:=\ell(\wlk_1,\wlk_2)$ and $\ell':=\ell(\wlk_2,\wlk_3)$:

\smallskip

\noindent \textit{Case 1.} Assume that $\ell + \ell' < \ell(\wlk_2)$ and take the following decompositions with $\ell=\ell(\wlk''_1)=\ell(\wlk'_2)$ and $\ell'=\ell(\wlk'''_2)=\ell(\wlk'_3)$,
\[
\xymatrix{ { \Bigl( } \ar@{.}@<-1.75ex>[rrr]_-{\wlk_1} \ar@{-}[rr]^-{\wlk'_1} & & { \am{[} } \ar@{-}[r]^-{\wlk''_1} & { \Bigr)  \Bigl(} \ar@{.}@<-1.75ex>[rrr]_-{\wlk_2} \ar@{-}[r]^-{\wlk'_2} & {\am{]}} \ar@{-}[r]^-{\wlk''_2} & {\am{(}} \ar@{-}[r]^-{\wlk'''_2} & { \Bigr)  \Bigl(} \ar@{.}@<-1.75ex>[rrr]_-{\wlk_3} \ar@{-}[r]^-{\wlk'_3} & {\am{)}} \ar@{-}[rr]^-{\wlk''_3} & & { \Bigr)}. }
\]
Then $\wlk_1 \cdot (\wlk_2 \cdot \wlk_3)=(\wlk'_1\wlk''_1) \cdot (\wlk'_2\wlk''_2\wlk''_3)=(\wlk'_1\wlk''_2\wlk''_3)$ and $(\wlk_1 \cdot \wlk_2) \cdot \wlk_3=(\wlk'_1\wlk''_2\wlk'''_2) \cdot (\wlk'_3\wlk''_3)=(\wlk'_1\wlk''_2\wlk''_3)$.

\smallskip

\noindent \textit{Case 2.} Assume that $\ell + \ell' = \ell(\wlk_2)$ and take the following decompositions with $\ell=\ell(\wlk'''_1)=\ell(\wlk'_2)$ and $\ell'=\ell(\wlk''_2)=\ell(\wlk'_3)$,
\[
\xymatrix{ { \Bigl( } \ar@{.}@<-1.75ex>[rrr]_-{\wlk_1} \ar@{-}[r]^-{\wlk'_1} & {} \ar@{-}[r]^-{\wlk''_1} & {\am{[}} \ar@{-}[r]^-{\wlk'''_1} & { \Bigr)  \Bigl(} \ar@{.}@<-1.75ex>[rr]_-{\wlk_2} \ar@{-}[r]^-{\wlk'_2} & {\am{](}} \ar@{-}[r]^-{\wlk''_2} & { \Bigr)  \Bigl(} \ar@{.}@<-1.75ex>[rrr]_-{\wlk_3} \ar@{-}[r]^-{\wlk'_3} & {\am{)}} \ar@{-}[r]^-{\wlk''_3} & {} \ar@{-}[r]^-{\wlk'''_3} & { \Bigr)}. }
\]
Note that the pair of walks $\wlk':=\wlk'_1\wlk''_1$ and $\wlk'':=\wlk''_3\wlk'''_3$ is concatenable, decomposed in such a way that $\ell(\wlk',\wlk'')=\ell(\wlk''_1)=\ell(\wlk''_3)$, that is, $\wlk''_3=\invw{(\wlk''_1)}$ and $\wlk'_1\wlk'''_3$ is a reduced walk. Then we have $\wlk_1 \cdot (\wlk_2 \cdot \wlk_3)=(\wlk'_1\wlk''_1\wlk'''_1) \cdot (\wlk'_2\wlk''_3\wlk'''_3)=\wlk'_1\wlk'''_3$ since $\wlk''_1\wlk'''_1=\invw{(\wlk'_2\wlk''_3)}$, and $(\wlk_1 \cdot \wlk_2) \cdot \wlk_3=(\wlk'_1\wlk''_1\wlk''_2) \cdot (\wlk'_3\wlk''_3\wlk'''_3) =\wlk'_1\wlk'''_3$ since $\wlk''_1\wlk''_2=\invw{(\wlk'_3\wlk''_3)}$.

\smallskip

\noindent \textit{Case 3.} Assume that $\ell + \ell' > \ell(\wlk_2)$ and take the next decompositions with $\ell=\ell(\wlk''_1\wlk'''_1)=\ell(\wlk'_2\wlk''_2)$ and $\ell'=\ell(\wlk''_2\wlk'''_2)=\ell(\wlk'_3\wlk''_3)$,
\[
\xymatrix{ { \Bigl( } \ar@{.}@<-1.75ex>[rrr]_-{\wlk_1} \ar@{-}[r]^-{\wlk'_1} & {\am{[}} \ar@{-}[r]^-{\wlk''_1} & {} \ar@{-}[r]^-{\wlk'''_1} & { \Bigr)  \Bigl(} \ar@{.}@<-1.75ex>[rrr]_-{\wlk_2} \ar@{-}[r]^-{\wlk'_2} & {\am{(}} \ar@{-}[r]^-{\wlk''_2} & {\am{]}} \ar@{-}[r]^-{\wlk'''_2} & { \Bigr)  \Bigl(} \ar@{.}@<-1.75ex>[rrr]_-{\wlk_3} \ar@{-}[r]^-{\wlk'_3} & {} \ar@{-}[r]^-{\wlk''_3} & {\am{)}} \ar@{-}[r]^-{\wlk'''_3} & { \Bigr)}. }
\]
We take the decompositions $\wlk''_1\wlk'''_1$ and $\wlk'_3\wlk''_3$ satisfying $\ell(\wlk''_1)=\ell(\wlk''_2)=\ell(\wlk''_3)$, which implies that $\wlk''_1=\invw{(\wlk''_2)}=\wlk''_3$. 
Similarly, it follows that $\wlk_1'''=(\wlk_2')^{-1}$ and $\wlk_2'''=(\wlk_3')^{-1}$.
Moreover, observe that the compositions $\wlk''_1\wlk'''_3$ and $\wlk'_1\wlk''_3$ are reduced, since $\wlk''_1=\wlk''_3$ and the walks $\wlk_1$ and $\wlk_3$ are reduced. Then we have
\[
\wlk_1 \cdot (\wlk_2 \cdot \wlk_3)=(\wlk'_1\wlk''_1\wlk'''_1) \cdot (\wlk'_2\wlk'''_3)=\wlk'_1\wlk''_1\wlk'''_3 \qquad \text{and} \qquad
(\wlk_1 \cdot \wlk_2) \cdot \wlk_3=(\wlk'_1\wlk'''_2) \cdot (\wlk'_3\wlk''_3\wlk'''_3)=\wlk'_1\wlk''_3\wlk'''_3.
\]
In every case $\wlk_1 \cdot (\wlk_2 \cdot \wlk_3)=(\wlk_1 \cdot \wlk_2) \cdot \wlk_3$, as claimed.
\end{proof}

\def\stringCpx#1{\widehat{X}^\bullet_{#1}}

Recall that for a marked ribbon graph $(G,\leq)$ we denote by $\OT{\alpha}$ the (unique) anti-walk with target vertex $\alpha$, see Subsection~\ref{subsec:antiwalks}, and take for simplicity $\TO{\alpha}:=\invw{(\OT{\alpha})}$. For a reduced walk $\wlk$ of $G$ with $\sou(\wlk)=\alpha$ and $\tar(\wlk)=\beta$ we take
\begin{equation}\label{eq:DEFplus}
{}_+\wlk := \TO{\beta} \cdot \wlk, \qquad
\wlk_+ := \wlk \cdot \OT{\alpha}, \qquad \text{and} \qquad {}_+\wlk_+ := ({}_+\wlk)_+={}_+(\wlk_+)=\TO{\beta} \cdot \wlk \cdot \OT{\alpha}.
\end{equation}
Observe that ${}_+\wlk_+$ is well-defined by Lemma~\ref{lem:associative}. Compare construction~\eqref{eq:DEFplus} with the analogous construction at the level of homotopy strings as in~\cite[pp.~652-3]{Bob}. We may say that \eqref{eq:DEFplus} unifies the several cases considered in \cite{Bob} into one formula.

\begin{lemma} \label{lem:plusTrivial}
Let $(G,\leq)$ be a marked ribbon graph, and $\wlk$ be a non-trivial reduced walk of $G$. Then ${}_+\wlk_+$ is a non-trivial reduced walk, and at most one of the walks ${}_+\wlk$ or $\wlk_+$ is trivial.
\end{lemma}
\begin{proof} 
Observe that a reduced concatenation $\wlk \cdot \wlk'$ is trivial if and only if $\wlk'=\invw{\wlk}$. Moreover, we have $\sou(\wlk \cdot \wlk')=\sou(\wlk')$ and $\tar(\wlk \cdot \wlk')=\tar(\wlk)$. Take $\sou(\wlk)=\alpha$ and $\tar(\wlk)=\beta$. Supposing that ${}_+\wlk_+$ is trivial, since ${}_+\wlk_+=\TO{\beta} \cdot (\wlk \cdot \OT{\alpha})$, we have $\OT{\beta}=\wlk \cdot \OT{\alpha}$. In particular, $\sou(\OT{\beta})=\sou(\OT{\alpha})$, and by the uniqueness of the anti-walk starting at a vertex (cf.~\eqref{eq:bijections}) we have $\OT{\beta}= \OT{\alpha}$. Since $\OT{\alpha}=\wlk \cdot \OT{\alpha}$, then necessarily $\wlk$ is trivial, a contradiction. Therefore, ${}_+\wlk_+$ is non-trivial. Supposing that both walks ${}_+\wlk$ and $\wlk_+$ are trivial, we have $\wlk=\TO{\alpha}=\OT{\beta}$. That is, $\wlk$ is at the same time an anti-walk and the inverse of an anti-walk, which is impossible.
\end{proof}

Note that either ${}_+\wlk$ or $\wlk_+$ may be a trivial walk, and that this happens exactly when $\wlk$ or its inverse is an anti-walk (compare to~\cite[Corollary~6.3]{Bob}). Recall that we treat $\stringcpx{m,\wlk}$ as the zero complex for any $m\in\ZZ$ and a trivial walk $\wlk$, cf.~Subsection~\ref{subsec:perfcpx}.

\subsection{Auslander-Reiten triangles}\label{subsec:ARtriang}

The following result describes, up to isomorphism, all almost split (or Auslander-Reiten) triangles of string complexes in the derived category of a gentle algebra, cf.~\cite{Bob, HDb3}. It may be viewed as a translation (and simplification, see \eqref{eq:DEFplus}) of the construction of Bobi\'nski \cite{Bob} into the language of walks in the marked ribbon graph.

\begin{proposition}\label{thm:ARtriangle}
 Let $(Q,I)$ be a gentle bound quiver with associated marked ribbon graph $(G,\leq)$.  Then for any non-trivial reduced walk $\Wlk$ of $G$ and any $m \in \Z$, there is an almost split triangle in $\per\,k\quiver/I$ of the form
\begin{equation}\label{eq:main64}
\xymatrix{ \stringcpx{m,\Wlk}  \ar[r] & \stringcpx{m+m(\Wlk),{}_+\Wlk} \oplus \stringcpx{m,\Wlk_+} \ar[r] & \stringcpx{m+m(\Wlk),{}_+\Wlk_+}  \ar[r] & \stringcpx{m-1,\Wlk} }
\end{equation}
where $m(\Wlk) := \ell(\TO{\beta})-2$ for $\beta:=\tar(\Wlk)$, cf.~\eqref{eq:DEFplus}.
\end{proposition} 

\begin{proof}
Take $\Wlk=i_1\cdots i_{L+1}$ ($L \geq 0$) with $\sou(\wt{\wlk})=\alpha$ and $\tar(\wt{\wlk})=\beta$, and consider the corresponding homotopy string $\wlk:=\GBQ(\Wlk)$ of $(\quiver,I)$, see Lemma \ref{lem:walkGQ} and the proof of Proposition \ref{pro:walkBelt} (here we set $\GBQ(\triv_\alpha):=\emptyset$ for every trivial walk $\triv_\alpha$ of $G$, where $\emptyset$ denotes the \lq\lq empty'' homotopy string as in \cite{Bob}). We want to translate to our graphical model Bobi\'nski's description of Auslander-Reiten triangles~\cite{Bob} in terms of homotopy strings (cf.~\eqref{eq:BMcpx}),
\begin{equation}\label{eq:BobAR}
\xymatrix{ \stringCpx{m,\wlk}  \ar[r] & \stringCpx{m+m'(\wlk),{}_+\wlk} \oplus \stringCpx{m,\wlk_+} \ar[r] & \stringCpx{m+m''(\wlk),{}_+\wlk_+}  \ar[r] & \stringCpx{m-1,\wlk}, }
\end{equation}
for which our goal is to show that $\GBQ({}_+\Wlk)={}_+\wlk$ and $m(\Wlk)=m'(\wlk)$ where $(m,\wlk) \mapsto (m+m'(\wlk),{}_+\wlk)$ is the main construction of~\cite{Bob}. We briefly recall such construction following the terminology and notation of that paper.

The first step is to consider the permitted path $\sigma_\wlk$ of maximal length such that $\sigma_\wlk \wlk$ is a homotopy presentation. We claim that $\ell(\sigma_\wlk)>0$ if and only if the (plain) concatenation $\TO{\beta}\Wlk$ is a reduced walk, in which case $\sigma_\wlk=\ppath{j'}{i_1}$ where $j'$ and $i_1$ are the last and first ordered edges of $\TO{\beta}$ and $\Wlk$ respectively. Indeed, recall that $\tar(\Wlk)=\tar(i_1)=\beta$ means that $i_1=(h^\beta_t,\iota h^\beta_t)$ for some $0 \leq t \leq \ell^\beta$, where
\[
\s^{-1}(\beta)=\{h^\beta_0>h^\beta_1>\ldots >h^\beta_{t-1}>h^\beta_t>\ldots >h^\beta_{\ell^\beta-1}>h^\beta_{\ell^\beta}\},
\]
see Remark~\ref{rem:linear}, and that by definition of $\TO{\beta}$ we have $j'=(\iota h^\beta_0,h^\beta_0)$.  Note that $\sigma_\wlk=\beta_1\cdots \beta_t=(h^\beta_0,h^\beta_1)\cdots (h^\beta_{t-1},h^\beta_t)$ which shows that $\ell(\sigma_\wlk)=t$, and our claim follows since $t=0$ if and only if $j'=\invw{i_1}$.

\begin{itemize}
 \item \emph{Assume that $\ell(\sigma_\wlk)>0$.} In this case ${}_+\wlk:=\invw{\theta}\sigma \wlk$ where $\sigma=\sigma_\wlk$ and $\theta=\theta_{\tar(\sigma),-T(\sigma)}$ is the forbidden path of maximal length with $\tar(\theta)=\tar(\sigma)$ and $T(\theta)=-T(\sigma)$ (cf.~\cite[Lemma 6.1]{Bob}). It is easy to see that by the maximality of $\sigma$ the path $\theta$ is a forbidden thread, and actually $\theta=\GBQ(\OT{\beta})$. Since $\sigma=\ppath{j'}{i_1}$ we have
 \begin{equation}\label{eq:proofOne}
\GBQ({}_+\Wlk)=\GBQ(\TO{\beta}\cdot \Wlk)=\GBQ(\invw{(\OT{\beta})}\Wlk)=\invw{\GBQ(\OT{\beta})}\ppath{j'}{i_1}\GBQ(\Wlk)= \invw{\theta}\sigma\wlk={}_+\wlk.
 \end{equation}
Moreover, $m'(\wlk):=\ell(\theta)-1=\ell(\TO{\beta})-2=m(\Wlk)$, see~(\ref{eq:lengthAntiWalk}).

 \item \emph{Assume that $\ell(\sigma_\wlk)=0$.} By the claim above, this means that there is a reduced walk $\wt{\theta}$ of maximal (positive) length such that
 \[
   \OT{\beta}=\wt{\theta}\wt{\theta}' \qquad \text{and} \qquad \Wlk=\wt{\theta}\Wlk',
 \]
for some reduced walks $\wt{\theta}'$ and $\Wlk'$ of $G$, so that $\TO{\beta}\cdot \Wlk=\invw{(\wt{\theta}')}\Wlk'$. Similarly, Bobi\'nski takes decompositions
\[
\GBQ(\OT{\beta})=\theta \theta' \qquad \text{and} \qquad \wlk=\theta \wlk',
\]
for a forbidden path $\theta$ of maximal length $r(\wlk) \geq 0$, and some walks $\theta'$ and $\wlk'$ of $\quiver$ (where $\theta'=\theta_{\tar \wlk',-T\wlk'}$ in the notation of~\cite{Bob}). The definition of ${}_+\wlk$ together with the remaining verification $\GBQ(\TO{\beta}\cdot \Wlk)=\GBQ(\invw{(\wt{\theta}')}\Wlk')={}_+\wlk$, which include several border cases, are shown in Table~\ref{tab:proofAR}. In this case, the integer $m'(\wlk)$ is given by $m'(\wlk):=\ell(\theta')+r(\wlk)-1=\ell(\theta\theta')-1=\ell(\TO{\beta})-2=m(\Wlk)$.
\end{itemize}
In any case, the homotopy string $\wlk_+$ is defined in~\cite{Bob} by $\wlk_+:=\invw{({}_+(\invw{\wlk}))}$, then using Lemma~\ref{lem:walkGQ}(c) we have
\[
\GBQ(\Wlk_+)=\GBQ(\Wlk\cdot \OT{\alpha})=\invw{\GBQ(\TO{\alpha}\cdot \invw{\Wlk})}=\invw{\GBQ(_+(\invw{\Wlk}))}=\invw{({}_+(\invw{\wlk}))}=\wlk_+,
\]
since clearly $\invw{(\Wlk_1 \cdot \Wlk_2)}=\invw{\Wlk_2}\cdot \invw{\Wlk_1}$ for any pair of reduced walks $\Wlk_1,\Wlk_2$ of $G$. The homotopy string ${}_+\wlk_+$ is given by ${}_+\wlk_+:=({}_+\wlk)_+$ if ${}_+\wlk\neq \emptyset$, and ${}_+\wlk_+:={}_+(\wlk_+)$ if $\wlk_+ \neq \emptyset$, and the integer $m''(\wlk)$ is given by $m''(\wlk):=m'(\wlk)$ and $m''(\wlk):=m'(\wlk_+)$ respectively. By Lemma~\ref{lem:associative} we have, in any case, $\GBQ(_+\Wlk_+)={}_+\wlk_+$, and clearly $m''(\wlk)=m(\Wlk)$. Using Proposition~\ref{pro:walkBelt}, this shows that the complexes of the triangle \eqref{eq:main64} in the main claim  coincide with the complexes of the triangle~\eqref{eq:BobAR}, which completes the proof.
\end{proof}

\begin{table}[!hbt]
\[\begin{array}{l l l l l l}
& \text{Conditions} & \text{Definition of ${}_+\wlk$} & \text{Walks of $G$} & \text{H. strings of $(\quiver,I)$} & \text{Example} \\[1ex]
& \ell(\theta')>0 & {}_+\wlk:=\invw{(\theta')}\wlk' & \OT{\beta}=\wt{\theta}\wt{\theta}', \quad \text{\footnotesize (with $\ell(\wt{\theta}') \geq 1$)} & \GBQ(\wt{\theta}')=\theta' & \Wlk=\invw{2}35 \\
& \ell(\wlk')>0 &  & \Wlk=\wt{\theta}\Wlk', \quad \text{\footnotesize (with $\ell(\Wlk') \geq 1$)} & \GBQ(\Wlk')={}^{[1]}\wlk' & \wlk=\alpha_1(\alpha_2\invw{\gamma}_1)  \\
& \text{$\sigma_1(\wlk')$ direct} & & {}_+\Wlk =\invw{(\wt{\theta}')}\Wlk' & \ppath{j}{i}=\sigma_1(\wlk')  & {}_{+}\Wlk=4\invw{1}35 \\
\bulito & \multicolumn{4}{l}{\text{Then $\GBQ({}_+\Wlk)=\invw{\GBQ(\wt{\theta}')}\ppath{j}{i}\GBQ(\Wlk')=\invw{(\theta')}\sigma_1(\wlk')({}^{[1]}\wlk')=\invw{(\theta')}\wlk'= {}_+\wlk$;}} & {}_{+}\wlk=\invw{\beta}_2(\alpha_2\invw{\gamma}_1). \\[3ex]
& \ell(\theta')>0 & {}_+\wlk:=\invw{(\theta')}\wlk' & \OT{\beta}=\wt{\theta}\wt{\theta}', \quad \text{\footnotesize (with $\ell(\wt{\theta}') > 1$)} &\GBQ(\wt{\theta}')={}^{[1]}\theta' & \Wlk=4\invw{1}3 \\
& \ell(\wlk')>0 &  & \Wlk=\wt{\theta}\Wlk', \quad \text{\footnotesize (with $\ell(\Wlk') \geq 1$)} & \GBQ(\Wlk')={}^{[1]}\wlk' & \wlk=(\invw{\beta}_2\alpha_2)  \\
& \text{$\sigma_1(\wlk')$ inverse} & & {}_+\Wlk =\invw{(\wt{\theta}')}\Wlk' & \ppath{j}{i}=\invw{\sigma_1(\theta')}\sigma_1(\wlk') & {}_+\Wlk=35\invw{1}3  \\
\bulito & \multicolumn{4}{l}{\text{Then $\GBQ({}_+\Wlk)=\invw{\GBQ(\wt{\theta}')}\ppath{j}{i}\GBQ(\Wlk')=\invw{(^{[1]}\theta')}\invw{\sigma_1(\theta')}\sigma_1(\wlk')({}^{[1]}\wlk')=\invw{(\theta')}\wlk'= {}_+\wlk$;}} & {}_+\wlk=\invw{\gamma}_1\invw{\beta}_3(\invw{\beta}_2\alpha_2). \\[3ex]
& \ell(\theta')>0 & {}_+\wlk:=\invw{(^{[1]}\theta')} & \OT{\beta}=\wt{\theta}\wt{\theta}', \quad \text{\footnotesize (with $\ell(\wt{\theta}') \geq 1$)} & \GBQ(\wt{\theta}')={}^{[1]}\theta' & \Wlk=\invw{2}1 \\
& \ell(\wlk')=0 &  & \Wlk=\wt{\theta}, \quad \text{\footnotesize (with $\ell(\Wlk') =0$)} & \GBQ(\Wlk')=\emptyset & \wlk=\alpha_1  \\
& \text{---------------} & & {}_+\Wlk =\invw{(\wt{\theta}')} & \text{---------------} & {}_+\Wlk=4 \\
\bulito & \multicolumn{4}{l}{\text{Then $\GBQ({}_+\Wlk)=\invw{\GBQ(\wt{\theta}')}=\invw{(^{[1]}\theta')}= {}_+\wlk$;}} & {}_+\wlk=\triv_{\vtx{4},-1}. \\[3ex]
& \ell(\theta')=0 & {}_+\wlk:=\wlk' & \OT{\beta}=\wt{\theta}\invw{j}, \quad \text{\footnotesize (with $\ell(\wt{\theta}') =1$)} & \GBQ(\invw{j})=\triv_{\vtx{j}} & \Wlk=\invw{2}1\invw{5}  \\
& \ell(\wlk')>0 & & \Wlk=\wt{\theta}\Wlk', \quad \text{\footnotesize (with $\ell(\Wlk') \geq 1$)} & \GBQ(\Wlk')={}^{[1]}\wlk' & \wlk=\alpha_1\beta_2(\beta_3)  \\
& \text{$\sigma_1(\wlk')$ direct} & & {}_+\Wlk =j\Wlk' & \ppath{j}{i}=\sigma_1(\wlk') & {}_+\Wlk=4\invw{5} \\
\bulito & \multicolumn{4}{l}{\text{Then $\GBQ({}_+\Wlk)=\GBQ(j\Wlk')=\ppath{j}{i}\GBQ(\Wlk')=\sigma_1(\wlk')({}^{[1]}\wlk')=\wlk'= {}_+\wlk$;}} & {}_+\wlk=(\beta_3).  \\[3ex]
& \ell(\theta')=0 & {}_+\wlk:={}^{[1]}\wlk' & \OT{\beta}=\wt{\theta}, \quad \text{\footnotesize (with $\ell(\wt{\theta}') =0$)} & \GBQ(\wt{\theta}')=\emptyset & \Wlk=4\invw{5}\invw{3}1\invw{5} \\
& \ell(\wlk')>0 & & \Wlk=\wt{\theta}\Wlk', \quad \text{\footnotesize (with $\ell(\Wlk') \geq 1$)} & \GBQ(\Wlk')={}^{[1]}\wlk' & \wlk=\beta_3\gamma_1(\invw{\alpha}_2\beta_2\beta_3) \\
& \text{$\sigma_1(\wlk')$ inverse} & & {}_+\Wlk =\Wlk' & \text{---------------}  & {}_+\Wlk=1\invw{5}  \\
\bulito & \multicolumn{4}{l}{\text{Then $\GBQ({}_+\Wlk)=\GBQ(\Wlk')={}^{[1]}\wlk'= {}_+\wlk$;}} & {}_+\wlk=\beta_2\beta_3.  \\[3ex]
& \ell(\theta')=0 & {}_+\wlk:=\emptyset & \OT{\beta}=\wt{\theta}, \quad \text{\footnotesize (with $\ell(\wt{\theta}') =0$)} &\GBQ(\wt{\theta}')=\emptyset & \Wlk=\invw{2}1\invw{4} \\
& \ell(\wlk')=0 & & \Wlk=\wt{\theta},\quad \text{\footnotesize (with $\ell(\Wlk') =0$)} & \GBQ(\Wlk')=\emptyset & \wlk=\alpha_1\beta_2 \\
& \text{---------------} & & {}_+\Wlk =\triv_\alpha,\quad \text{\footnotesize (with $\alpha:=\sou(\Wlk)$)} & \text{---------------} & {}_+\Wlk=\triv_\delta \\
\bulito & \multicolumn{4}{l}{\text{Then $\GBQ({}_+\Wlk)=\GBQ(\triv_\alpha)=\emptyset= {}_+\wlk$;}}
  & {}_+\wlk=\emptyset.
\end{array}\]
\caption{Verification of the identity $\GBQ({}_+\Wlk)={}_+\GBQ(\Wlk)$ for the case $\ell(\sigma_{\GBQ(\Wlk)})=0$. We use notation of \cite{Bob}. When possible, $j$ and $i$ denote the last and first ordered edge of $\invw{(\wt{\theta}')}$ and $\Wlk'$ respectively. The given example refers to the gentle bound quiver $(\quiver^{(1)},I^{(1)})$ and marked ribbon graph $(G^{(1)},\leq^{(1)})$ of Example~\ref{ex:bigARAmiot}.}
\label{tab:proofAR}
\end{table}

In the following observation, whose proof follows from Proposition~\ref{pro:walkBelt}(b) and~\cite[Main Theorem Part~I]{Bob}, we describe those almost split triangles of perfect complexes over gentle algebras which are not included in Proposition~\ref{thm:ARtriangle}.

\begin{remark}\label{rem:beltARtriangle}
 Let $(Q,I)$ be a gentle bound quiver with associated marked ribbon graph $(G,\leq)$. Then for any belt $\wlk$ of $(G,\leq)$, any $m \in \Z$ and any almost split sequence $\xymatrix{0 \ar[r] & \mu' \ar[r] & \bigoplus_{t=1}^n \mu_t \ar[r] & \mu'' \ar[r] & 0}$ in the category of automorphisms of finite dimensional vector spaces where $\mu',\mu_1,\ldots,\mu_n,\mu''$ are indecomposable automorphisms, there is an almost split triangle in $\per\,kQ/I$ of the form
\[
\xymatrix{ \stringcpx{m,\wlk,\mu'}  \ar[r] & \bigoplus_{t=1}^n \stringcpx{m,\wlk,\mu_t} \ar[r] & \stringcpx{m,\wlk,\mu''}  \ar[r] & \stringcpx{m-1,\wlk,\mu'}. }
\]
\end{remark}

\begin{example}\label{ex:bigARAmiot}
Consider the gentle bound quiver $(Q^{(1)},I^{(1)})$ of Example~\ref{exa:amiot}(b) with associated ribbon graph $(G^{(1)},\leq^{(1)})$ and its anti-walks (to simplify the notation we draw the edges of $G^{(1)}$ together with a fixed order, cf.~\eqref{eq:ordedg}):
\[
\xy 0;/r.20pc/:
(-90, 25)="Fr1" *{};
( 90, 25)="Fr2" *{};
( 90,-5)="Fr3" *{};
(-90,-5)="Fr4" *{};
( -25, 10)="T1" *{\xymatrix@C=1pc@R=.8pc{ {\vBlock{\gamma}{53}} \ar@{|->}[rr]^-{3} \ar@{<-|}[dd]_-{5} && {\vBlock{\alpha}{213}} \\ \\
*++{\vBlock{\beta}{2145}} \ar@{|->}[rr]_-{4} \ar@{|->}@<.5ex>[rruu]^-{1} \ar@{|->}@<-.5ex>[rruu]_-{2} && {\vBlock{\delta}{4}} }};
( 20, 9)="T2" *{\begin{tabular}{r c l}
$\OT{\alpha}$    & = &  $2\invw{1}3$ \\
$\OT{\beta}$     & = &  $\invw{2}1\invw{4}$ \\
$\OT{\gamma}$ & = & $5$ \\
$\OT{\delta}$     & = &  $4\invw{5}\invw{3}$ \\
\end{tabular}};
\endxy
\]
Take the reduced walk $\wlk=\invw{1}\,3\,5$ with $\tar(\wlk)=\beta=\sou(\wlk)$ and the reduced concatenations constructed in Proposition~\ref{thm:ARtriangle}:
\[
\xy 0;/r.20pc/:
(-90, 25)="Fr1" *{};
( 90, 25)="Fr2" *{};
( 90,-5)="Fr3" *{};
(-90,-5)="Fr4" *{};
(5, 10)="T1" *{\xymatrix@!0@C=70pt@R=25pt{
 &\stringcpx{m+1, (4\,\invw{1}\,2\,\invw{1}\,3\,5)}\ar@{->}[rd]& \\
\stringcpx{m, (\invw{1}\,3\,5)} \ar@{->}[ru]\ar@{->}[rd]& & \ar@{-->}[ll]_{\tau}
\stringcpx{m+1, (4\,\invw{1}\,2\,\invw{1}\,3\,5\,\invw{2}\,1\,\invw{4})}\\
 &\stringcpx{m, (\invw{1}\,3\,5\,\invw{2}\,1\,\invw{4})}\ar@{->}[ru]& }};
(-50, 10)="A1" *{\begin{tabular}{r c l}
${}_+\wlk$    & = &  $\TO{\beta}\cdot \wlk =4\,\invw{1}\,2\,\invw{1}\,3\,5$ \\
$\wlk_+$     & = &  $\wlk \cdot \OT{\beta}=\invw{1}\,3\,5\,\invw{2}\,1\,\invw{4}$ \\
${}_+\wlk_+$ & = & $\TO{\beta}\cdot \wlk \cdot \OT{\beta}=4\,\invw{1}\,2\,\invw{1}\,3\,5\,\invw{2}\,1\,\invw{4}$ \\
$m(\wlk)$     & = &  $\ell(\TO{\beta})-2=1$ \\
\end{tabular}};
\endxy
\]
Then the corresponding Auslander-Reiten triangle of perfect complexes of the gentle algebra $A^{(1)}=kQ^{(1)}/I^{(1)}$ has the following shape
\[
\xy 0;/r.20pc/:
(-90, 25)="Fr1" *{};
( 90, 25)="Fr2" *{};
( 90,-45)="Fr3" *{};
(-90,-45)="Fr4" *{};
( -45, -3.5)="T1" *{\xymatrix@!0@C=40pt@R=10pt{
P_1 \\
\oplus \ar[rr]^-{\left[ \begin{smallmatrix} \alpha_2 & \gamma_1\end{smallmatrix} \right]} && P_3 \\
\underline{P_5}  }};
( -27, 10)="T2" *{\xymatrix@!0@C=40pt@R=10pt{
&& P_1 \\
&& \oplus && P_4 \\
 P_2 \ar[rr]^-{\left[ \begin{smallmatrix} \alpha_1 \\ \beta_1\\ 0\end{smallmatrix} \right]} && P_1 \ar[rr]^-{\left[ \begin{smallmatrix} \beta_2 &0&0 \\ 0&\alpha_2&\gamma_1 \end{smallmatrix} \right]} && \oplus \\
&& \oplus && P_3 \\
&& \underline{P_5}  }};
( -27,-13)="T3" *{\xymatrix@!0@C=40pt@R=10pt{
&& P_1 \\
&& \oplus && P_3 \\
 P_2 \ar[rr]_-{\left[ \begin{smallmatrix} 0\\\beta_1\beta_2\beta_3\\ \alpha_1\end{smallmatrix} \right]}  && P_5 \ar[rr]_-{\left[ \begin{smallmatrix} \alpha_2&\gamma_1&0 \\ 0&0&\beta_2 \end{smallmatrix} \right]} && \oplus \\
&& \oplus && P_4 \\
&& \underline{P_1}   }};
(10, 0)="T4" *{\xymatrix@!0@C=40pt@R=10pt{
 && P_1 \\
 && \oplus && P_4\\
 P_2 && P_1 && \oplus \\
 \oplus \ar[rr]^-{\left[ \begin{smallmatrix} \alpha_1&0 \\ \beta_1&0 \\ 0&\beta_1\beta_2\beta_3 \\ 0& \alpha_1  \end{smallmatrix} \right]} && \oplus \ar[rr]^-{\left[ \begin{smallmatrix} \beta_2&0&0&0 \\ 0&\alpha_2&\gamma_1&0 \\ 0&0&0&\beta_2 \end{smallmatrix} \right]} && P_3 \\
 P_2 && P_5 && \oplus \\
 && \oplus && P_4 \\
 && \underline{P_1} }};
(-29, 3)="A1" *{\xymatrix{&& {} \\ {} \ar[rru] }};
(-5, 3)="A2" *{\xymatrix{{} \ar[rrd] \\ && {} }};
(-29, -8)="A3" *{\xymatrix{{} \ar[rrd] \\ && {} }};
(-5, -8)="A4" *{\xymatrix{&& {} \\ {} \ar[rru] }};
(-26, -5.7)="A" *{\xymatrix{{} \ar@{<--}[rrrrr]^-{\tau} &&&&& {}  }};
\endxy
\]
\smallskip
\noindent
where underline indicates cohomological degree $m$ (cf.~Subsection~\ref{subsec:perfcpx}). Observe that here $q([\stringcpx{m,\wlk}])=q([\stringcpx{m+1,{}_+\wlk_+}])=2$ and $q([\stringcpx{m+1,{}_+\wlk}])=q([\stringcpx{m,\wlk_+}])=1$, where $q=q_{A^{(1)}}$ is the Euler form of $A^{(1)}$, compare with Theorem~\ref{thm:rootsind}.
\end{example}
In the following example we compute a complete Auslander-Reiten component with the use of Proposition~\ref{thm:ARtriangle}.

\begin{example}\label{ex:oneloopAR} 
Take $A=kQ/I$ for $(Q,I)$ being the one loop gentle bound quiver {\xymatrix@C=-1pt{{}_1 \ar@(dl,ul)^-{\alpha_1}="a" & {}
\ar@{.}@/^3pt/"a"!<-6pt,-5pt>;"a"!<-6pt,5pt> }}. Recall that then $\GPC(\quiver,I)=(G,\leq)$ with $G={\xymatrix{  {{\alpha}} \ar@{-}@(ld,lu)^(.5){1}}}$ and every (non-zero) indecomposable perfect complex is the string complex $\stringcpx{m,1^\ell}$  for some $m\in\ZZ$ and $\ell\geq 1$, see Example~\ref{ex:extremal}. Then $\OT{\alpha}=\invw{1}$, $\TO{\alpha}=1$, thus
${}_+1^\ell = 1^{\ell+1}$, $1^\ell_+ = 1^\ell\cdot \invw{1}=1^{\ell-1}$, ${}_+1^\ell_+ = 1\cdot 1^\ell\cdot \invw{1}=1^\ell$, $m(1^\ell) = \ell(\TO{\alpha})-2=-1$. Therefore, every string complex $\stringcpx{m,1^\ell}$ induces Auslander-Reiten triangle of the form
\begin{equation}\label{eq:exoneloopAR}
\xymatrix{ \stringcpx{m,1^\ell}  \ar[r] & \stringcpx{m-1,1^{\ell+1}} \oplus \stringcpx{m,1^{\ell-1}} \ar[r] & \stringcpx{m-1,1^\ell}  \ar[r] & \stringcpx{m-1,1^\ell}. }
\end{equation}
 In particular, for $\ell=1$ the walk $1^{\ell-1}$ is the trivial walk $\triv_\alpha$ so the triangle~\eqref{eq:exoneloopAR} has the form $\stringcpx{m,1} \to \stringcpx{m-1,11} \to \stringcpx{m-1,1}\to  \stringcpx{m-1,1}$. This shows that the Auslander-Reiten quiver of the category $\per\,\La$ consists of one connected component of the shape $\ZZ\AA_\infty$ having the complexes $\{\stringcpx{m,1}\}_{m\in\ZZ}$ on the boundary,  and the Auslander-Reiten translation on $\per\,\La$ is given by $\tau_\La X^\bullet=X^\bullet[-1]$, compare with~\cite[Theorem B]{BGSdisc}.
\end{example}

\subsection{The Coxeter transformation of a gentle algebra}\label{subsec:Cox}

\def\Janti{\widehat{J}}

If $\La$ is a Gorenstein algebra we have at our disposal a triangle autoequivalence  $\tau_\La:\per\,\La\to\per\,\La$, called the \textbf{Auslander-Reiten translation} of $\per\,\La$, such that for any indecomposable perfect complex $Z^\bullet$ the Auslander-Reiten triangle in $\per\,\La$ ending at $Z^\bullet$ has the form $X^\bullet \to Y^\bullet \to Z^\bullet \to X^\bullet[1]$ with $X^\bullet \cong \tau_\La Z^\bullet$, see for instance~\cite{GeiRei05} and~\cite{LadkaniICRA}, cf.~introduction to Section \ref{sec:AR}. The induced automorphism $\Psi_\La:K_0(\per\,\La)\to K_0(\per\,\La)$ given by $\Psi_\La([Z^\bullet]):=[\tau_\La Z^\bullet]$ is called the \textbf{Coxeter transformation} of $\La$, cf.~\cite{HDb2} and~\cite{LadkaniICRA}, and its characteristic polynomial $\psi_\La(z)\in\ZZ[z]$ is called the \textbf{Coxeter polynomial}, which is a derived equivalence invariant of $A$, see~\cite{LadkaniICRA}. Moreover, if $\La$ has finite global dimension, then the categories $\Db{\La}$ and $\per\,\La$ coincide and therefore $\Db{\La}$ has Auslander-Reiten triangles. In this case, the Coxeter transformation and the Coxeter polynomial of $\La$ were widely studied, for instance in~\cite{LPtrends, Happel.Trace,RCox,PMahler}.

For a gentle bound quiver $(Q,I)$ with marked ribbon graph $(G,\leq)$, we collect in a $E(G)\times V(G)$-matrix $\Janti$ the incidence vectors of the anti-walks of $(G,\leq)$,
\begin{equation}\label{eq:Janti}
\Janti:=[\inc(\OT{\alpha})]_{\alpha \in \verts{G}}, \quad \text{where $\inc(i_1\cdots i_Li_{L+1})=\sum_{t=0}^{L}(-1)^t\bas_{i_{t+1}}$ as in~\eqref{eq:defincvect}},
\end{equation}
and use this matrix to give an expression for the Coxeter transformation $\Psi_\La$ of the gentle algebra $\La=kQ/I$ as a matrix in the basis of projectives, and to compute the Coxeter polynomial $\psi_\La(z)$ in terms of the AAG invariant of $(Q,I)$, see~(\ref{eq:AAGdef}). Observe that if $(Q,I)$ has finite global dimension, then the matrix $\Janti$ coincides with the incidence matrix $\Inc(\fmrgs{G},\fmrgs{\sigma})$ of the forbidden marked ribbon graph $\GFC(\quiver,I)=(\fmrgs{G},\fmrgs{\leq},\fmrgs{\sigma})$ of $(Q,I)$, see Definition~\ref{rem:forb}, up to a permutation and appropriate change of sign of its columns. Indeed, for an anti-walk $\OT{\alpha}$ of $(G,\leq)$ with source vertex $\beta:=\sou(\OT{\alpha})$, via the bijection $\Phi_2=\sou \circ \GBQ^{-1}$ of~(\ref{eq:bijections}), the $\beta$-th column of $\Inc(\fmrgs{G},\fmrgs{\sigma})$ is given by $\Inc(\fmrgs{G},\fmrgs{\sigma})\bas_\beta=\Inc(\fmrgs{G},\fmrgs{\sigma})\bas_{\GBQ(\OT{\alpha})}=(-1)^{\ell(\OT{\alpha})+1}\inc(\OT{\alpha})$, cf.~(\ref{eq:incantistring}). Since the $\alpha$-th column of $\Janti$ is $\Janti\bas_\alpha=\inc(\OT{\alpha})$, and $\beta=\VerPer(\alpha)$ where $\VerPer$ is the permutation~(\ref{Eq:VerPer}), then $\Inc(\fmrgs{G},\fmrgs{\sigma})\bas_{\VerPer(\alpha)}(-1)^{\ell(\OT{\alpha})}=-\inc(\OT{\alpha})=-\Janti\bas_\alpha$. This shows that
\begin{equation}\label{eq:relationIJ}
\Inc(\fmrgs{G},\fmrgs{\sigma})\Lambda=-\Janti, \qquad \text{with the orthogonal matrix $\Lambda:=P(\VerPer)S$,}
\end{equation}
where $P(\VerPer)$ is the permutation matrix given by $P(\VerPer)\bas_\alpha=\bas_{\VerPer(\alpha)}$, and $S$ is the diagonal matrix $\diag[(-1)^{\ell(\OT{\alpha})}]$ with $\alpha$ running over the set of vertices of $G$. We will use matrix $\Lambda$ in the proof of Theorem~\ref{cor:CoxPol} below.

By applying arguments analogous to those for Claim \ref{claimone} we obtain its dual variant as follows.

\begin{claim}\label{claimtwo} {If $\OT{\alpha}=i_1\ldots i_Li_{L+1}$ is the anti-walk of $(G,\leq)$ having as target the vertex $\tar(\OT{\alpha})=\alpha$, then the injective resolution of the string module $M_{\alpha}$ induces the following exact sequence of $\La$-modules
\begin{equation*}\label{eq:resolutionE}
\xymatrix{0 \ar[r]& M_{\alpha} \ar[r] & E_{i_{1}} \ar[r] & E_{i_{2}} \ar[r] & \cdots \ar[r] & E_{i_L} \ar[r] & E_{i_{L+1}}\ar[r]&0.}
\end{equation*}}
\end{claim}
In particular, we have (cf.~Remark \ref{rem:euler} and \eqref{eq:IncString})
\begin{equation}\label{eq:injInc}
C_A^{\tr}\Janti=\Inc(G).
\end{equation} 

\begin{corollary}\label{cor:CoxTra}
The Coxeter transformation $\Psi_\La:K_0(\per\,\La) \to K_0(\per\,\La)$ associated to the perfect category of a gentle bound quiver $(Q,I)$ in the basis of projectives of $K_0(\per\,\La)$ (with $\La:=kQ/I$) is given by the $n \times n$-matrix
\[
\Psi_A=\Id_n-\Janti\Janti^{\tr}C_A^{\tr},
\]
where $n=|Q_0|$ and $\Id_n, C_\La\in\M_{n,n}(\Z)$ are the identity matrix and the Cartan matrix of $\La$, respectively.
\end{corollary}

Observe that if $\La$ has finite global dimension, then the claim of Corollary~\ref{cor:CoxTra} follows easily from Theorem~\ref{thm:mrgvsCar} and the observation~(\ref{eq:relationIJ}). Indeed, recall that in this case the Coxeter transformation $K_0(\mod{\La}) \to K_0(\mod{\La})$ of $\La$, as a matrix in the basis determined by simple modules, has the form $-\Car_\La^{\tr}\Car_\La^{-1}$ (cf.~\cite{LPtrends, ASS}). Using Remark~\ref{rem:euler}(b) we have the following commutative diagram,
\[
\xymatrix{K_0(\proj\,\La) \ar[d]_-{-\Car_\La^{-1}\Car_\La^\tr} \ar[r]^-{\Car_\La} & K_0(\mod{\La}) \ar[d]^-{-\Car_\La^\tr\Car_\La^{-1}} \\
K_0(\proj\,\La) \ar[r]_-{\Car_\La} & K_0(\mod{\La}),}
\]
where we consider the basis $\{[P_i]\}_i$ of the group $K_0(\proj\,\La)$. Recalling that the groups $K_0(\proj\,\La)$ and $K_0(\per\,\La)$ can be identified through the corresponding bases determined by indecomposable projective modules, the Coxeter transformation on $K_0(\per\,\La)$ with respect to such basis is given by $-\Car_\La^{-1}\Car_\La^\tr$. Using Theorem~\ref{thm:mrgvsCar} and equation~(\ref{eq:relationIJ}), we have $\Janti\Janti^\tr=\Inc(\fmrgs{G},\fmrgs{\sigma})\Lambda\Lambda^\tr\Inc(\fmrgs{G},\fmrgs{\sigma})^\tr=\Inc(\fmrgs{G},\fmrgs{\sigma})\Inc(\fmrgs{G},\fmrgs{\sigma})^\tr=\Car_\La^{-1}+\Car_\La^{-\tr}$, that is,
\[
 \Id_n-\Janti\Janti^\tr \Car_\La^\tr=\Id_n-(\Car_\La^{-1}+\Car_\La^{-\tr})\Car_\La^\tr=-\Car_\La^{-1}\Car_\La^\tr = \Psi_\La.
\]
The following alternative proof includes the case of infinite global dimension.

\begin{proof}
Note first that the matrix $\Psi_A$ is $\Z$ invertible with the inverse given by $\Psi_A^{-1}=\Id_n-\Janti\Janti^{\tr}C_A$. Indeed, we have
\begin{align*}
\Psi_A\Psi_A^{-1} & = (\Id_n-\Janti\Janti^{\tr}C_A^{\tr})(\Id_n-\Janti\Janti^{\tr}C_A) \\
& = \Id_n-\Janti\Janti^{\tr}C_A-\Janti\Janti^{\tr}C_A^{\tr}+\Janti\Janti^{\tr}C_A^{\tr}\Janti\Janti^{\tr}C_A \\
& = \Id_n-\Janti\Janti^{\tr}(C_A+C_A^{\tr})+\Janti\Janti^{\tr}\Inc(G)\Inc(G)^{\tr} && \text{by~(\ref{eq:injInc})} \\
& = \Id_n-\Janti\Janti^{\tr}(C_A+C_A^{\tr})+\Janti\Janti^{\tr}(C_A+C_A^{\tr})=\Id_n, && \text{by Theorem~\ref{thm:mrgvsCar}.}
\end{align*}

Assume now that $X^\bullet$ is an indecomposable perfect complex of the gentle algebra $A$ with Auslander-Reiten translation $\tau_A X^\bullet$, which is again a perfect complex since $\La$ is Gorenstein. Moreover, if $X^\bullet$ is a string complex, then so is $\tau_A X^\bullet$ by Proposition~\ref{thm:ARtriangle}. In this case, by Proposition~\ref{pro:walkBelt} there are a reduced walk $\wlk$ of the associated marked ribbon graph $(G,\leq)$ and $m \in \Z$ such that $\tau_A X^\bullet$ is isomorphic to $\stringcpx{m,\wlk}$. Taking $\alpha=\tar(\wlk)$ and $\beta=\sou(\wlk)$ we have
\begin{align*}
\Psi_A^{-1}[\vdimp{\stringcpx{m,\wlk}}] & = (-1)^m[\Id_n-\Janti\Inc(G)^{\tr}]\inc(\wlk) && \text{by Lemma~\ref{lem:incvectsvsdim}(a) and~(\ref{eq:injInc})} \\
& = (-1)^m[\inc(\wlk)-\Janti(\bas_\alpha-(-1)^{\ell(\wlk)}\bas_\beta)] && \text{by Lemma~\ref{lem:incvects}(e)} \\
& = (-1)^m[\inc(\wlk)-\inc(\OT{\alpha})+(-1)^{\ell(\wlk)}\inc(\OT{\beta})] && \text{by definition of $\Janti$.}
\end{align*}
On the other hand, using Lemma~\ref{lem:incvects}(a,b) it is easy to show that $\inc(\wlk \cdot \wlk')=\inc(\wlk \wlk')$ for any pair of reduced concatenable walks $\wlk$ and $\wlk'$. Then we have
\begin{align*}
\inc({}_+\wlk_+) & = \inc(\TO{\alpha}\cdot \wlk \cdot \OT{\beta}) = \inc(\TO{\alpha} \wlk \OT{\beta}) = \inc(\TO{\alpha})+(-1)^{\ell(\TO{\alpha})}\inc(\wlk \OT{\beta}) && \text{by Lemma~\ref{lem:incvects}(a)} \\
& = (-1)^{\ell(\TO{\alpha})+1}\inc(\OT{\alpha})+(-1)^{\ell(\TO{\alpha})}(\inc(\wlk)+(-1)^{\ell(\wlk)}\inc(\OT{\beta})) && \text{by Lemma~\ref{lem:incvects}(a,b)} \\
& = (-1)^{\ell(\TO{\alpha})}[\inc(\wlk)-\inc(\OT{\alpha})+(-1)^{\ell(\wlk)}\inc(\OT{\beta})].
\end{align*}
Using Lemma~\ref{lem:incvectsvsdim}(a), the definition $m(\wlk):=\ell(\TO{\alpha})-2$ and the above identities we get
\[
\vdimp{\stringcpx{m+m(\wlk),{}_+\wlk_+}}=(-1)^{m+m(\wlk)}\inc({}_+\wlk_+)=(-1)^m[\inc(\wlk)-\inc(\OT{\alpha})+(-1)^{\ell(\wlk)}\inc(\OT{\beta})]=\Psi_A^{-1}[\vdimp{\stringcpx{m,\wlk}}].
\]
By Proposition~\ref{thm:ARtriangle}, the complex $X^\bullet$ is isomorphic to $\stringcpx{m+m(\wlk),{}_+\wlk_+}$ which shows that $\Psi_A[\vdimp{X^\bullet}]=\vdimp{\tau_A X^\bullet}$. This completes the proof, since $K_0(\per\,\La)$ contains a basis consisting of classes of string complexes.
\end{proof}

Denote by $\chr[M](z):=\det(z\Id_n-M)$ the characteristic polynomial of a $n \times n$ matrix $M$ and by $\psi_A(z):=\chr[\Psi_A](z)$ the Coxeter polynomial of an algebra $A$ (when defined). We can now prove Theorem C from the introduction, which we restate here for convenience.
\begin{theorem}\label{cor:CoxPol}
Let $(Q, I)$ be a gentle bound quiver. The Coxeter polynomial $\psi_A(z)$ of the gentle algebra $A=kQ/I$ with AAG invariant $\varphi_A$ is given by
\begin{equation}\label{eq:coxpAAG}
\psi_A(z)=(z-1)^{|Q_1|-|Q_0|}\prod_{\substack{n>0 \\ m \geq 0}}(z^n-(-1)^{n+m})^{\varphi_A(n,m)}.
\end{equation}
\end{theorem} 
Observe that the theorem shows that AAG invariant (more precisely, its \lq\lq finite part'': $\varphi_A(n,m)$ for $n>0$) completely determines the shape of the Coxeter polynomial of a gentle algebra.
\begin{proof} Let  $(G,\leq)$ be the marked ribbon graph associated with $(Q,I)$. We set $V=\verts{G}$ and $E=\edgs{G}$.
Consider the orthogonal matrix $\Lambda:=P(\VerPer)S$ with  $\VerPer$ as in~(\ref{Eq:VerPer}) and $S:=\diag[(-1)^{\ell(\OT{\alpha})}]_{\alpha\in V}$, see~(\ref{eq:relationIJ}), and  observe that $\Lambda=\Id_{|V|}-\Janti^{\tr}C_A\Janti$. Indeed, using~(\ref{eq:injInc}) and Lemma~\ref{lem:incvects}(e), for any vertex $\alpha \in V$ we have
\[
(\Id_{|V|}-\Janti^{\tr}C_A\Janti)\bas_\alpha = \bas_\alpha-\Inc(G)^{\tr} \inc(\OT{\alpha}) = \bas_\alpha-[\bas_{\tar(\OT{\alpha})} - (-1)^{\ell(\OT{\alpha})}\bas_{\sou(\OT{\alpha})}]= (-1)^{\ell(\OT{\alpha})}\bas_{\VerPer(\alpha)} = P(\VerPer)S\bas_\alpha = \Lambda\bas_\alpha.
\]
Using that for a $a \times b$ matrix $M$ and a $b \times a$ matrix $N$ we have $\chr[MN](z)=z^{a-b}\chr[NM](z)$, see for instance~\cite[2.4]{fZhang99}, and that the characteristic polynomials of a square matrix and its transpose coincide, we get from Corollary~\ref{cor:CoxTra}
\begin{align*}
\psi_A(z) & = \det[z\Id_{|E|}-(\Id_{|E|} -\Janti\Janti^\tr C_A^\tr)] =\chr[-\Janti\Janti^\tr C_A^\tr](z-1) \\
& = \chr[-C_A\Janti\Janti^\tr](z-1) = (z-1)^{|E|-|V|}\chr[-\Janti^\tr C_A \Janti](z-1) \\
& = (z-1)^{|E|-|V|}\det[z\Id_{|V|}-\Id_{|V|} +\Janti^\tr C_A \Janti] = (z-1)^{|E|-|V|}\chr[\Lambda](z).
\end{align*}
Recall that the characteristic polynomial of an orthogonal matrix $\Lambda=P(\xi)S$ with integer coefficients, where $\xi$ is a permutation of $V$ and $S=\diag(s_\alpha)_{\alpha \in V}$ with $s_\alpha\in\{\pm 1\}$ for $\alpha\in V$, is given by
\[
\chr[\Lambda](z)=\prod_{\mathcal{O} \in V/\xi}[z^{|\mathcal{O}|}-\sgn(\mathcal{O})], \qquad \text{where $\sgn(\mathcal{O}):=\prod_{\alpha \in \mathcal{O}}s_\alpha$.}
\]
By definition of the AAG invariant $\varphi_A$~(\ref{eq:AAGdef}), and using the bijections $\VerPer=\sou \circ \tar^{-1}$ and $\Phi=\Phi_1 \Phi_2$ defined in Subsection~\ref{subsec:antiwalks}, see~(\ref{eq:bijections}), for fixed $n>0$ and $\epsilon \in \{0,1\}$ we have
\begin{eqnarray*}
\sum_{\substack{m \geq 0 \\ n+m \equiv \epsilon (\mod 2)}} \varphi_A(n,m) & = &
|\{ \mathcal{O} \in \forb_{\quiver,I}/\Phi \mid \text{$n=n_\mathcal{O}=|\mathcal{O}|$ and $n_\mathcal{O}+m_\mathcal{O}\equiv \epsilon (\mod 2)$} \}| \\
&=& |\{ \mathcal{O}' \in V/\VerPer \mid \text{$n=|\mathcal{O}'|$ and $\sgn(\mathcal{O}')=(-1)^\epsilon$} \}|.
\end{eqnarray*}
Note that for the last identity, if $\wlk=\OT{\alpha^0}\OT{\alpha^1}\cdots \OT{\alpha^a}$ is the face of $(G,\leq)$ corresponding to an orbit $\mathcal{O} \in \forb_{\quiver,I}/\Phi$ (correspondingly, to an orbit $\mathcal{O}' \in V/\VerPer$), then we have $n_\mathcal{O}+m_\mathcal{O}=\ell(\wlk)$ by~(\ref{eq:AAGone}), and
\[
(-1)^{\ell(\wlk)}=\prod_{t=0}^a(-1)^{\ell(\OT{\alpha^t})}=\prod_{\alpha \in \mathcal{O}'}(-1)^{\ell(\OT{\alpha})} =\sgn(\mathcal{O}'),
\]
since $\Lambda:=P(\VerPer)S$ with $S:=\diag[(-1)^{\ell(\OT{\alpha})}]$. This shows that $\chr[\Lambda](z)=\prod_{\substack{n>0 \\ m \geq 0}}[z^n-(-1)^{n+m}]^{\varphi_A(n,m)}$, and the claim follows since $|Q_1|-|Q_0|=|E|-|V|$ by~(\ref{eq:mrbvsCarcrk}).
\end{proof}

\begin{example}\label{ex:Amiotcoxp} 
\ Consider the three gentle algebras $A^{(i)}:=kQ^{(i)}/I^{(i)}$  from Examples~\ref{exa:amiot}(b) and~\ref{ex:Amiotder}.  Recall that all three algebras have the same AAG invariant $\{(4,6)\}$ (that is, $\varphi_{A^{(i)}}(4,6)=1$). Then by applying Theorem~\ref{cor:CoxPol} we find that the Coxeter polynomial in each case has the shape
$$\psi_{A^{(i)}}(z)=(z-1)^{6-5}(z^4-(-1)^{4+6})^1=z^5-z^4-z+1.$$
\end{example}

Recall that for an algebra $\La$ of finite global dimension the matrix formula $-\Car_\La^{\tr}\Car_\La^{-1}$ for the Coxeter transformation implies that the Coxeter polynomial is palindromic (self-reciprocal), see~\cite{LPtrends}. The final example shows that Theorem~\ref{cor:CoxPol} provides a method to search for non-palindromic Coxeter polynomials, which seem to be new in the literature.

\begin{example}\label{ex:nonpalin}
Let $(Q,I)$ be the gentle bound quiver
\xymatrix{\mathmiddlescript{1} \ar@/^7pt/[r]^-{\alpha_1}="a" & \mathmiddlescript{2} \ar@/^7pt/[l]^-{\beta_1}="b"
\ar@{.}@/^3pt/"a"!<7pt,-6pt>;"b"!<7pt,7pt>
\ar@{.}@/^7pt/"b"!<-3pt,6pt>;"a"!<-3pt,-4pt> } from Example~\ref{ex:small}. We verify directly that its AAG invariant is $\{(2,0), (0,2)\}$. By applying Theorem~\ref{cor:CoxPol} we get that the algebra $A=kQ/I$ has the following non-palindromic Coxeter polynomial
$$\psi_{A}(z)=(z-1)^0(z^2-(-1)^2)=z^2-1.$$
\end{example}

\let\OLDthebibliography\thebibliography
\renewcommand\thebibliography[1]{
  \OLDthebibliography{#1}
  \setlength{\parskip}{0pt}
  \setlength{\itemsep}{0pt plus 0.3ex}
}

\bibliographystyle{abbrv}
{\footnotesize
\bibliography{biblio}
}
\end{document}